\documentclass[fullpage]{article}

\usepackage[numbers]{natbib}
\usepackage{amsthm,amsmath,mathrsfs}
\usepackage{amssymb,amsthm,amsmath}
\usepackage{dsfont}
\usepackage[dvips]{geometry}
\usepackage{amscd}
\usepackage{amsfonts}
\usepackage[latin1]{inputenc}
\usepackage[all] {xy}
\usepackage[normalem]{ulem}
\usepackage{hyperref}
\usepackage{mathtools}
\usepackage[dvips]{geometry}
\usepackage{enumitem}
\usepackage{graphicx}
\usepackage{subcaption}
\usepackage{tikz-cd}
\usepackage[nottoc]{tocbibind}
\usepackage[titletoc,toc,title]{appendix}
\DeclareMathAlphabet{\mathpzc}{OT1}{pzc}{m}{it}

\newtheorem{theorem}{Theorem}[section]
\newtheorem{maintheorem}{Theorem}

\newtheorem{proposition}[theorem]{Proposition}
\newtheorem{corollary}[theorem]{Corollary}
\newtheorem{question}[theorem]{Question}
\newtheorem{lemma}[theorem]{Lemma}

\newtheorem{definition}[theorem]{Definition}

\newtheorem{remark}[theorem]{Remark}

\newtheorem{claim}[theorem]{Claim}

\newcommand{\T}{\mathbb{T}}

\newcommand{\chu}{\chi^{uu}}
\newcommand{\chcu}{\chi^{c}_+}
\newcommand{\chcs}{\chi^{c}_-}

\newcommand{\Z}{\mathbb{Z}}
\newcommand{\N}{\mathbb{N}}
\newcommand{\R}{\mathbb{R}}
\newcommand{\C}{\mathscr{C}}

\newcommand{\oz}{\overline{\zeta}}
\newcommand{\oet}{\overline{\eta}}
\newcommand{\tx}{\tilde{\xi}}
\newcommand{\tb}{\tilde{\mathcal{B}}}

\newcommand{\information}{{
  \bigskip
  \footnotesize
  	\textbf{Davi Obata}: \textsc{Department of Mathematics, University of Chicago, Chicago, IL, USA, 60637 } \par\nopagebreak
  \textit{E-mail:} \texttt{daviobata@uchicago.edu}
}}

\title{Open sets of partially hyperbolic skew products having a unique SRB measure}
\author{Davi Obata \footnote{D.O. was partially supported by the projects ANR BEKAM : ANR-15-CE40-0001 and ERC project 692925 NUHGD.}}
\date{\today}

\begin{document}

\maketitle
\begin{abstract}
In this paper we obtain $C^2$-open sets of dissipative, partially hyperbolic skew products having a unique SRB measure with full support and full basin. These partially hyperbolic systems have a two dimensional center  bundle which presents both expansion and contraction but does not admit any further dominated splitting of the center. These systems are non conservative perturbations of an example introduced by Berger-Carrasco.  

To prove the existence of SRB measures for these perturbations,  we obtain a general measure rigidity result for $u$-Gibbs measures for partially hyperbolic skew products. This is an adaptation to the partially hyperbolic setting of a measure rigidity result by A. Brown and F. Rodriguez Hertz for stationary measures of random product of surface diffeomorphisms.  In particular, we classify all the possible $u$-Gibbs measures that may appear in a neighborhood of the example. Using this classification, and ruling out some of the possibilities, we obtain open sets of systems, in a neighborhood of the example, having a unique $u$-Gibbs measure which is SRB. 
\end{abstract}
\setcounter{tocdepth}{1}
\tableofcontents

\section{Introduction}
 In dynamics one usually tries to understand the asymptotic behavior of the orbit of many points. In this direction, it is natural to try to understand properties, and the existence, of certain invariant measures that capture the statistical behavior of a set of points that is relevant for the Lebesgue measure. Let us make this more precise.  In what follows, we refer the reader to section \ref{sec.preliminaries} for the definitions of the dynamical objects that appear in this section. 

Let $f$ be a diffeomorphism of a closed, compact, connected, orientable  manifold $M$. Given an invariant ergodic probability measure $\mu$, its \textbf{basin} is defined as
\[
B(\mu) =\left\{ p\in M: \displaystyle \frac{1}{n} \sum_{j=0}^{n-1} \delta_{f^j(p)} \xrightarrow{n\to +\infty} \mu\right\},
\] 
where $\delta_p$ is the dirac measure on $p$ and the convergence is for the weak*-topology. The measure $\mu$ is \textbf{physical} if its basin has positive Lebesgue measure. In other words, physical measures are the measures that capture the asymptotic behavior of many points in the Lebesgue point of view.

In the 1970s, Sinai, Ruelle and Bowen \cite{ch5sinai72,ruelle,ch5bowenbook} proved that $C^{1+\alpha}$ uniformly hyperbolic systems have finitely many physical measures that describes the statistical behavior of Lebesgue almost every point. Nowadays, the measures they constructed are called \textbf{SRB measures} (SRB for Sinai-Ruelle-Bowen), see Definition \ref{defi.srb}. These measures have an important geometrical property: they admit conditional measures along unstable manifolds which are absolutely continuous with respect to the volume of the unstable manifolds. After the work of Ledrappier in \cite{ledrappiersrb}, there is a well developed ergodic theory for these measures. The hyperbolic  SRB measures form an important class of physical measures. 

We remark that in the hyperbolic setting there are uniform expansion/contraction, and a dominated splitting (which implies that the angle between the expanding/contracting directions is uniformly bounded from below). These two points are important to carry the constructions of such measures.

There are many works that study conditions that guarantee the existence of hyperbolic SRB measures outside the uniformly hyperbolic setting, see for instance \cite{ch5young, ch5bonattiviana, alvesbonattiviana, climenhagadolgopyatpesin, climenhagaluzzattopesinsurface, ovadia}. We also refer the reader to the recent survey \cite{climenhagaluzzattopesin} for a discussion on the different methods of construction of such measures (with a focus on the geometrical method). We now mention some of the examples of systems admitting hyperbolic SRB measures.
\begin{itemize}
\item Some derived from Anosov examples, in particular the ones introduced by Bonatti-Viana in \cite{ch5bonattiviana}: these examples have a dominated splitting, and nonuniform expansion, or contraction (also known as mostly contracting, or mostly expanding), see also \cite{alvesbonattiviana, ch5tahzibi}. It gives open sets of systems having an unique hyperbolic SRB measure.
\item H\'enon maps: in \cite{benedicksyoung} it is proved that for a set of positive Lebesgue measure of parameters $(a,b)$ with $b>0$ small, the map
\[
h_{a,b}(x,y) = (x^2+y+a,-bx),
\] 
admits a hyperbolic SRB measure. This example has non-uniform expansion/contraction, it is dissipative (it does not preserve the Lebesgue measure), and it does not admit a dominated splitting. However, it is not guaranteed the existence of an SRB measure for an open set of parameters $(a,b)$.
\item Some robustly non-uniformly hyperbolic volume preserving diffeomorphisms whose Oseledec's splitting is not dominated: by the absolute continuity of the unstable partition after the work of Pesin (see \cite{ch5pesin77}), in the volume preserving scenario, the existence of a hyperbolic SRB measure is equivalent to prove non-uniform hyperbolicity. Let us mention a few of such examples. The Berger-Carrasco's example in \cite{ch5bergercarrasco2014} (which we will study in more detail in this work). We also mention Avila-Viana in \cite{ch5avilavianainvariance}, and Liang-Marin-Yang in \cite{ch5lmy}, where they obtain $C^2$-open sets of symplectomorphisms which are non-uniformly hyperbolic. These examples are conservative, they have non-uniform expansion/contraction, and the expanding/contracting directions are not dominated.
\item Some ``large local'' perturbations of Axiom A systems, which appeared in \cite{climenhagadolgopyatpesin}: these examples also present non-uniform expansion/contraction, and no dominated splitting. But the proof of the existence of a hyperbolic SRB measure does not guarantee the robust existence of a hyperbolic SRB measure. 
\end{itemize} 

We remark that the list above is not a complete list of examples, but they represent well the examples according to the presence of non-uniform expansion/contraction, domination, and volume preserving or not. 

In this work we  give examples of open sets of dissipative systems having a unique SRB measure in the presence of non-uniform expansion/contraction and ``no domination'' between expanding and contracting directions. These properties create many difficulties in the study of the existence and uniqueness of SRB measures. 

The example we will study was introduced in \cite{ch5bergercarrasco2014} by Berger-Carrasco. It is a partially hyperbolic system, with two dimensional center,  and such that among the volume preserving systems it is robustly non-uniformly hyperbolic with both expansion/contraction along the center and it does not admit a decomposition of the center in dominated directions.

In \cite{obataergodicity}, the author proves that the Berger-Carrasco's example and any $C^2$-small volume preserving perturbation of it is ergodic. In this work we study dissipative perturbations of this example. In particular, we will find an open set of systems having a unique hyperbolic SRB measure with full basin, and each system in this open set has non-uniform expansion/contraction whose angle between the expanding/contracting directions is not bounded away from zero.

\subsection*{The example and precise statement of the results}

For $N \in \R$ we denote by $s_N(x,y) = (2x-y + N\sin(x), x)$ the standard map on $\T^2 := \R^2/ 2\pi \Z^2$. For every $N$ the map $s_N$ preserves the Lebesgue measure induced by the usual metric of $\T^2$. This map is related to several physical problems, see for instance \cite{ch5chirikovstandard}, \cite{ch5i80chaos} and \cite{ch5s95chaos}. 

It is conjectured that for $N \neq 0$ the map $s_N$ has positive entropy for the Lebesgue measure, see \cite{sinaiergodictheory} page $144$. By Pesin's entropy formula, see \cite{ch5pesin77} Theorem $5.1$, this is equivalent to the existence of a set of positive Lebesgue measure, whose points have a positive Lyapunov exponent. The existence of those sets is not known for any value of $N$. See \cite{ch5young17standard, ping, ch5duarte94standard, ch5goro12standard} for some results related to this conjecture.

In what follows we refer the reader to Section \ref{sec.preliminaries} for some basic definitions regarding partially hyperbolic dynamics.  Let $A\in SL(2,\Z)$ be a hyperbolic matrix that defines an Anosov diffeomorphism on $\T^2$, let $P_x: \T^2 \to \T^2$ be the projection on the first coordinate of $\T^2$, this projection is induced by the linear map of $\R^2$, which we will also write $P_x$, given by $P_x(a,b) = (a,0)$. 

Consider the torus $\T^4= \T^2 \times \T^2$ and represent it using the coordinates $(x,y,z,w)$, where $x,y,z,w \in [0,2\pi)$. We may naturally identify a point $(z,w)$ on the second torus with a point $(x,y)$ on the first torus by taking $x=z$ and $y=w$. For each $N \in \N$ define
$$
\begin{array}{rcccc}
f_N &:& \T^2\times \T^2 & \longrightarrow & \T^2\times\T^2\\
&&(x,y,z,w)& \mapsto & (s_N(x,y) + P_x\circ A^N(z,w),A^{2N}(z,w)).
\end{array}
$$

This diffeomorphism preserves the Lebesgue measure. For $N$ large enough it is a partially hyperbolic diffeomorphism, with two dimensional center direction given by $E^c = \R^2 \times \{0\}$. This type of system was considered by Berger-Carrasco in \cite{ch5bergercarrasco2014}, where they proved that for $N$ large enough $f_N$ is $C^2$-robustly non-uniformly hyperbolic among the volume preserving diffeomorphisms.

For $r\geq 1$ we consider $\mathrm{Diff}^r(\T^4)$ to be the set of $C^r$-diffeomorphisms of $\T^4$. Inside $\mathrm{Diff}^r(\T^4)$, we may consider the subspace $\mathrm{Sk}^r(\T^2\times \T^2)$ of skew products, which is the set of $C^r$-diffeomorphisms $g$ of the form
\[
g(x,y,z,w) = (g_1(x,y,z,w), g_2(z,w)),
\]
where $g_2(.,.)$ is a $C^r$-diffeomorphism of $\T^2$, and for each $(z,w) \in \T^2$, $g_1(.,.,z,w)$ is a $C^r$-diffeomorphism of $\T^2$ as well.   Observe that $f_N\in \mathrm{Sk}^2(\T^2 \times \T^2)$. We also remark that for $N$ large enough, if $g$ is a skew product $C^1$-close enough to $f_N$, then $g_2$ is an Anosov diffeomorphism, and $g$ is partially hyperbolic.  

We recall that for a map $g$, a $g$-invariant measure $\mu$ is \textbf{Bernoulli} if the system $(g,\mu)$ is measurably conjugated to a Bernoulli shift. 

Our main result is the following:

\begin{maintheorem}
\label{thm.thmA}
Let $\alpha \in (0,1)$. For $N$ large enough, there exist $\mathcal{U}_N^{sk}$ a $C^2$-neighborhood of $f_N$ contained in $\mathrm{Sk}^2(\T^2\times \T^2)$, and $\mathcal{V}$ a $C^2$-open and $C^2$-dense subset of $\mathcal{U}_N^{sk}$ such that for any $g\in \mathcal{V}$ having regularity $C^{2+\alpha}$,  there exists a unique $g$-invariant measure $\mu_g$ with the following properties:
\begin{enumerate}
\item $\mu_g$ is a hyperbolic SRB measure and Bernoulli;
\item $\mathrm{Leb}(B(\mu_g)) = 1$;
\item $\mathrm{supp(\mu_g)} = \T^4$.
\end{enumerate}
\end{maintheorem}

The proof of Theorem \ref{thm.thmA} is based in the study of the so called \textbf{$u$-Gibbs measures}, see Definition \ref{defi.ugibbs}. These measures play a key role in the study of ergodic properties of partially hyperbolic systems. Indeed, they capture the asymptotic statistical behavior of Lebesgue almost every point, see Theorem \ref{thm.bdvugibbs}. For a partially hyperbolic diffeomorphism $g$, we write $\mathrm{Gibbs}^u(g)$ as the set of $u$-Gibbs measures for $g$.

To prove Theorem \ref{thm.thmA} we will first classify all the ergodic $u$-Gibbs measures that may appear in a neighborhood of $f_N$. This is given in the following theorem:

\begin{maintheorem}
\label{thm.thmB}
Let $\alpha \in (0,1)$. For $N$ large enough, there exists $\mathcal{U}_N^{sk}$ a $C^2$-neighborhood of $f_N$ contained in $\mathrm{Sk}^2(\T^2\times \T^2)$, such that for $g\in \mathcal{U}_N^{sk}$ having regularity $C^{2+\alpha}$,  if $\mu \in \mathrm{Gibbs}^u(g)$ is ergodic, then either:
\begin{enumerate}
\item $\mu$ is a hyperbolic SRB measure, or
\item there exists a finite number  of $C^1$ two dimensional tori $T^1_{\mu}, \cdots, T^l_{\mu}\subset \T^4$ such that each of them is tangent to $E^{ss}_g \oplus E^{uu}_g$, and $\mathrm{supp}(\mu) = \cup_{j=1}^l\T^j_{\mu}$. 
\end{enumerate}
\end{maintheorem}

The proof of Theorem \ref{thm.thmB} uses an adaptation to the partially hyperbolic skew product setting of a recent result by Brown-Rodriguez Hertz in \cite{brownhertz}. In their paper they classify all the ergodic, hyperbolic stationary measures for random products of surface $C^2$-diffeomorphisms. Their proof is inspired in ideas from Benoist-Quint \cite{benoistquint1} and Eskin-Mirzakhani \cite{eskinm}.  In the partially hyperbolic skew product setting, we can actually get a result more general than Theorem \ref{thm.thmB}, see Theorem \ref{thm.generalmeasurerigidity} below.

We remark that there are also some recent works that ``push'' the ideas from \cite{benoistquint1, eskinm, el, brownhertz} to different settings. There is the work of Cantat-Dujardin in \cite{cantatdujardin} which attempts to classify stationary measures of random products of automorphisms of real and complex projective surfaces. There is also the work of Katz, \cite{katz}, which pushes the ideas of \cite{eskinm, el} to prove rigidity of ``$u$-Gibbs measures'' of Anosov flows under a technical hypothesis called QNI (quantified non-integrability). 

The uniqueness of the SRB measure, and some other properties that appear in the statement of  Theorem \ref{thm.thmA}, will be  a consequence of the following theorem:

\begin{maintheorem}
\label{thm.thmC}
For $N$ large enough, there exists $\mathcal{U}_N$ a $C^2$-neighborhood of $f_N$ in $\mathrm{Diff}^2(\T^4)$ such that if $g\in \mathcal{U}_N$, then $g$ has at most one SRB measure. Moreover, if $\mu_g$ is an SRB measure for $g$, then $\mathrm{supp}(\mu_g) =\T^4$, it is Bernoulli and hyperbolic.
\end{maintheorem}

\begin{remark}
Theorems \ref{thm.thmA} and \ref{thm.thmB} hold for a neighborhood of $f_N$ inside the set of skew product diffeomorphisms, $\mathrm{Sk}^2(\T^2\times \T^2)$. Theorem \ref{thm.thmC} guarantees that there exists at most one SRB measure in a neighborhood of $f_N$ inside $\mathrm{Diff}^2(\T^4)$. However, it does not guarantee the existence of an SRB measure. 
\end{remark}

As we mentioned before, the proof of Theorem \ref{thm.thmB} uses the following theorem, which holds for more general partially hyperbolic skew products and not only perturbations of Berger-Carrasco's example.  Let $S$ be a compact surface. We can define $\mathrm{Sk}^r(S\times \T^2)$ as the set of $C^r$-diffeomorphisms $g$ of the form 
\[
\begin{array}{rcc}
S\times \T^2 & \to & S\times \T^2 \\
(p_1, p_2) & \mapsto & (g_1(p_1, p_2), g_2(p_2))
\end{array}
\]
such that $g_2(.)$ is a $C^r$-diffeomorphism of $\T^2$ and for each $p_2\in \T^2$, $g_1(., p_2)$ is a $C^r$-diffeomorphism of $S$. We say that $g$ is a partially hyperbolic skew product of $S\times \T^2$ if $g$ is partially hyperbolic and $g_2$ is an Anosov diffeomorphism of $\T^2$.  Let $g$ be a partially hyperbolic skew product of $S\times \T^2$.  In what follows we write $\|Dg|_{E^{ss}}\| := \displaystyle \sup_{p\in S\times \T^2} \|Dg(p)|_{E^{ss}}\|$ and $m(Dg|_{E^c}): = \displaystyle \inf_{p\in S\times \T^2} m(Dg(p)|_{E^c})$, where $m(Dg(p)|_{E^c}): = \|\left(Dg(p)|_{E^c}\right)^{-1}\|^{-1}$ is the co-norm of $Dg(p)|_{E^c}$.

\begin{maintheorem}
\label{thm.generalmeasurerigidity}
Let  $S$ be a compact surface and let $\alpha,\theta  \in (0,1)$ be two constants. Let $g\in \mathrm{Sk}^{2+\alpha}( S\times \T^2)$ be a partially hyperbolic skew product  such that:
\begin{enumerate}[label =(\alph*)]
\item $g$ is $(2,\alpha)$-center bunched (see \eqref{eq.newbunching} for the definition) ;
\item $E^{uu}$ is $\theta$-H\"older and $\|Dg|_{E^{ss}}\|^{\theta} < m(Dg|_{E^c})$.
\end{enumerate}
If $\mu \in \mathrm{Gibbs}^u(g)$ is an ergodic measure having one positive and one negative Lyapunov exponent along the center direction, then either:
\begin{enumerate}
\item $\mu$ is an SRB measure;
\item the Oseledets direction $E^-$ is invariant by linear unstable holonomies (see item $2$ of Theorem \ref{thm.rigidity1u} for a precise definition);
\item there exist a finite number of two dimensional $su$-tori $T^1_{\mu}, \cdots, T^{l}_\mu$, such that $\mathrm{supp}(\mu) = \cup_{j=1}^l T^j_{\mu}$.
\end{enumerate}
\end{maintheorem}

This theorem will be a direct consequence of the combination of Theorems \ref{thm.theoremurigiditygeneral} and \ref{thm.atomicugibbs} below (see also Remark \ref{remark.ucomment}).

We remark that Theorem \ref{thm.generalmeasurerigidity} has its own interest,  since it gives a good ``general'' strategy to approach the problem of existence of SRB measures for partially hyperbolic skew products with two-dimensional fibers.

\subsection*{Discussion on the techniques and strategy of the proofs}

Theorem \ref{thm.thmA} is an easy consequence of Theorems \ref{thm.thmB}, \ref{thm.thmC}, and of some recent results on accessibility classes for skew products with two dimensional fibers from \cite{ch5horitasambarino2017}, given by Theorem \ref{thm.nontrivialaccessclasses} below.

Using the calculations to prove non-uniform hyperbolicity of $f_N$ from \cite{ch5bergercarrasco2014}, and the adaptations made in \cite{obataergodicity}, we prove that in a neighborhood of $f_N$ in $\mathrm{Diff}^2(\T^4)$, every $u$-Gibbs measure is hyperbolic with both a positive and a negative Lyapunov exponent along the center.

The proof of Theorem \ref{thm.thmC} is based on the techniques developed by the author in \cite{obataergodicity}. Using such techniques we can prove that any $u$-Gibbs measure has a set of large measure, whose points have ``large'' stable and unstable manifolds. Furthermore, we can obtain precise control on the ``geometry'' of these invariant manifolds. This allows us to prove that any two $u$-Gibbs measures are homoclinically related (see Definition \ref{defi.homrelatedmeasures} and Theorem \ref{thm.homoclinicrelatedugibbs}). Hence, we conclude that in a neighborhood of $f_N$ (inside $\mathrm{Diff}^2(\T^4)$) there exists at most one SRB measure. The techniques will also allow us to conclude that such a measure is Bernoulli. Using some arguments from the recent work \cite{carrascoobata} of the author with P. Carrasco, we prove that if there exists an SRB measure then it has full support.  A key point in this proof is a quantified version of Pesin theory that appeared in \cite{ch5crovisierpujals2016}.  We remark that this type of strategy using this quantified Pesin theory allowed the author to prove the uniqueness of the measure of maximal entropy for the standard map itself (see \cite{obmme}).
 
One of the key ingredients in the proof of Theorem \ref{thm.thmB} is an adaptation for the partially hyperbolic skew product setting of the main results from \cite{brownhertz}. There are two parts in this adaptation, which are given by Theorems \ref{thm.rigidity1u} and \ref{thm.atomicugibbs}.  To prove Theorem \ref{thm.rigidity1u}, we show that for $g$ sufficiently close to $f_N$ and for an ergodic $u$-Gibbs measure $\mu$,  after a measurable change of coordinates using the unstable holonomies, we are in the setting of Theorem $4.10$ from  \cite{brownhertz}. To justify that the change of coordinates mentioned above take us to the setting of Brown-Rodriguez Hertz's rigidity result, we use the version of the invariance principle by Tahzibi-Yang in \cite{tahzibiyang}.  We then obtain that there are only three possibilities for an ergodic $u$-Gibbs measure: either it is an SRB measure; or it has atomic disintegrations along the center foliation; or the Oseledets direction for the negative center Lyapunov exponent is invariant by the derivative of unstable holonomies.  Using some estimates from \cite{ch5bergercarrasco2014}, we prove that the third case never happens (see proposition \ref{prop.noninvariancestable}).  We are left to deal with the $u$-Gibbs measures having an atomic disintegration along the center foliation. This is done with Theorem \ref{thm.atomicugibbs}.

Theorem \ref{thm.atomicugibbs} corresponds to the adaptation of Theorem $4.8$ from \cite{brownhertz}.  The proof of this theorem is done in Sections \ref{section.atomiccenter} and \ref{subsection.sinvariance}.  If the $u$-Gibbs measure has atomic disintegration along the center foliation and the stable Oseledets direction is not invariant by the derivative of unstable holonomies, we prove that the center disintegration is invariant by stable and unstable holonomies. Since the system also verifies a condition called center bunching (see Definition \ref{ob.centerbunched}), using some results on accessibility classes (see Theorem \ref{thm.accclassesc1}), we may conclude the existence of the tori tangent to the strong stable and unstable directions (see Theorem \ref{thm.atomicugibbs}) which contain the support of the measure.

Let us finish with a remark on item $(b)$ in the hypothesis of Theorem \ref{thm.generalmeasurerigidity}. This condition states that we need $E^{uu}$ to be ``H\"older enough'' to apply the theorem.  It is well known that the invariant directions of a partially hyperbolic diffeomorphism are usually H\"older. Let $g$ be a partially hyperbolic skew product.  If $\theta \in (0,1)$ is a number such that
\[
\displaystyle \frac{\|Dg(p)|_{E^c}\|}{m(Dg(p)|_{E^{uu}})} < m(Dg(p)|_{E^{ss}})^{\theta},
\]
for every point $p\in S\times \T^2$, then $E^{uu}$ is $\theta$-H\"older (see Section 4 from \cite{ch5pughshubwilkinson12}).  This condition gives an upper bound on $\theta$. Indeed, we obtain that
\[
\theta < \inf_{p\in S\times \T^2} \left\lbrace \displaystyle \frac{\log m(Dg(p)|_{E^{uu}}) - \log \|Dg(p)|_{E^c}\|}{-\log m(Dg(p)|_{E^{ss}})}\right\rbrace.
\]

On the other hand, to obtain condition $(b)$ in the hypothesis of Theorem \ref{thm.generalmeasurerigidity} we need that $\|Dg|_{E^{ss}}\|^{\theta} < m(Dg|_{E^c})$, which implies 
\[
\theta > \displaystyle \frac{\log m(Dg|_{E^c})}{\log\|Dg|_{E^{ss}}\|}.
\]
Thus, a sufficient condition to obtain the hypothesis $(b)$ is that
\[
 \displaystyle \frac{\log m(Dg|_{E^c})}{\log\|Dg|_{E^{ss}}\|} <\inf_{p\in S\times \T^2} \left\lbrace \displaystyle \frac{\log m(Dg(p)|_{E^{uu}}) - \log \|Dg(p)|_{E^c}\|}{-\log m(Dg(p)|_{E^{ss}})}\right\rbrace.
\] 

\subsection*{Further remarks and questions}

The $\alpha$ that appears in the statements of Theorems \ref{thm.thmA} and \ref{thm.thmB} and \ref{thm.generalmeasurerigidity}, only appears because in the statement of the main result from \cite{brownhertz}, the surface diffeomorphisms they consider have regularity $C^2$. If one obtains a version of their result for $C^{1+\beta}$-diffeomorphisms, then one could remove the $\alpha$ from the statement (see section \ref{section.rigidity}).

Let us make a few remarks about the skew product hypothesis in the statement Theorems \ref{thm.thmA}, \ref{thm.thmB} and \ref{thm.generalmeasurerigidity}. This condition implies that the center foliation is  smooth.  This is  used to prove Proposition \ref{prop.ugibbsequgibbs}, which states that we may use the invariance principle (see also Corollary \ref{cor.whatiwant}). We also use the smoothness of the center foliation to prove that an $u$-Gibbs measure projects to the unique SRB measure for the $C^2$-Anosov diffeomorphism on the basis (see Lemma \ref{lemma.projectugibbs}). This is important in our proof because SRB measures for a $C^2$-Anosov diffeomorphism have a property called local product structure (see Section \ref{section.atomiccenter}). This local product structure is a key property used to obtain Lemma \ref{lem.continuousdisintegration}, which is used in the proof of the existence of the $su$-tori in item $2$ of the statement of Theorem \ref{thm.thmB} (and item $3$ of Theorem \ref{thm.generalmeasurerigidity}).  It is an interesting question to know if one can remove the skew product condition in the hypothesis to work with more general partially hyperbolic systems with two dimensional center.

An important notion in the study of dynamical properties of partially hyperbolic systems is accessibility (see Section \ref{sec.preliminaries} for the definition). It is not known if $f_N$ is accessible or not. If it were, we would obtain several interesting consequences, such as:
\begin{itemize}
\item $f_N$ would be $C^1$-stably ergodic (we refer the reader to \cite{obataergodicity} for the definition and discussion on stable ergodicity);
\item for $N$ large and $\mathcal{U}_N^{sk}$ small enough such that Theorem \ref{thm.thmB} is satisfied, for any $g\in \mathcal{U}_N^{sk}\cap \mathrm{Diff}^{2+\alpha}(\T^4)$, there would be an unique $SRB$ measure $\mu_g$, which is Bernoulli, it has full support and full basin. Furthermore, this measure would be the unique $u$-Gibbs measure for $g$.
\end{itemize}

We emphasize the question made by Berger-Carrasco in \cite{ch5bergercarrasco2014}:
\begin{question}
For every $L>0$, does it exist $N\in [L, +\infty)$ such that $f_N$ is accessible?
\end{question}

An interesting strategy to prove the existence of an SRB measure in a neighborhood of $f_N$ inside $\mathrm{Diff}^2(\T^4)$ is to use the results from \cite{climenhagadolgopyatpesin}. In order to do that, one needs to prove that the condition called \textbf{effective hyperbolicity} is satisfied (see Section $1.2$ in \cite{climenhagadolgopyatpesin}). This condition seems hard to prove, however it could give the existence of SRB measures outside the fibered case.

\begin{question}
For $N$ large enough, for any diffeomorphism $g$ which is sufficiently $C^2$-close to $f_N$, does it hold that $g$ is effective hyperbolic?
\end{question}

In \cite{ch5vianamaps}, Viana introduced two examples of systems (sometimes called Viana maps), which exhibit non-uniformly hyperbolic attractors . The first one is an endomorphism (see Theorem A in \cite{ch5vianamaps}), which is an skew product over an expanding map of the circle (on the basis), and the dynamics on the fiber is based on the quadratic family. For this example there were several works that studied its ergodic properties, in particular the existence of SRB measure, see for instance \cite{alves1, alves2, alvesviana, buzzisestertsujii}. 

The second example introduced by Viana is a diffeomorphism on a $5$-dimensional manifold (see Theorem B in \cite{ch5vianamaps}). It is an skew product with a solenoid on the basis, and the dynamics on the fiber is based on the H\'enon maps (which are dissipative). Viana proved that Lebesgue almost every point has a positive Lyapunov exponent along the fiber. This example is not well understood. In particular, nothing has been done regarding the existence of SRB measures for this type of Viana maps.

\begin{question}
Can the same strategy we use to study SRB measures be applied to study the existence of SRB measure for the second type of Viana maps? 
\end{question}

\subsection*{Organization of the paper}
In Section \ref{sec.preliminaries}, we review several tools that we will use in this work. In particular, results on partially hyperbolic systems and accessibility classes, $u$-Gibbs and SRB measures, and the invariance principle. In section \ref{section.proofthmA} we prove Theorem \ref{thm.thmA} assuming Theorems \ref{thm.thmB} and \ref{thm.thmC}. Sections \ref{section.centerlyapunovexpoenents} and \ref{section.proofthmC} are dedicated to prove Theorem \ref{thm.thmC}. In these sections we show how the techniques from \cite{obataergodicity}, and \cite{ch5bergercarrasco2014}, are used to obtain precise control on the center Lyapunov exponents of $u$-Gibbs measures, and how to obtain the uniqueness of the SRB measure. 

 In Section \ref{section.rigidity} we state Theorem $4.10$ from \cite{brownhertz}, and we show how after a measurable change of coordinates of our systems we are in the setting of their result. In Section \ref{section.noninvariancestable} we prove that in a neighborhood of Berger-Carrasco's example, the Oseledets direction for the negative center Lyapunov exponent is not invariant by the derivative of unstable holonomies, for any $u$-Gibbs measure. 
 
In  Sections \ref{section.atomiccenter} and \ref{subsection.sinvariance}, we deal with the case where a $u$-Gibbs measure has atomic center disintegration. This is done by using the invariance principle and adapting the proof of Theorem $4.8$ from \cite{brownhertz}. In the appendix we prove that with some stronger bunching condition the strong unstable holonomy between center manifolds has regularity $C^2$, this is used in the proof of Theorem \ref{thm.thmB}.

\subsection*{Acknowledgments}
The author thanks Sylvain Crovisier for careful reading of this work, and useful conversations. The author also would like to thank Aaron Brown, Alex Eskin, Todd Fisher,  Mauricio Poletti, Rafael Potrie, Federico Rodriguez Hertz, Ali Tahzibi and Amie Wilkinson for useful conversations and comments.  The author also thanks the anonymous referee.

\section{Preliminaries}
\label{sec.preliminaries}
\subsection{Partial hyperbolicity, holonomies and accessibility classes}
\subsubsection*{Partial hyperbolicity and foliations}

A $C^r$-diffeomorphism $f$, with $r\geq 1$, is {\bf partially hyperbolic} if the tangent bundle has a decomposition $TM = E^{ss} \oplus E^c \oplus E^{uu} $, there is a riemannian metric on $M$ and continuous functions  $\chi^{ss}, \chi^{uu}, \chi^{c}_-, \chi^{c}_+:M\to \R$, such that for any $m\in M$ 
\[
\chi^{ss}(m)<1< \chi^{uu}(m) \textrm{ and } \chi^{ss}(m) < \chi^c_-(m) \leq \chi^c_+(m) < \chi^{uu}(m),
\]
it also holds
\[\arraycolsep=1.2pt\def\arraystretch{2}
\begin{array}{c}
\chi^c_-(m) \leq  m(Df(m)|_{E^c_m}) \leq  \|Df(m)|_{E^c_m}\| \leq   \chi^c_+(m);\\
\|Df(m)|_{E^{ss}_m}\| \leq  \chi^{ss}(m)  \textrm{ and } \chi^{uu}(m) \leq  m(Df(m)|_{E^{uu}_m}).
\end{array}
\]
If the functions in the definition of partial hyperbolicity can be taken constant, we say that $f$ is {\bf absolutely partially hyperbolic}.	
 	
It is well known that the distributions $E^{ss}$ and $E^{uu}$ are uniquely integrable, that is, there are two unique foliations $\mathcal{F}^{ss}$ and $\mathcal{F}^{uu}$, with $C^r$-leaves, that are tangent to $E^{ss}$ and $E^{uu}$ respectively. For a point $p\in M$ we will denote by $W^{ss}(p)$ a leaf of the foliation $\mathcal{F}^{ss}$, we will call such leaf the strong stable manifold of $p$. Similarly we define the strong unstable manifold of $p$ and denote it by $W^{uu}(p)$.

\begin{definition}
\label{ob.centerbunched}
A partially hyperbolic diffeomorphism is {\bf center bunched} if   
\[
\chi^{ss}(m) < \frac{\chi^c_-(m)}{\chi^c_+(m)} \textrm{ and }\frac{\chi^c_+(m)}{\chi^c_-(m)} < \chi^{uu}(m), \textrm{ for every $m\in M$}.
\]
\end{definition}

We denote $E^{cs} = E^{ss} \oplus E^c$ and $E^{cu} = E^c \oplus E^{uu} $.

\begin{definition}
\label{ob.dynamicalcoherence}
A partially hyperbolic diffeomorphism $f$ is {\bf dynamically coherent} if there are two invariant foliations $\mathcal{F}^{cs}$ and $\mathcal{F}^{cu}$, with $C^1$-leaves, tangent to $E^{cs}$ and $E^{cu}$ respectively. From those two foliations one obtains another invariant foliation $\mathcal{F}^c = \mathcal{F}^{cs} \cap \mathcal{F}^{cu}$ that is tangent to $E^c$. We call those foliations the center-stable, center-unstable and center foliation.

\end{definition}

For any $R>0$ we write $W^*_R(p)$ to be the disc of size $R$ centered on $p$, for the Riemannian metric induced by the metric on $M$, contained in the leaf $W^*(p)$, for $*=ss,c,uu$. 

The definition below allows one to obtain higher regularity of the leaves of such foliations.

\begin{definition}
\label{ob.rnormalhyperbolicity}
We say that a partially hyperbolic diffeomorphism $f$ is  {\bf $r$-normally hyperbolic} if for any $m\in M$
\[
\chi^{ss}(m)< (\chi^c_-(m))^r \textrm{ and } (\chi^c_+(m))^r< \chi^{uu}(m).
\] 
\end{definition}

\begin{definition}
Let $f$ and $g$ be partially hyperbolic diffeomorphisms of $M$ that are dynamically coherent. Denote by $\mathcal{F}^c_f$ and $\mathcal{F}^c_g$ the center foliations. We say that $f$ and $g$ are {\bf leaf conjugated} if there is a homeomorphism $h:M\to M$ that takes leaves of $\mathcal{F}^c_f$ to leaves of $\mathcal{F}^c_g$ and such that for any $L \in \mathcal{F}^c_f$ it is satisfied
\[
h(f(L)) = g(h(L)).
\]

\end{definition}

One may study the stability of partially hyperbolic systems up to leaf conjugacy. Related to this there is a technical notion called {\bf plaque expansivity} which we will not define here, see chapter 7 of \cite{ch5hps} for the definition. The next theorem is important for the theory of stability of partially hyperbolic systems.

\begin{theorem}[\cite{ch5hps}, Theorem $7.4$]
\label{ob.leafconjugacy}
Let $f: M \to M$ be a $C^r$-partially hyperbolic and dynamically coherent diffeomorphism. If $f$ is $r$-normally hyperbolic and plaque expansive then any $g:M \to M$ in a $C^r$-neighborhood of $f$ is partially hyperbolic and dynamically coherent. Moreover, $g$ is leaf conjugated to $f$ and the center leaves of $g$ are $C^r$-immersed manifolds. 
\end{theorem}

\begin{remark}
\label{ob.continuitycoherent}
Fix $R>0$, and let $f$ be a diffeomorphism that satisfies the hypothesis of the previous theorem. The proof of this theorem implies that for $g$ sufficiently $C^r$-close to $f$, for any $m\in M$ we have that $W^c_{f,R}(m)$ is $C^r$-close to $W^c_{g,R}(m)$. In particular, if the center foliation is uniformly compact then for every $g$ sufficiently $C^r$-close to $f$, for any $m\in M$, $W^c_f(m)$ is $C^r$-close to $W^c_g(m)$.
\end{remark}

It might be hard to check the condition of plaque expansiviness, but this is not the case when the center foliation of a dynamically coherent, partially hyperbolic diffeomorphism is at least $C^1$, see Theorem $7.4$ of \cite{ch5hps}. Usually the invariant foliations that appear in dynamics are only H\"older.

We can also obtain a better regularity for the center direction given by the following theorem, see section $4$ of \cite{ch5pughshubwilkinson12} for a discussion on this topic.

\begin{theorem}
\label{ob.regularitycenter}
Let $f$ be a $C^2$-partially hyperbolic diffeomorphism and let $\theta>0$ be a number such that for every $m\in M$ it is satisfied
\[
\chi^{ss}(m)< \chi^c_-(m) m(Df(m)|_{E^{ss}})^{\theta} \textrm{ and } \chi^c_+(m) \|Df(m)|_{E^{uu}}\|^{\theta} < \chi^{uu}(m),
\]
then $E^c$ is $\theta$-H\"older.
\end{theorem}

\subsubsection*{Unstable holonomies}

Let $f$ be a partially hyperbolic, dynamically coherent diffeomorphism. Each leaf of the foliation $\mathcal{F}^{cs}$ is foliated by strong stable manifolds. For a point $p\in M$ and $q\in W^{ss}_{1}(p)$, where $W^{ss}_{1}(p)$ is the strong stable manifold of size $1$, we can define the stable holonomy map restricted to the center-stable manifold, between center manifolds. Let us be more precise. We can choose two small numbers $R_1,R_2>0$, with the property that for any $z\in W^c_{R_1}(p)$, there is only one point in the intersection $W^{ss}_{2}(z) \cap W^c_{R_2}(q)$. We define $H^s_{p,q}(z) =W^{ss}_{2}(z) \cap W^c_{R_2}(q)$. With this construction we obtain a map $H^s_{p,q}: W^c_{R_1}(p) \to W^c_{R_2}(q)$. By the compactness of $M$ we can take the numbers $R_1$ and $R_2$ to be constants, independent of $p$ and $q$. 

We can define analogously the unstable holonomy map, for $p\in M$ and $q\in W^{uu}_1(p)$, which we will denote by $H^u_{p,q}: W^c_{R_1}(p) \to W^c_{R_2}(q)$.

In \cite{ch5pughshubwilkinson97} and \cite{ch5pughshubwilkinsoncorrection}, the authors prove that the map $H^s_{p,q}$ is $C^1$ if $f$ is a partially hyperbolic, center bunched and dynamically coherent $C^2$-diffeomorphism. Indeed, the authors prove that the strong stable foliation is $C^1$ when restricted to a center-stable leaf. Consider the family of $C^1$-maps $\displaystyle \{H^s_{p,q}\}_{p\in M, q\in W^{ss}_1(p)}$. 

\begin{theorem}
\label{ob.holonomies}
Let $f$ be an absolutely partially hyperbolic, dynamically coherent diffeomorphism with regularity $C^2$. Suppose also that $f$ verifies:
\begin{enumerate}
\item $\chi^c_-<1$ and $\chi^c_+>1$;
\item there exists $\theta\in (0,1)$, such that
\begin{equation}
\label{eq.ch5thetabunching}
(\chi^{ss})^{\theta} < \frac{\chi^c_-}{\chi^c_+} \textrm{ and } \frac{\chi^c_+}{\chi^c_-} < (\chi^{uu})^{\theta};
\end{equation}
and also
\begin{equation}
\label{eq.ch5thetapinching}
\chi^{ss}< \chi^c_- m(Df|_{E^{ss}})^{\theta} \textrm{ and } \chi^c_+ \|Df|_{E^{uu}}\|^{\theta} < \chi^{uu}.  
\end{equation}
\end{enumerate} 
Then the family $\{H^s_{p,q}\}_{p\in M, q\in W^{ss}_1(p)}$ is a family of $C^1$-maps depending continuously in the $C^1$-topology with the choices of the points $p$ and $q$. Furthermore, there exists a constant $C>0$ such that for any $p\in M$, $q\in W^{ss}_1(p)$, and any unit vector $v\in E^c_p$, it is satisfied
\begin{equation}
\label{eq.holderconditionholonomy}
d\left(\frac{H^s_{p,q}(p)v}{\|H^s_{p,q}(p)v\|}, v\right) < Cd(p,q)^{\theta}.
\end{equation} 
Similar results holds for the family of unstable holonomies $\{H^u_{p,q}\}_{p\in M, q\in W^{uu}_1(p)}$. 
\end{theorem}
Theorem \ref{ob.holonomies} has no assumption on the dimensions of the invariant directions. The proof of this theorem can be found in \cite{ch5obataholonomies}, which is an adaptation of the arguments from \cite{ch5brownholonomy} by Brown. In what follows, we give the main points of this proof mostly to justify (\ref{eq.holderconditionholonomy}). For all the details, we refer the reader to \cite{ch5obataholonomies}.
\begin{proof}[Sketch of the proof]
By Theorem \ref{ob.regularitycenter}, condition (\ref{eq.ch5thetapinching}) implies that the center bundle $E^c$ is $\theta$-H\"older (see section $4$ in \cite{ch5pughshubwilkinson12}). The condition (\ref{eq.ch5thetabunching}) is sometimes called the \textbf{strong bunching} condition.

We may fix a local approximation of the holonomy $H^s_*$, which we will denote by $\pi^s_*$, that verifies the following: there exists a constant $\tilde{C}>0$ such that for any $p\in M$ and $q\in W^{ss}_1(p)$, there exists a $C^{1+\theta}$-map, which is a diffeomorphism onto its image, $\pi^s_{p,q}: W^c_{R_1}(p) \to W^c(q)$ that verifies
\begin{enumerate}
\item $d(\pi^s_{p,q}(p),q) \leq \tilde{C} d(p,q)$;
\item $d(D\pi^s_{p,q}(p).v,v) \leq \tilde{C} d(p,q)^{\theta}$, where $v\in SE^c_p$, and $SE^c_p$ is the unit sphere on $E^c_p$;
\item if $p' \in W^c_{loc}(p)$ and $q' \in W^{ss}_1(p')  \cap W^c_{loc}(q)$, then $\pi_{p,q}^s$ coincides with $\pi^s_{p',q'}$ on $W^c_{loc}(p) \cap W^c_{loc}(p')$.
\end{enumerate}

This can be done in the following way: Consider a smooth subbundle $\widetilde{E}$ inside a cone which is close to the direction perpendicular to the subbundle $E^c$, with dimension $\mathrm{dim}(M) - \mathrm{dim}(E^c)$. Since $E^c$ is $\theta$-H\"older, the center manifolds are $C^{1+\theta}$. Hence, the restriction of $\widetilde{E}$ to any center manifold is a $C^{1+\theta}$-bundle.  For each point $q\in M$ and $\rho>0$, consider $L_{q,\rho}:=\exp_q(\widetilde{E}(q,\rho))$ to be the projection of the ball of radius $\rho$ by the exponential map over $q$. By the uniform transversality and the compactness of $M$, there exists a constant $\rho_0$ such that for any center leaf $W^c_{R_1}(p)$, the set $\{L_{q,\rho}\}_{q\in W^c_{R_1}(p)}$ forms an uniform foliated neighborhood of $W^c_{R_1}(p)$ (or a tubular neighborhood). Let $\pi^s_{p,q}$ be the holonomy defined by this local foliation, up to rescaling of the metric we may assume that it is well defined for $p\in M$ and $q\in W^{ss}_1(p)$. By the compactness of $M$ we obtain the constant $\tilde{C}>0$ above. Observe also that since the center leaves vary continuously in the $C^1$-topology, we obtain that the map $\pi^s_{p,q}$ varies continuously in the $C^1$-topology with the points $p$ and $q$.

For any $p,q\in M$ and each $n\in \N$, write $p_n = f^n(p)$ and $q_n = f^n(q)$. We define 
\[
H^s_{p,q,n} = f^{-n} \circ \pi^s_{p_n,q_n} \circ f^n.
\]
If it is clear that we are talking about two points $p$ and $q\in W^{ss}_1(p)$ we will only write $H^s_n = H^s_{p,q,n}$ and similarly $\pi^s_n = \pi^s_{p_n,q_n}$.

Since we are assuming that $f$ is absolutely partially hyperbolic, only for this proof, we write its partially hyperbolic constants as $\chi_s = \chi^{ss}$, $\chi_c= \chi^c_-$ and $\widehat{\chi}_c = (\chi^c_+)^{-1}$. Also, for a diffeomorphism $g:N_1 \to N_2$, between manifolds $N_1$ and $N_2$, we will write $g_*: SN_1 \to SN_2$, the action induced by the derivative on the unitary bundles of $N_1$ and $N_2$.

The proof of Theorem \ref{ob.holonomies} follows the steps in \cite{ch5brownholonomy}. The first step is to prove that $(H^s_n)_{n\in \N}$ is uniformly Cauchy in the $C^0$-topology. The second step is to prove that the sequence $\left((H^s_n)_*\right)_{n\in \N}$ is uniformly Cauchy. The third step is to prove that for any vector $v\in E^c_p$, the sequence $\left(\|DH^s_n(p)v\|\right)_{n\in \N}$ is also uniformly Cauchy. In all these three steps it is obtained that the rate of convergence of these sequence does not depend on the choices of the points $p$ and $q$. The uniform convergence in the $C^1$-topology of the sequence $(H^s_n)_{n\in \N}$ then follows from these three steps. In this paper, we  only describe in more details step two, for the details of the other two steps we refer the reader to \cite{ch5obataholonomies}.

Observe that the Lipschitz norm of $f^{-1}_*$ restricted to a fiber $S_xE^c$ is $(\chi_c \widehat{\chi}_c)^{-1}$. Since $f$ is a $C^2$-diffeomorphism, then $f^{-1}_*$ is a $C^1$-diffeomorphism of $SM$, let $C_1>0$ be the $C^1$-norm of $f^{-1}$ on $M$ and $C_2$ to be the $C^1$-norm of $f^{-1}_*$ on $SM$. For $\xi =(x,v) \in S_xM$, write $\xi_k =f^k_*(x,v)= (x_k,v_k)$, with $k\in \Z$.

In \cite{ch5brownholonomy}, the author uses the strong bunching condition (\ref{eq.ch5thetabunching}) above, but he also uses another type of bunching (see Theorem $4.1$ in \cite{ch5brownholonomy}). In the proof, this different type of bunching is only used to obtain a version of lemma \ref{ob.auxiliarylemma} below. In our setting, instead of asking for this other type of bunching, we ask that $\chi_c<1$ and $\hat{\chi}_c<1$. We obtain the following lemma.   

\begin{lemma}
\label{ob.auxiliarylemma}
There are constants $\delta,\alpha\in (0,1)$, that satisfy the following: if $\xi=(x,v)$, $\zeta=(y,u) \in SW^c(p)$, $K>0$ and $n\geq 0$ verify $d(x_{n}, y_{n})< K \chi_s^n$, $d(\xi_n, \zeta_n) \leq K \chi_s^{n\theta}$, and for every $0\leq k \leq n$,
\[
d(x_k, y_k) \leq \delta. 
\]
Then, for all $0\leq k \leq n$,
\[
d(x_k,y_k) \leq K \chi_s^n.\chi_c^{-(n-k)} \textrm{ and } d(\xi_k, \zeta_k) \leq K \chi_s^{n\theta}.(\chi_c\widehat{\chi}_c)^{-(n-k)(1+\alpha)}.
\]
In particular, 
\[
d(\xi,\zeta) \leq K \chi_s^{n\theta}.(\chi_c \widehat{\chi}_c)^{-n(1+\alpha)}.
\]
Furthermore, $\alpha$ can be chosen such that 
\[
\chi_s^{\theta}.(\widehat{\chi}_c\chi_c)^{-(1+\alpha)}<1.
\]
\end{lemma}

\begin{proof}
The proof is by backward induction in $k$. We will first denote by $\alpha$ and $\delta$ quantities that will be fixed later. Since $x_k$ and $y_k$ belongs to the same center manifold, we obtain
\[
d(x_{k-1}, y_{k-1}) \leq \chi_c^{-1} d(x_k,y_k) \leq K\chi_s^n . \chi_c^{-n+k+1}.
\]

For any $\beta\in (0,1)$, and since $d(x_k, y_k) \leq \delta$, we have

\[
\arraycolsep=1.2pt\def\arraystretch{2}
\begin{array}{rcl}
d(f^{-1}_*(x_k,v_k), f^{-1}_*(y_k,u_k)) & \leq & d(f^{-1}_*(x_k,v_k), f^{-1}_*(x_k,u_k)) + d(f^{-1}_*(x_k,u_k), f^{-1}_*(y_k,u_k))\\
& \leq & (\chi_c \widehat{\chi}_c)^{-1} d(v_k,u_k) + C_2 d(x_k,y_k).\\
& \leq & (\chi_c \widehat{\chi}_c)^{-1}[1 + C_2 .( \chi_c \widehat{\chi}_c) d(x_k,y_k)^{1-\beta}] . \max \{ d(x_k,y_k)^{\beta},d(v_k,u_k) \}\\
&\leq & (\chi_c \widehat{\chi}_c)^{-1}[1 + C_2 .( \chi_c \widehat{\chi}_c) \delta^{1-\beta}] \\
&& .K \max \{ \chi_s^{n\beta}.\chi_c^{-(n-k)\beta},  \chi_s^{n\theta}.(\chi_c\widehat{\chi}_c)^{-(n-k)(1+\alpha)}\}.
\end{array}
\]

We claim that we can choose $\alpha$ and $\beta$ such that for any $n\in \N$ and $0\leq k \leq n$ it holds
\[
\chi_s^{n\beta}.\chi_c^{-(n-k)\beta} \leq \chi_s^{n\theta}.(\chi_c\widehat{\chi}_c)^{-(n-k)(1+\alpha)}.
\]

This inequality is equivalent to 
\begin{equation}
\label{ob.eq1}
1 \leq \chi_s^{n(\theta-\beta)}.(\chi^{(\beta-1-\alpha)}_c\widehat{\chi}^{-(1+\alpha)}_c)^{(n-k)}.
\end{equation}

Since $\widehat{\chi}^{-1}_c>1$, we can fix $\beta>\theta$ close enough to $1$ such that
$1<\chi^{(\beta-1-\alpha)}_c\widehat{\chi}^{-(1+\alpha)}_c$. Let us explain. Observe that $(\chi_c)^{-\alpha}>1$, for any $\alpha>0$. Hence, 
\[
\chi_c^{\beta-1}(\hat{\chi}_c \chi_c)^{-\alpha}\hat{\chi}_c^{-1} > \chi_c^{\beta-1}\hat{\chi}_c^{-1}.
\]
From this, one can see that if $\beta$ is sufficiently close to $1$, we have that $1<\chi^{(\beta-1-\alpha)}_c\widehat{\chi}^{-(1+\alpha)}_c$. Since $\beta>\theta$, and hence $\theta-\beta$ is negative, we conclude (\ref{ob.eq1}).  

We also need that 
\begin{equation}
\label{ob.eq2}
\chi_s^{\theta}.(\widehat{\chi}_c\chi_c)^{-(1+\alpha)}<1.
\end{equation}
By the strong center bunching condition (\ref{eq.ch5thetabunching}), the inequality above holds if $\alpha$ is sufficiently close to $0$. Fix $\alpha>0$ that verifies (\ref{ob.eq2}).

Now fix $\delta>0$ small enough such that
\[
[1 + C_2 .( \chi_c \widehat{\chi}_c) \delta^{1-\beta}]\leq (\chi_c \widehat{\chi}_c)^{-\alpha}.
\]

We conclude,
\[
\arraycolsep=1.2pt\def\arraystretch{2}
\begin{array}{rcl}
d(f^{-1}_*(\xi_k), f^{-1}_*(\zeta_k)) &\leq &(\chi_c \widehat{\chi}_c)^{-(1+\alpha)}.K  \chi_s^{n\theta}.(\chi_c\widehat{\chi}_c)^{-(n-k)(1+\alpha)}\\
&=& K \chi_s^{n\theta}.(\chi_c\widehat{\chi}_c)^{-(n-k-1)(1+\alpha)}
\end{array}
\]

\end{proof}

Fix $\xi= (z,l)\in SW_{R_1}^c(p)$. Write $ \zeta^n := (H^s_n)_*(\xi)$ and $\zeta^n_j:= f^j_*(\zeta^n)$, for any $j\in \Z$. We warn the reader to not confuse the notation $\zeta^n$ with the notation that we were using before $\zeta_n = f^n_*(\zeta)$, for a given $\zeta$. We also write $w = H^s_{p,q}(z)$, $\zeta^n =(H^s_n)_*(\xi)= (x,v)$ and $\zeta^{n+1}=(H^s_{n+1})_*(\xi)=(y,u)$. Observe that $\zeta^n_n = (\pi^s_n)_*(\xi_n)$ and $\zeta^{n+1}_n = f^{-1}_*( (\pi^s_{n+1})_*( \xi_{n+1}))$. First we have 
\[
\arraycolsep=1.2pt\def\arraystretch{1.2}
\begin{array}{rcl}
d(\pi^s_n(z_n), f^{-1}(\pi^s_{n+1}(z_{n+1}))) &\leq &d(z_n, \pi^s_n(z_n)) + d(f^{-1}(z_{n+1}), f^{-1}(\pi^s_n(z_{n+1}))\\
&\leq & \tilde{C} \chi_s^n d(z,w) + C_1 \tilde{C} \chi_s^{n+1}d(z,w)\\
& \leq & 2\tilde{C}C_1 \chi_s^n d(z,w).  
\end{array}
\]
 
 The previous estimate shows that $d(x_n,y_n) \leq 2\tilde{C}C_1d(z,w) \chi_s^n$. Also, it is satisfied for any $0 \leq k \leq n$
\begin{equation}
\label{ob.distanciapontos}
d(x_k,y_k) \leq 2 \tilde{C} C_1 d(z,w) \chi_s^n \chi_c^{-(n-k)}.
\end{equation}
Let $\delta$ be the constant given by lemma \ref{ob.auxiliarylemma}. By domination, if $n$ is large enough, we conclude that $d(x_k,y_k) < \delta$. This $n$ can be taken uniform, independently of $p$ and $q$. 

Also, using that $f^{-1}_* (\xi_{n+1}) = \xi_n$, we obtain
\[
\arraycolsep=1.2pt\def\arraystretch{1.2}
\begin{array}{rcl}
d(\zeta^n_n, \zeta^{n+1}_n)&=& d((\pi^s_n)_*(\xi_n), f^{-1}_* (\pi^s_{n+1})_*(\xi_{n+1}))\\
&\leq & d( \xi_n, (\pi^s_n)_*( \xi_n))+ d(f^{-1}_* (\xi_{n+1}), f^{-1}_*(\pi^s_{n+1})_* (\xi_{n+1})).
\end{array} 
\]
 
By property $2$ of $\pi^s_*$, we have $d( \xi_n, (\pi^s_n)_*( \xi_n)) \leq \tilde{C} d(z,w)^{\theta} \chi_s^{n\theta}$. For the second term in the inequality we have
\[
\arraycolsep=1.2pt\def\arraystretch{1.4}
\begin{array}{lcl}
d(f^{-1}_* (\xi_{n+1}), f^{-1}_*(\pi^s_{n+1})_* (\xi_{n+1})) & = & d(f^{-1}_*(z_{n+1},l_{n+1}), f^{-1}_*(y_{n+1},u_{n+1}))\\
& \leq & d(f^{-1}_*(z_{n+1},l_{n+1}), f^{-1}_*(z_{n+1},u_{n+1}))\\
&& +  d(f^{-1}_*(z_{n+1},u_{n+1}), f^{-1}_*(y_{n+1},u_{n+1}))\\
&\leq & C_2 d(l_{n+1},u_{n+1}) + C_2 d(z_{n+1}, y_{n+1})\\
& \leq & \tilde{C} C_2 d(z,w)^{\theta} \chi_s^{(n+1)\theta} + \tilde{C} C_2 d(z,w) \chi^{n+1}_s\\
&\leq &  (\tilde{C} C_2 + \tilde{C}C_2 d(z,w)^{1-\theta} \chi_s^{(n+1)(1-\theta)})d(z,w)^{\theta}\chi^{(n+1)\theta}_s\\
& \leq &(\tilde{C} C_2 + \tilde{C}C_2 d(z,w)^{1-\theta})d(z,w)^{\theta}\chi^{(n+1)\theta}_s.
\end{array} 
\] 

Thus, 
\[
d(\zeta^n_n, \zeta^{n+1}_n)  \leq [\tilde{C} + (\tilde{C} C_2 + \tilde{C}C_2 d(z,w)^{1-\theta})]d(z,w)^{\theta}\chi^{n\theta}_s.
\]
By compactness, $d(z,w)$ is bounded from above independently of $p$ and $q$. Hence, take a constant $C_3$ such that $d(\zeta^n_n, \zeta^{n+1}_n)  \leq C_3 d(z,w)^{\theta}\chi^{n\theta}_s$. Fix $K_1 = \max \{2 \tilde{C} C_1, C_3 \}$, and observe that we are in the setting of lemma \ref{ob.auxiliarylemma}, for $K = K(z,w) := K_1 d(z,w)^{\theta}$. Let $\alpha$ be the constant given by the same lemma. We conclude that 
\[
d(\zeta^n, \zeta^{n+1}) \leq K \chi_s^{n\theta}.(\chi_c \widehat{\chi}_c)^{-n(1+\alpha)} = K_1 \chi_s^{n\theta}.(\chi_c \widehat{\chi}_c)^{-n(1+\alpha)} d(z,w)^{\theta} , \textrm{ for $n$ large enough.}
\]
In particular, the sequence $(\zeta^n)_{n\in \N}$ is Cauchy. Since this holds uniformly for any $\xi$, we obtain that $\left((H^s_n)_*\right)_{n\in \N}$ is a Cauchy sequence whose speed of convergence does not depend on the choices of the the points $p$ and $q$.

If $d(p,q) \leq \delta$ then for any $n\geq 0$ it holds that 
\[
d((H^s_n)_* , (H^s_{n+1})_*) \leq K_1 \chi_s^{n\theta}.(\chi_c \widehat{\chi}_c)^{-n(1+\alpha)}d(p,q)^{\theta}.
\]
Write $(H^s_{p,q})_* = \displaystyle \lim_{n\to +\infty} (H^s_n)_*$. Hence, there exists a constant $K_2>0$ such that for $p,q\in M$ with $d(p,q) <\delta$, we have
\[
d(Id_*,(H^s_{p,q})_*) \leq d(Id_*, (\pi^s)_*) + \displaystyle \sum_{j=0}^{+\infty} d((H^s_j)_*, (H^s_{j+1})_*) \leq K_2d(p,q)^{\theta}.  
\] 
Since $\delta>0$ is a constant, there is a maximum number $T=[\frac{1}{\delta}]$ such that there are at most $T+1$ points, $\{x_1, \cdots, x_{T+1}\}\subset W^{s}_1(p)$ verifying $x_1 = p$, $x_{T+1} = q$ and $d(x_i, x_{i+1}) < \delta$. Since $H^s_{p,q}(.) = H^s_{x_T, x_{T+1}} \circ \cdots \circ H^s_{x_1, x_2}(.)$, we conclude that there exists a constant $C>0$ such that 
\begin{equation}
\label{ob.distancia1}
d(Id_*, (H^s_{p,q})_*) \leq C d(p,q)^{\theta}.
\end{equation}
This concludes the proof of the second step that we mentioned above. In particular, it also proves the conclusion (\ref{eq.holderconditionholonomy}) in the statement of this theorem.
\end{proof}

Suppose that $f$ is a partially hyperbolic, center bunched skew product on $\T^4= \T^2 \times \T^2$, with the Anosov map on the base $f_2: \T^2 \to \T^2$. Observe that for any $p\in \T^4$, its unstable manifold $W^{uu}(p)$ projects to the unstable manifold of $\pi_2(p)$ of $f_2$. In particular, for each $p\in \T^4$ and $q\in W^{uu}(p)$ and since the center leaves are uniformly compact (indeed they are just the fibers), the unstable holonomy map can be defined on the entire center leaf $H^u_{p,q}: W^c(p) \to W^c(q)$. By Theorem \ref{ob.leafconjugacy}, this property is $C^1$-open. 

Using the $f$-invariance of the center and strong unstable foliations, it is easy to see that for any $n\in \Z$, for each $p,q$ as above, we have
\[
H^u_{f^n(p),f^n(q)} \circ f^n = f^n \circ H^u_{p,q}.
\] 
We remark that in the skew product case, we may also use the notation $H^u_{p_2,q_2}$ to denote the unstable holonomy between $\pi_2^{-1}(p_2)$ and $\pi^{-1}_2(q_2)$, for $p_2$ and $q_2$ belonging to the same unstable manifold of $f_2$. Sometimes we will use this notation. 

\subsection*{Higher regularity of unstable holonomies}

Let $f$ be a $C^{2+\alpha}$ absolutely partially hyperbolic skew product of $\T^4 = \T^2 \times \T^2$ and let $\chi^{ss}, \chcs,\chcu , \chu$ be the partially hyperbolic constants of $f$. We say that $f$ verifies the \textbf{$(2,\alpha)$-center unstable bunching condition} if
\begin{equation}
\label{eq.newbunching}
\displaystyle  \left( \frac{\chcu}{\chcs}\right)^2 < \chu \textrm{ and } \frac{\chcu}{(\chcs)^2} < (\chu)^{\alpha}.
\end{equation}
Similarly, $f$ verifies the \textbf{$(2,\alpha)$-center stable bunching condition} if
\begin{equation}
\label{eq.stablenewbunching}
\displaystyle \chi^{ss}< \left(\frac{\chcs}{\chcu}\right)^2 \textrm{ and } (\chi^{ss})^{\alpha} < \frac{\chcs}{(\chcu)^2}.
\end{equation}
If $f$ verifies condition (\ref{eq.newbunching}) and (\ref{eq.stablenewbunching}) then we say that $f$ is \textbf{$(2,\alpha)$-center bunched}. 

We use the $(2,\alpha)$-center bunching condition to obtain $C^2$-regularity of the unstable holonomy inside a center unstable leaf. This is given in the following theorem.

\begin{theorem}
\label{thm.holonomyc2}
Let $f$ be a $C^{2+\alpha}$ absolutely partially hyperbolic skew product of $\T^4$, and fix $R>0$. If $f$ is $(2,\alpha)$-center unstable bunched, then $\{H^u_{p,q}\}_{p\in \T^4, q\in W_R^{uu}(p)}$ is a family of $C^2$-diffeomorphisms of $\T^2$ whose $C^2$-norm varies continuously with the choices of $p$ and $q$.
\end{theorem}

This theorem is proved in the appendix (see section \ref{appendix.A}).

\subsubsection*{ Accessibility classes}

For a partially hyperbolic diffeomorphism $f$, an \textbf{$su$-path} is a curve which is the concatenation of finitely many curves, each of them being contained in a stable or unstable leaf. Given a point $m\in M$, its \textbf{accessibility class} is defined as
\[
AC(m) = \{p\in M: \textrm{ there exists an $su$-path connecting $m$ and $p$.}\}
\]
We say that $f$ is \textbf{accessible} if for any $m\in M$, $AC(m) = M$. Suppose that $f$ is dynamically coherent, we say that an accessibility class $AC(m)$ is \textbf{trivial} if $AC(m) \cap W^c(m)$ is totally disconnected. We say that $f$ has the \textbf{global product structure} if there is a covering $\pi:\tilde{M} \to M$ and a lift $\tilde{f}: \tilde{M} \to \tilde{M}$ for any $\tilde{x}, \tilde{y} \in \tilde{M}$ we have
\[
\#\{\tilde{\mathcal{F}}^{cs}(\tilde{x}) \cap \tilde{\mathcal{F}}^{uu}(\tilde{y})\} = 1 \textrm{ and } \#\{\tilde{\mathcal{F}}^{cu}(\tilde{x}) \cap \tilde{\mathcal{F}}^{ss}(\tilde{y})\} = 1, 
\]
where $\tilde{\mathcal{F}}^*$ if the lift of the foliation $\mathcal{F}^*$, for $*= ss,cs,cu,uu$. We now describe some results from Horita-Sambarino in \cite{ch5horitasambarino2017}. In what follows we will restrict ourselves to the case that $M= \T^4$. 

We define $\mathcal{E} = \mathcal{E}^2(\T^4)$ to be the set of $C^2$-partially hyperbolic diffeomorphisms $f$ such that
\begin{itemize}
\item $f$ is dynamically coherent, $2$-normally hyperbolic and plaque expansive;
\item $f$ is center bunched;
\item $f$ has the global product structure;
\item the set of compact center leaves that are $f$-periodic is dense in $M$.
\end{itemize}
The set $\mathcal{E}$ is $C^1$-open in $\mathrm{Diff}^2(M)$.

Inside $\mathcal{E}$ let us define the set of skew-products over a fixed Anosov diffeomorphism. Let $g:\T^2 \to \T^2$ be a $C^2$-Anosov diffeomorphism, let $\mathcal{V}_g \subset \mathrm{Diff}^2(\T^2)$ be the open set such that if $h\in \mathcal{V}_g$ then $h\times g$ is partially hyperbolic, center bunched and $2$-normally hyperbolic. Let $f:\T^2 \to \mathcal{V}_g$ be a continuous map and denote by $f(.,y)$ the diffeomorphism $f(y):\T^2 \to \T^2$. We define the skew product given by $f$ over $g$ as
\[
f_g(x,y) =  (f(x,y), g(y)).
\]   
Observe that $f_g\in \mathcal{E}$. Let $\mathcal{E}^{sp}_g$ be the set of partially hyperbolic skew products over $g$ which take value on $\mathcal{V}_g$. That is, $f_g\in \mathcal{E}^{sp}_g$ if and only if $f_g(x,y) = (f(x,y), g(y))$, where $f(.,y) \in \mathcal{V}_g$ for every $y\in \T^2$. Observe that $\mathcal{E}^{sp}_g$ can be identified with $\mathcal{G}_g:= \{f:\T^2 \to \mathcal{V}_g, \textrm{ s.t. $f$ is continuous}\}$, since $g$ is fixed. We say that $f,\tilde{f}\in \mathcal{G}_g$ are $C^2$-close if for each $y\in \T^2$, the diffeomorphisms $f(.,y)$ and $\tilde{f}(.,y)$ are $C^2$-close. Of course, $\mathcal{E}_g^{sp} \subset \mathcal{E}$. We state the following theorem of Horita-Sambarino in our scenario, but we remark that their theorem is more general than the statement we give. 
\begin{theorem}[\cite{ch5horitasambarino2017}, Theorem 2]
\label{thm.hsthm2}
Let $g:\T^2 \to \T^2$ be a $C^2$-Anosov diffeomorphism, then the set $\mathcal{R}_0$ of diffemorphisms in $\mathcal{E}^{sp}_g$ whose accessibility classes are all non trivial is $C^1$-open and $C^2$-dense.
\end{theorem}

Another important result from \cite{ch5horitasambarino2017} is the following:
\begin{proposition}[\cite{ch5horitasambarino2017}, Corollary $4.3$]
\label{prop.hscor}
If $f\in \mathcal{E}$ has all its accessibility classes non trivial, then there exists a $C^1$-open neighborhood of $f$,  $\mathcal{V}(f)$, in $\mathcal{E}$ of partially hyperbolic diffeomorphisms whose accessibility classes are all non trivial. 
\end{proposition}

Recall that $\mathrm{Sk}^2(\T^2\times \T^2)$ is the set of $C^2$-diffeomorphisms $h$ that are skew products, that is, $h(x,y) = (h_1(x,y), h_2(y))$ where $x,y\in \T^2$ and $h(.,y)$ is a $C^2$-diffeomorphism of $\T^2$ that changes continuously with the choice of $y$.

 Recall that in the introduction we defined, for each $N\in \N$, the diffeomorphism $f_N (x,y,z,w) = (s_N(x,y) + P_x\circ A^N(z,w), A^{2N}(z,w))$, such that $s_N$ is the standard map, $P_x$ is the projection on the horizontal direction, and $A$ is a linear Anosov diffeomorphism on $\T^2$. Observe that $f_N$  belongs to $\mathrm{Sk}^2(\T^2\times \T^2)$. Furthermore, for $N$ large enough we have that $f_N$ is $2$-normally hyperbolic and center bunched, in particular, it belongs to $\mathcal{E}$. Using Theorem \ref{thm.hsthm2} and Proposition \ref{prop.hscor}, we obtain the following theorem.

\begin{theorem}
\label{thm.nontrivialaccessclasses}
For $N$ large enough, for each sufficiently small $C^1$-neighborhood $\mathcal{W}$ of $f_N$ in $\mathrm{Sk}^2(\T^2\times \T^2)$, there exists a set $\mathcal{V}\subset \mathcal{W}$, which is $C^1$-open and $C^2$-dense in $\mathcal{W}$ such that for any $g\in \mathcal{V}$ all its accessibility classes are non trivial. 
\end{theorem}
\begin{proof}
If $\mathcal{W}$ is sufficiently $C^1$-small, then for any $g\in \mathcal{W}$ the basis dynamics $g_2$ is a $C^2$-Anosov diffeomorphism which is $C^1$-close to $A^{2N}$. 

Let $\mathcal{N}$  be a small $C^1$-neighborhood of $A^{2N}$ in $\mathrm{Diff}^2(\T^2)$. For each $g_2\in \mathcal{N}$ we consider $\mathcal{W}_{g_2} = \mathcal{W} \cap \mathcal{E}^{sp}_{g_2}$ and observe that this set is $C^1$-open in $\mathcal{G}_{g_2}$.

By Theorem \ref{thm.hsthm2}, there exists a $C^1$-open and $C^2$-dense subset $\tilde{\mathcal{V}}_{g_2}$ of $\mathcal{W}_{g_2}$ such that for each skew product $g\in \tilde{\mathcal{V}}_{g_2}$ all its accessibility classes are non trivial. By Proposition \ref{prop.hscor}, for each $g\in \tilde{\mathcal{V}}_{g_2}$ there exists a $C^1$-open subset of $\mathcal{E}$, which we denote it by $\mathcal{V}(g)$, of diffeomorphisms whose accessibility classes are all non trivial. Now define
\[
\mathcal{V} := \displaystyle \bigcup_{g_2 \in \mathcal{N}} \bigcup_{g\in \tilde{\mathcal{W}}_{g_2}} \mathcal{V}(g).
\]
It is easy to see that $\mathcal{V}$ is $C^1$-open and $C^2$-dense in $\mathcal{W}$. Moreover, for each $g\in \mathcal{V}$ all its accessibility classes are non trivial.    
\end{proof}

In our work we will also need the following result that describes the structure of accessibility classes.

\begin{theorem}[\cite{janavasquez}, Theorem B]
\label{thm.accclassesc1}
Let $f$ be a dynamically coherent $C^2$-partially hyperbolic diffeomorphism with two dimensional center, and which is center bunched. Then every accessibility class is an immersed $C^1$-submanifold. 
\end{theorem}

\subsection{Pesin's theory and SRB measures} 
\label{subsection.pesintheory}

Let $f$ be a $C^1$-diffeomorphism. A number $\lambda \in \R$ is a \textbf{Lyapunov exponent} if there exists a point $p\in M$ and a non zero vector $v\in T_pM$ such that $\lim_{n\to \pm \infty} \frac{1}{n} \log \|Df^n(p)v\| = \lambda$. We write $\lambda(p,v):= \lim_{n\to \pm \infty} \frac{1}{n} \log \|Df^n(p)v\|$.

We say that a set $R$ has full probability if for any $f$-invariant probability measure $\nu$, $\nu(R)=1$.  The following theorem is known as the Oseledets theorem.

\begin{theorem}[\cite{ch5barreirapesinbook}, Theorems $2.1.1$ and $2.1.2$]
\label{ob.oseledets}
For any $C^1$-diffeomorphism $f$, there is a set $\mathcal{R}$ of full probability, such that for every $\varepsilon>0$ it exists a measurable function $C_{\varepsilon}: \mathcal{R} \to (1, +\infty)$ with the following properties:
\begin{enumerate}
\item for any $p\in \mathcal{R}$ there are numbers $s(p)\in \N$, $\lambda_1(p) < \cdots < \lambda_{s(p)}(p)$ and a decomposition $T_pM = E^{1}_p \oplus \cdots \oplus E^{s(p)}_p$;\\
\item $s(f(p)) = s(p)$, $\lambda_i(f(p)) = \lambda_i(p)$ and $Df(p).E^{i}_p= E^{i}_{f(p)}$, for every $i= 1, \cdots, s(p)$;
\item for every $v\in E^{i}_p- \{0\}$, $\lambda(p,v) = \lambda^i(p)$.
\end{enumerate} 
\end{theorem} 

We call the set $\mathcal{R}$ the set of \textbf{regular points}. A point $p\in \mathcal{R}$ has $k$ negative Lyapunov exponents if 
\[
\displaystyle \sum_{i: \lambda_i(p) <0} dim (E^i_p) = k.
\] 
Similarly for positive or zero Lyapunov exponents. From now on, we assume that $\nu$ is a $f$-invariant measure, not necessarily ergodic, and there are numbers $k$ and $l$ such that $\nu$-almost every point $p\in \mathcal{R}$ has $k$ negative and $l$ positive Lyapunov exponents. 

For a regular point we write 
\begin{equation}
\label{ob.oseledecsdirection}
E^s_p = \displaystyle \bigoplus_{i: \lambda_i(p)<0} E^i_p \textrm{ and } E^u_p = \bigoplus_{i: \lambda_i(p)>0} E^i_p.
\end{equation}

It is well known that for a $C^2$-diffeomorphism $f$ and an invariant measure $\nu$, then for $\nu$-almost every $p$, the set defined by 
\[
W^s(p) =\{ q\in M: \displaystyle \limsup_{n\to +\infty} \frac{1}{n} \log d(f^n(p), f^n(q)) <0 \}
\] 
is an immersed submanifold such that $T_pW^s(p) = E^s_p$ (see section $4$ of \cite{ch5pesin77}). We call $W^s(p)$ the \textbf{stable Pesin manifold} of the point $p$. Similarly, the set defined by 
\[
W^u(p) = \{ q\in M: \displaystyle \limsup_{n\to +\infty} \frac{1}{n} \log d(f^{-n}(p), f^{-n}(q)) <0 \}
\]
is an immersed submanifold such that $T_pW^u(p)= E^u_p$. We call $W^u(p)$ the \textbf{unstable Pesin manifold} of the point $p$. Since these manifolds exist for $\nu$-almost every point, the unstable manifolds $\{W^u(p)\}_{p\in \mathcal{R}}$ form a partition of a $\nu$-full measure subset of $M$.

\begin{remark}
\label{ob.osedetspesin}
If $f$ is also partially hyperbolic, with $TM = E^{ss} \oplus E^c \oplus E^{uu}$ then the Oseledets splitting refines the partially hyperbolic splitting. This means that for a regular point $p\in \mathcal{R}$, there are numbers $1\leq l_1< l_2 < s(p)$ such that 
\[
E^{ss}_p = \displaystyle \bigoplus_{i=1}^{l_1} E^i_p, \textrm{ } E^c_p = \bigoplus_{i=l_1+1}^{l_2} E^i_p \textrm{ and } E^{uu}_p = \bigoplus_{i=l_2+1}^{s(p)} E^i_p.
\]

This follows from a standard argument similar to the proof of the uniqueness of dominated splittings, see section $B.1.2$ from \cite{ch5bonattidiazvianabook}. It also holds that for any regular point $p$, $E^{ss}_p \subset E^s_p$ and $E^{uu}_p \subset E^u_p$.
\end{remark}

A partition $\xi$ of $M$ is \textbf{measurable} with respect to a probability measure $\nu$, if up to a set of $\nu$-zero measure, the quotient $M/\xi$ is separated by a countable number of measurable sets. Denote by $\hat{\nu}$ the quotient measure in $M/\xi$.

By Rokhlin's disintegration theorem \cite{ch5rokhlin1}, for a measurable partition $\xi$, there is set of conditional measures $\{\nu_D^{\xi}: D\in \xi\}$ such that for $\hat{\nu}$-almost every $D\in \xi$ the measure $\nu_D^{\xi}$ is a probability measure supported on $D$, for each measurable set $B\subset M$ the application $D \mapsto \nu^{\xi}_D(B)$ is measurable and
\begin{equation}
\label{ob.disintegration}
\nu(B) = \displaystyle \int_{M/\xi} \nu_D^{\xi}(B) d\hat{\nu}(D).
\end{equation}

From now on we suppose that $f$ is a $C^2$-diffeomorphism and $\nu$ has no zero Lyapunov exponents. We call such a measure \textbf{hyperbolic}. We remark that usually the unstable partition $\{W^u(p)\}_{p\in \mathcal{R}}$ is not a measurable partition. We say that a $\nu$-measurable partition $\xi^u$ is \textbf{$u$-subordinated} if for for $\nu$-almost every $p$, the following conditions are satisfied:
\begin{itemize}
\item $\xi^u(p) \subset W^u(p)$;
\item $\xi^u(p)$ contains an open neighborhood of $p$ inside $W^u(p)$. 
\end{itemize}

\begin{definition}[SRB measure]
\label{defi.srb}
A measure $\nu$ is \textbf{SRB} if for any $u$-subordinated measurable partition $\xi^u$, for $\nu$-almost every $p$, the conditional measure $\nu^{u}_{\xi^u(p)}$ is absolutely continuous with respect to the riemannian volume of $W^u(p)$.
\end{definition}

Recall that an invariant probability measure  $\mu$ is ergodic if and only if any $f$-invariant measurable set $\Lambda$ has measure $0$ or $1$.There is a well developed ergodic theory for hyperbolic SRB measures. We now state some results obtained by Ledrappier in \cite{ledrappiersrb}.

\begin{theorem}[\cite{ledrappiersrb}, Corollary $4.10$ and Theorem $5.10$.]
\label{thm.ergodiccomponentsrb}
Let $f$ be a $C^2$-diffeomorphism and $\nu$ a hyperbolic SRB measure. Then there are at most countably many ergodic components of $\nu$, that is, 
\[
\nu = \displaystyle \sum_{i\in \N} c_i \nu_i,
\]
 where $c_i \geq 0$, $\displaystyle \sum_{i\in \N} c_i = 1$, each $\nu_i$ is an $f$-invariant ergodic SRB measure such that if $i\neq j$, and $c_i,c_j>0$ then $\nu_i \neq \nu_j$. Moreover, for each $i\in \N$ such that $c_i>0$, there exists $k_i\in \N$ such that 
\[
\nu_i = \frac{1}{k_i} \displaystyle \sum_{j=1}^{k_i} \nu_{i,j},
\]
where each $\nu_{i,j}$ is an $f^{k_i}$-invariant probability measure, the system $(f^{k_i},\nu_{i,j})$ is Bernoulli and $\nu_{i,j} \neq \nu_{i,l}$ if $j \neq l$. Furthermore, $f$ permutes the measures $\nu_{i,j}$, that is, $f_*(\nu_{i,j}) = \nu_{i,j+1}$ for $j=1, \cdots, k_i-1$ and $f_*(\nu_{i,k_i}) = \nu_{i,1}$, where $f_*(\nu)$ denotes the pushforward of a measure $\nu$ by $f$. 
\end{theorem}

Now given two hyperbolic ergodic measure, $\mu$ and $\nu$, we say that stable manifolds of $\mu$ intersects transversely unstable manifolds of $\nu$ if the following holds: there exist a set $\Lambda^s$ with positive $\mu$-measure and a set $\Lambda^u$ with positive $\nu$-measure, such that for each $p\in \Lambda^s$ and $q\in \Lambda^u$, there exists $n_1,n_2\in \Z$ with
\[
W^s(f^{n_1}(p)) \pitchfork W^u(f^{n_2}(q)) \neq \emptyset.
\]
In this case we write $\mu \pitchfork_{su} \nu$. 
\begin{definition}
\label{defi.homrelatedmeasures}
For $\mu$ and $\nu$ hyperbolic ergodic measures, we say that $\mu$ is \textbf{homoclinically related} with $\nu$, if $\mu \pitchfork_{su} \nu$ and $\nu \pitchfork_{su} \mu$. We write $\mu \sim_{hom} \nu$. 
\end{definition}
In the case that $\mu$ and $\nu$ are ergodic SRB measures, homoclinic relation actually implies that they are the same.
\begin{theorem}
\label{thm.srbthesame}
Let $\mu$ and $\nu $ be two hyperbolic, ergodic SRB measures. If $\mu \sim_{hom}\nu$ then $\mu=\nu$.
\end{theorem}
The proof of Theorem \ref{thm.srbthesame} is a consequence of Hopf's argument adapted to the non-uniformly hyperbolic scenario. This type of argument has been done in many places, see for instance Lemma $3.2$ in \cite{hirayamasumi}.

We remark that all the results stated in this section were stated for $C^2$-diffeomorphisms, but they hold for $C^{1+\alpha}$-diffeomorphisms.

\subsection{$u$-Gibbs measures and the invariance principle}
\label{section.preinvariance}
\subsubsection*{$u$-Gibbs measures}
Let $f$ be a $C^2$-partially hyperbolic diffeomorphism and let $\mu$ be an $f$-invariant measure. We say that a $\mu$-measurable partition $\xi^{uu}$ is subordinated to the foliation $\mathcal{F}^{uu}$, if for $\mu$-almost every $p$, $\xi^{uu}(p) \subset W^{uu}(p)$ and $\xi^{uu}(p)$ contains an open neighborhood of $p$ inside $W^{uu}(p)$. For simplicity, we will write the conditional measure $\mu^{uu}_{\xi^{uu}(p)}$ by $\mu^{uu}_p$.

\begin{definition}[$u$-Gibbs]
\label{defi.ugibbs}
An $f$-invariant measure $\mu$ is \textbf{$u$-Gibbs} if for any $\mu$-measurable partition $\xi^{uu}$ subordinated to $\mathcal{F}^{uu}$, for $\mu$-almost every point $p$, the conditional measure $\mu^{uu}_p$ is absolutely continuous with respect to the Lebesgue measure of $W^{uu}(p)$. We denote the set of $u$-Gibbs measures of $f$ by $Gibbs^u(f)$.
\end{definition}

These measures have an important role in the study of ergodic theory of partially hyperbolic systems. The next lemma states that they capture all possible statistical behavior of Lebesgue almost every point. Recall that for any $p\in M$ and $n\in \N$, we defined 
\[
\mu_n(p)= \displaystyle \frac{1}{n} \sum_{j=0}^{n-1} \delta_{f^j(p)}.
\]

\begin{theorem}[\cite{ch5bonattidiazvianabook}, Theorem $11.16$]
\label{thm.bdvugibbs}
Let $f$ be a $C^2$-partially hyperbolic diffeomorphism, then for Lebesgue almost every point $p\in M$, every accumulation point of the sequence of probability measures $(\mu_n(p))_{n\in \N}$ belongs to $\mathrm{Gibbs}^u(f)$.
\end{theorem}

Let us consider the strong unstable foliation $\mathcal{F}^{uu}$ and $\mu$ an $f$-invariant measure. We say that a $\mu$-measurable partition $\xi^{uu}$ subordinated to $\mathcal{F}^{uu}$ is \textbf{increasing} if for $\mu$-almost every $p$, we have $\xi^{uu}(f(p)) \subset f(\xi^{uu}(p))$. We define the \textbf{$\mu$-partial entropy along $\mathcal{F}^{uu}$} by
\begin{equation}
\label{eq.upartialentropy}
\displaystyle h_{\mu} (f, \mathcal{F}^{uu}) = H_{\mu}(f^{-1}\xi^{uu} | \xi^{uu}) := -\int_M \log \mu^{uu}_p(f^{-1}\xi^{uu}(p)) d\mu(p), 
\end{equation} 
where $f^{-1}\xi^{uu}(p)$ is the element of the partition $f^{-1}\xi^{uu}$ containing $p$. The definition above does not depend on the choice of the $\mu$-measurable partition $\xi^{uu}$. The notion of partial entropy along expanding foliations has been introduced in \cite{vianayang} and \cite{ch5yang2016} (see also \cite{ledrappieryoung1}).

Let $Jac^{uu}(p) = |\det( Df(p)|_{E^{uu}})|$. In the case that $E^{uu}$ has dimension one, for any ergodic $f$-invariant measure, we write $\lambda^{uu}_{\mu}$ to be the Lyapunov exponent of the strong unstable direction. The following result can be found in \cite{ch5yang2016} and \cite{ledrappiersrb}.

\begin{proposition}[\cite{ch5yang2016}, Proposition $5.2$, and \cite{ledrappiersrb}, Theorem $3.4$]
\label{prop.entropyformula}
Let $\mu$ be an $u$-Gibbs measure. Then
\[
h_{\mu}(f,\mathcal{F}^{uu}) = \displaystyle \int_M \log Jac^{uu}(p) d\mu(p).
\]
In particular, if $E^{uu}$ is one dimensional and $\mu$ is ergodic then $h_{\mu}(f,\mathcal{F}^{uu}) = \lambda^{uu}_{\mu}$.
\end{proposition}

\subsubsection*{The invariance principle}
An important tool in this work is the invariance principle which was first developed by Furstenberg in \cite{furstenberg} and by Ledrappier in \cite{ledrappierinvariance}. We also mention the work of Avila-Viana in \cite{ch5avilavianainvariance}. In this work we use the version of the invariance principle given by Tahzibi-Yang in \cite{tahzibiyang}, which we describe in this section. This relates entropy along strong unstable foliations with the so called $u$-invariance of certain measures. Their results hold for large classes of partially hyperbolic skew products, however, we will state them for skew products on $\T^2\times \T^2$. 

Let $f$ be a $C^2$-partially hyperbolic center bunched skew product and let $f_2$ be the Anosov diffeomorphism on the base. We remark that on $\T^2$, every Anosov diffeomorphism is transitive. Fix a $f_2$-invariant measure $\nu$. Let $\mathcal{\xi}^{uu}_2$ be a $\nu$-measurable partition of $\T^2$ which is subordinated to the foliation $\mathcal{F}^{uu}_2$ (the unstable foliation of $f_2$ on $\T^2$), and consider the $\mu$-measurable partition $\xi^{uu}$ of $\T^4$ subordinated to $\mathcal{F}^{uu}$ which refines the partition $\pi^{-1}_2(\xi^{uu}_2)$ with the property that for $\mu$-almost every $p$, $\pi_2(\xi^{uu}(p)) = \xi^{uu}_2(\pi_2(p))$. 

\begin{definition}
\label{def.ugibbstahzibiyang}
We say that an $f$-invariant measure $\mu$ is an \textbf{$u$-state projecting on $\nu$}, if $(\pi_2)_* \mu = \nu$ and for $\mu$-almost every $p$,
\begin{equation}
\label{eq.ugibbsnu}
(\pi_2)_*\mu^{uu}_p = \nu^{uu}_{\pi_2(p)}.
\end{equation}
We denote the set of $u$-state measures projecting on $\nu$ by  $\mathrm{State}^u_{\nu}(f)$. We say that a measure $\mu$ projecting on $\nu$ is an \textbf{$s$-state projecting on $\nu$}, if $\mu \in \mathrm{State}^u_{\nu}(f^{-1})$. We denote the set of $s$-state measures by $\mathrm{State}^s_{\nu}(f)$.
\end{definition}
\begin{remark}
In \cite{tahzibiyang}, the authors call the measures from definition \ref{def.ugibbstahzibiyang} $u$-Gibbs measures projecting on $\nu$. Since we already use the name $u$-Gibbs for the measures from definition \ref{defi.ugibbs}, we changed the name in our paper. Even though later we will see that in our setting both definitions coincide once the measure $\nu$ is an SRB-measure for the Anosov diffeomorphism on the basis (see proposition \ref{prop.ugibbsequgibbs}).
\end{remark}
The following result is a characterization using entropy for a measure to belong to $\mathrm{State}^u_{\nu}(f)$.

\begin{theorem}[\cite{tahzibiyang}, Theorem A]
\label{thm.tahzibiyanginvariance}
Let $f$ be a $C^2$-partially hyperbolic skew product as above and let $\nu$ be an $f_2$-invariant measure. Suppose that $\mu$ is an $f$-invariant measure such that $(\pi_2)_*\mu = \nu$. Then $h_{\mu}(f,\mathcal{F}^{uu}) \leq h_{\nu}(f_2)$ and the equality holds if and only if $\mu \in \mathrm{State}^{u}_{\nu}(f)$.
\end{theorem}

\begin{proposition}[\cite{tahzibiyang}, Proposition $5.4$]
\label{prop.uinvariancemeasure}
A measure $\mu$ is an $u$-state projecting on $\nu$ if and only if there exists a set $X\subset \T^2$ of full $\nu$-measure such that for any two points $p_2, q_2\in X$  in the same unstable leaf, we have that 
\begin{equation}
\label{eq.uinvariancemeasure}
\mu^c_{q_2} = (H^u_{p_2,q_2})_* \mu^c_{p_2}.
\end{equation}
\end{proposition}    	
The property described by (\ref{eq.uinvariancemeasure}) is called \textbf{$u$-invariance of the conditional measures $\{\mu^c_{p_2}\}_{p_2\in \T^2}$}.

Since $f_2$ is a transitive $C^2$-Anosov diffeomorphism, it is well known that it admits an unique SRB measure $\nu$, see \cite{ch5bowenbook, ruelle, ch5sinai68}. Consider now the set $\mathrm{State}^u_{\nu}(f)$. In what follows we will show that $\mathrm{Gibbs}^u(f) = \mathrm{State}^u_{\nu}(f)$.

\begin{proposition}
\label{prop.ugibbsequgibbs}
For $f$ and $\nu$ as above, $\mathrm{Gibbs}^u(f) = \mathrm{State}^u_{\nu}(f)$.
\end{proposition}
To prove this proposition, we will need the following lemma.

\begin{lemma}
\label{lemma.projectugibbs}
Let $\mu \in \mathrm{Gibbs}^u(f)$, then $(\pi_2)_* \mu = \nu$. 
\end{lemma}
 \begin{proof}
It is enough to prove that $\tilde{\nu}:=(\pi_2)_*\mu$ is an SRB measure for $f_2$. Since $f_2$ admits only one SRB measure, it follows that $\tilde{\nu} = \nu$. 

Let $\xi^{uu}_2$ be a $\tilde{\nu}$-measurable partition subordinated to $\mathcal{F}^{uu}_2$. Observe that the partition $\xi^{cu} = \pi_2^{-1}(\xi^{uu})$ is $\mu$-measurable and denote by $\mu^{cu}_p$ the conditional measures of $\mu$ with respect to this partition. The partition $\xi^{cu}$ is refined by the $\mu$-measurable partition $\xi^{uu}$ which is subordinated to $\mathcal{F}^{uu}$ and such that for $\mu$-almost every $p$, we have $\pi_2(\xi^{uu}(p)) = \xi^{uu}_2(\pi_2(p))$.

Take a $\tilde{\nu}$-generic point $p_2\in \T^2$ and let $B\subset \xi^{uu}_2(p_2)$ be a set of zero Lebesgue measure inside the unstable manifold of $p_2$. Since the foliation by center fibers is smooth (because we are in the skew product setting), and the strong unstable manifolds of $f$ are uniformly transverse to the center direction inside the $cu$-leaves, we have that for $\mu^{cu}_{p_2}$-almost every $q$ the set $\xi^{uu}(q) \cap \pi^{-1}_2(B)$ has zero Lebesgue measure inside $W^{uu}_f(q)$. In particular, the $u$-Gibbs property of $\mu$ implies that $\mu^{uu}_q(\pi^{-1}_2(B)) = 0$. We conclude 
\[
\tilde{\nu}^{uu}_{p_2}(B) = \displaystyle \int_M \mu^{uu}_{q}(\pi^{-1}_2(B)) d\mu^{cu}(q) =0.
\]
This is true for any set $B$ of zero Lebesgue measure. This implies that $\tilde{\nu}^{uu}_{p_2}$ is absolutely continuous with respect to the Lebesgue measure of $W^{uu}_{f_2}(p_2)$ and the measure $\tilde{\nu}$ is SRB.

 \end{proof}

\begin{proof}[Proof of Proposition \ref{prop.ugibbsequgibbs}]
From (\ref{eq.ugibbsnu}) and the fact that the foliation by horizontal fiber is smooth, it is immediate that $\mathrm{State}^u_{\nu} (f) \subset \mathrm{Gibbs}^u(f)$. 

Since the strong unstable direction is uniformly transverse to the center fibers inside the $cu$-leaves and it projects to $E^{uu}_{f_2}$, and since the center direction is orthogonal to the base, there exists a constant $C\geq 1$ such that for any $p\in \T^4$ and any $v^{uu} \in E^{uu}_p$ we have
\[
\displaystyle \frac{1}{C} \|v^{uu}\| \leq \|D\pi_2(p) v^{uu}\| \leq \|v^{uu}\|.
\]

Suppose that $\mu \in \mathrm{Gibbs}^u(f)$ is an ergodic measure. Let $p$ be a generic point for $\mu$ and let $v^{uu}\in E^{uu}_p$ be an unit vector. Observe that for any $n\in \N$ we have
\[
\displaystyle \frac{1}{C} \|Df^n(p)v^{uu}\| \leq \|D\pi_2(f^{n}(p)) Df^n(p) v^{uu}\| \leq \|Df^n(p)v^{uu}\|.
\]
Since $f$ is a skew product and $\pi _2\circ f = f_2\circ \pi_2$, we obtain that $D\pi_2(f^{n}(p)) Df^n(p) v^{uu} = Df^n_2(\pi_2(p)) D\pi_2(p) v^{uu}$. By lemma \ref{lemma.projectugibbs}, we may assume that $\pi_2(p)$ is a generic point for $\nu$. We conclude that
\[
\displaystyle \lambda^{uu}_{\mu} = \lim_{n\to +\infty} \frac{1}{n} \log \|Df^n(p)v^{uu}\| = \lim_{n\to +\infty} \frac{1}{n} \log \|Df^n_2(\pi_2(p)) D\pi_2(p) v^{uu}\| = \lambda^{uu}_\nu,
\]
where $\lambda^{uu}_{\mu}$ is the Lyapunov exponent of $f$ for $\mu$ along the strong unstable direction and $\lambda^{uu}_{\nu}$ is the Lyapunov exponent of $f_2$ for $\nu$ along the unstable direction. 

It is well known that the measure $\nu$ verifies Pesin's formula (since it is also an SRB measure for $f_2$), see \cite{ledrappiersrb}, and hence $h_{\nu}(f_2) = \lambda^{uu}_{\nu}$. By proposition \ref{prop.entropyformula}, we have that $h_{\mu}(f,\mathcal{F}^{uu}) = \lambda^{uu}_{\mu}$. We conclude that $h_{\mu}(f,\mathcal{F}^{uu}) = h_{\nu}(f_2)$. By Theorem \ref{thm.tahzibiyanginvariance} we obtain that $\mu \in \mathrm{State}^u_{\nu}(f)$.
\end{proof}

The main conclusion of proposition \ref{prop.ugibbsequgibbs} is the following corollary.

\begin{corollary}
\label{cor.whatiwant}
For $f$ as above, any $u$-Gibbs measure $\mu$ has $u$-invariant center conditional measures. 
\end{corollary}

\section{Proof of Theorem \ref{thm.thmA} assuming Theorems \ref{thm.thmB} and \ref{thm.thmC}}
\label{section.proofthmA}
Fix $N$ large enough and let $\mathcal{U}_N^{sk}$ be a $C^2$-neighborhood of $f_N$ inside $\mathrm{Sk}^2(\T^2\times \T^2)$ small enough such that it verifies the conclusions of Theorems \ref{thm.thmB} and \ref{thm.thmC}. By Theorem \ref{thm.thmC}, any $g\in \mathcal{U}_N^{sk}$ has at most one SRB-measure. By Theorem \ref{thm.thmB}, every $u$-Gibbs measure is either SRB or it is supported on a finite union of two dimensional tori tangent to the strong stable and unstable directions. In particular, in the second, the support of the $u$-Gibbs measure is contained in the finite union of trivial accessibility classes. 

By Theorem \ref{thm.nontrivialaccessclasses}, there exists a subset $\mathcal{V} \subset \mathcal{U}_N^{sk}$ which is $C^2$-open and $C^2$-dense in $\mathcal{U}_N^{sk}$ such that any $g\in \mathcal{V}$ does not have any trivial accessibility class. In particular, for such $g$, there cannot be a two dimensional torus tangent to $E^{ss}_g \oplus E^{uu}_g$. Therefore, as a consequence of Theorems \ref{thm.thmB} and \ref{thm.thmC}, we conclude that for any $g\in \mathcal{V}$, there exists an unique $u$-Gibbs measure $\mu_g$. It is hyperbolic SRB and Bernoulli. Moreover, $\mathrm{supp}(\mu_g) = \T^4$.

Fix $g\in \mathcal{V}$. By Theorem \ref{thm.bdvugibbs}, for Lebesgue almost every point $p$, any accumulation point of the sequence 
\[
\mu_n(p) = \frac{1}{n} \displaystyle \sum_{j=0}^{n-1} \delta_{g^j(p)} 
\]  															 
is an $u$-Gibbs measure. Since there exists only one $u$-Gibbs measure $\mu_g$, we conclude that for Lebesgue almost every point $p$,
\[
\displaystyle \lim_{n\to + \infty} \mu_n(p) = \mu_g.
\]
Therefore, $Leb(B(\mu_g)) = 1$ and we conclude the proof of Theorem \ref{thm.thmA}.

\section{Center Lyapunov exponents for $u$-Gibbs measures}
\label{section.centerlyapunovexpoenents}
In this section we explain how the techniques developed by Berger-Carrasco in \cite{ch5bergercarrasco2014}, and the adaptations of their techniques made by the author in \cite{obataergodicity}, actually give estimates for the Lyapunov exponents for any $u$-Gibbs measure. We prove the following theorem:

\begin{theorem}
\label{thm.estimateexponentsugibbs}
For every $\delta \in (0,1)$, there exists $N_0 = N_0(\delta)$ such that for every $N\geq N_0$, there exists $\mathcal{U}_N$ a $C^2$-neighborhood of $f_N$ inside $\mathrm{Diff}^2(\T^4)$ with the following property. If $g\in \mathcal{U}_N$ and $\mu$ is an $u$-Gibbs measure, then $\mu$-almost every point has a positive and a negative Lyapunov exponent along the center whose absolute value is greater than $(1-\delta) \log N$. 
\end{theorem}

\begin{remark}
Even though the results from \cite{ch5bergercarrasco2014, obataergodicity} are in the volume preserving scenario, several of the lemmas and propositions are still valid for dissipative perturbations. In what follows, we will use several results from these works. The only point in this section that will need an adaptation for $u$-Gibbs measures is given in proposition \ref{ob.estimaterob}.  
\end{remark}

 Let $A\in SL(2,\Z)$ be the linear Anosov matrix considered in the definition of the map $f_N$. Denote by $0<\lambda <1< \tilde{\mu}=\lambda^{-1}$ the eigenvalues of $A$. Let $e^s$ and $e^u$ be unit eigenvectors of $A$ for $\lambda$ and $\tilde{\mu}$, respectively. Recall that we defined the linear map $P_x: \R^2 \to \R^2$ given by $P_x(a,b) = (a,0)$.

\begin{lemma}[\cite{ch5bergercarrasco2014}, Proposition $1$]
\label{proposition1bc}
There is a differentiable function $\alpha: \T^4 \to \R^2$ such that the unstable direction of $f_N$ is generated by the vector field $(\alpha(m), e^u)$, where 
$$
Df_N(m).(\alpha(m), e^u)= \tilde{\mu}^{2N}(\alpha(f_N(m)), e^u) \textrm{ and } \|\alpha(m) - \lambda^N P_x(e^u)\| < \lambda^{2N}.
$$
\end{lemma}

Observe that $|\det Df_N|_{E^c_{f_N}}| = 1$. 

\begin{lemma}[\cite{obataergodicity}, Lemma 7.17]
\label{ob.manyconsiderationsnotfibered}
For $\varepsilon_1>0$ and $\beta>0$ small, if $N$ is large and $\mathcal{U}_N$ is small enough then for every $g\in \mathcal{U}_N$ and for all unit vectors $v^s\in E^{ss}_g$, $v^c\in E^c_g$ and $v^u\in E^{uu}_g$, the following holds:
\begin{enumerate}
\item $e^{-\varepsilon_1} \tilde{\mu}^{2N}  \leq  \|Dg(v^u)\| \leq  e^{\varepsilon_1} \tilde{\mu}^{2N}$;
\item $\frac{1}{2N}  \leq  \|Dg(v^c)\|  \leq  2N;$
\item $\|D^2g^{-1}\|\leq 2N$ and $\|D^2g\|\leq 2N$;
\item $|\det Dg|_{E^c_g}| \in (e^{-\beta}, e^{\beta})$;
\item $E^c_g$ is $\frac{1}{2}$-H\"older.
\end{enumerate}
\end{lemma}

A key element in Berger-Carrasco's proof is to consider center vector fields over certain pieces of strong unstable curve. Consider $g\in \mathcal{U}_N$.  

\begin{definition}[\cite{ch5bergercarrasco2014}, Definition 7.18]
An $u$-curve for $g$ is a $C^1$-curve $\gamma=(\gamma_x,\gamma_y,\gamma_z,\gamma_w):[0,2\pi] \to M$ tangent to $E^{uu}_g$ and such that $\left|\frac{d \gamma_x}{dt}(t)\right| = 1$, $\forall t\in [0,2\pi]$. For every $k\geq 0$ there exists an integer $N_k = N_k(\gamma)\in \left[(e^{-\varepsilon_1}\tilde{\mu}^{2N})^k,(e^{\varepsilon_1}\tilde{\mu}^{2N})^k\right]$ such that the curve $g^k\circ \gamma$ can be writen as 
$$
g^k \circ \gamma = \gamma_1^k \ast\cdots \ast \gamma_{N_k} \ast \gamma_{N_k+1}^k
$$
where $\gamma_j^k$ for $j=1,\cdots, N_k$, are $u$-curves and $\gamma_{N_k+1}^k$ is a segment of $u$-curve. 
\end{definition}

The following lemma controls the length of $u$-curves.

\begin{lemma}[\cite{ch5bergercarrasco2014}, Corollary $5$]
If $N$ is large and $\mathcal{U}_N$ small enough then for every $g\in \mathcal{U}_N$ and any unit vector $v^u\in E^{uu}_{g,m}$, it holds that
\[
|P_x(D\pi_1.v^u)| \in [(\lambda^N(\|P_x(e^u)\| - 3\lambda^N), (\lambda^N(\|P_x(e^u)\| + 3\lambda^N)].
\]
\end{lemma}

An easy consequence of this lemma is the following.

\begin{corollary}
\label{ob.length}
For any $\varepsilon_2>0$, if $N$ is large and $\mathcal{U}_N$ is small enough, then any two $u$-curves $(\gamma, \gamma')$ satisfy:
\begin{equation}
\label{eq.sizeucurves}
e^{-\varepsilon_2} |\gamma| \leq |\gamma^{\prime} | \leq e^{\varepsilon_2}|\gamma|,
\end{equation}
where $|\gamma|$ denotes the length of the curve $\gamma$.
\end{corollary}

We define the unstable jacobian of $g^k$ as
\begin{equation}
\label{eq.unstablejacobian}
J^{uu}_{g^k}(m) = | \det Dg^k(m)|_{E^{uu}_g}|, \textrm{ $\forall m \in \T^4$.}
\end{equation}
By item $1$ of lemma \ref{ob.manyconsiderationsnotfibered}, for $g\in \mathcal{U}_N$ and for every $m\in \T^4$
\begin{equation}
\label{eq.upperboundjacobian}
e^{-\varepsilon_1}\lambda^{2N} \leq J^{uu}_{g^{-1}}(m) \leq e^{\varepsilon_1}\lambda^{2N}.
\end{equation}

\begin{lemma}[\cite{obataergodicity}, Lemma 7.20]
\label{boundeddist}
For $\varepsilon_3>0$ small, if $N$ is large and $\mathcal{U}_N$ is small enough, for every $g\in \mathcal{U}_N$ and any $u$-curve $\gamma$ for $g$, for every $k\geq 0$, we have
\[
\forall m, m^{\prime} \in \gamma, \textrm{  } e^{-\varepsilon_3} \leq \frac{J^{uu}_{g^{-k}}(m)}{J^{uu}_{g^{-k}}(m^{\prime})}\leq e^{\varepsilon_3}.
\] 
\end{lemma}

This lemma implies that for $g\in \mathcal{U}_N$ and for any $u$-curve $\gamma$ for $g$, if $A\subset \gamma$ is any measurable set, for every $k\geq 0$, it holds
\[
e^{-\varepsilon_3} \frac{Leb(A)}{Leb(\gamma)} \leq \frac{Leb(g^{-k}(A))}{Leb(g^{-k}(\gamma))} \leq e^{\varepsilon_3} \frac{Leb(A)}{Leb(\gamma)}.
\]

\begin{definition}
\label{ob.adaptedfield}
An adapted field $(\gamma,X)$ over an $u$-curve $\gamma$ is a unitary vector field $X$ such that
\begin{enumerate}
\item $X$ is tangent to the center direction;
\item $X$ is $(C_X,1/2)$-H\"older along $\gamma$, that is 
$$
\|X_m-X_{m^{\prime}}\| \leq C_X d_{\gamma}(m,m^{\prime})^{\frac{1}{2}},\textrm{ }\forall m, m' \in \gamma,
$$
where $C_X<30N^2\lambda^N$ and $d_{\gamma}$ is the distance measured along $\gamma$. 
\end{enumerate}
\end{definition}
\begin{remark}
The estimate on the H\"older constant used in \cite{ch5bergercarrasco2014, obataergodicity} is $20N^2\lambda^N$, instead of $30 N^2 \lambda^N$ as above. This is due to the fact that the parametrization of the torus $\T^4$ is by intervals of length $2\pi$ instead of $1$ in the proof of lemma $2$ in \cite{ch5bergercarrasco2014}. However, this change on the estimate of the H\"older constant does not affect the rest of the proof. 
\end{remark}

Let $(\gamma,X)$ be an adapted field, and define
\[
I_n^{\gamma,X} = \frac{1}{|\gamma|} \int_{\gamma} \log \|Dg^n.X\| d\gamma.
\]

\begin{proposition}
\label{ob.estimaterob}
Suppose that there exists $C>0$ with the following property: for every $u$-curve $\gamma$ there exists an adapted vector field $(\gamma,X)$ for $g$ and for all $n> 0$ large enough
\[
\frac{I_n^{\gamma,X}}{n} >C.
\]
Then any $u$-Gibbs measure $\mu$ for $g$ has a positive Lyapunov exponent along the center direction greater than $e^{-2\varepsilon_3}C$.
\end{proposition}
 
\begin{proof}
Suppose not, then there exist an $u$-Gibbs measure $\mu$ and a measurable set $B$ with positive $\mu$-measure such that every point in $B$ has exponents in the center direction strictly smaller than $e^{-2\varepsilon_3}C$. Since $\mu$ has disintegration along unstable leaves equivalent to the Lebesgue measure along the leaves, there is an unstable manifold $\gamma$ that intersects $B$ on a set of positive measure for the Lebesgue measure of $\gamma$. Let $b\in \gamma\cap B$ be a density point and take $\gamma_k = g^{-k} \circ \beta_k$, where $\beta_k$ is a $u$-curve with $\beta_k(0) = g^k(b)$. We have that $l(\gamma_k)\to 0$ and by bounded distortion (lemma \ref{boundeddist}) 
\[
\frac{Leb(\gamma_k \cap B )}{Leb(\gamma_k)} \longrightarrow 1.
\]

Take $k$ large enough such that
\[
\frac{Leb(\gamma_k \cap B^c )}{Leb(\gamma_k)}< \frac{e^{-2\varepsilon_3} (e^{\varepsilon_3} - 1) C}{2\log 2N}.
\]
Using bounded distortion again, for any $m^k \in g^k(\gamma_k)$
\[
J^{uu}_{g^{-k}}(m^k) \geq \frac{Leb(\gamma_k)}{Leb(g^k(\gamma_k))} e^{-\varepsilon_3}.
\]

Define $\chi_k(m) = \displaystyle \limsup_{n\to +\infty} \frac{1}{n} \log \|Dg^n(g^k(m)) . X^k_{g^k(m)}\|$ for all $m\in \gamma_k$, where $X^k$ is the vector field such that $(\beta_k,X^k)$ verifies the hypothesis of the proposition. Since for $\mu$-almost every point the Lyapunov exponents exist, using the dominated convergence theorem, we have

\[\arraycolsep=1.2pt\def\arraystretch{2}
\begin{array}{rcl}
\displaystyle \int_{\gamma_k} \chi_k d\gamma_k & = & \displaystyle \int_{\beta_k} \chi_k \circ g^{-k} J^{uu}_{g^{-k}} d\beta_k\\
& \geq & \displaystyle e^{-\varepsilon_3} \frac{Leb(\gamma_k)}{Leb(\beta_k)} \int_{\beta_k} \chi_k \circ g^{-k} d\beta_k\\
& = & \displaystyle e^{-\varepsilon_3} \frac{Leb(\gamma_k)}{Leb(\beta_k)} \limsup_{n\to + \infty} \frac{I_n^{\beta_k, X^k}}{n}.Leb(\beta_k)  \geq e^{-\varepsilon_3} C Leb(\gamma_k).
\end{array}
\]

On the other hand,
\[\arraycolsep=1.2pt\def\arraystretch{2}
\begin{array}{rcl}
\displaystyle \int_{\gamma_k} \chi_k d\gamma_k & = & \displaystyle \int_{\gamma_k \cap B} \chi_k d\gamma_k + \int_{\gamma_k \cap B^c} \chi_k d\gamma_k\\
 & \leq & \displaystyle e^{-2\varepsilon_3} C Leb(\gamma_k) + \frac{\log 2N e^{-2\varepsilon_3}(e^{\varepsilon_3}-1)C Leb(\gamma_k)}{2 \log 2N}\\
 &< & e^{-\varepsilon_3} C Leb(\gamma_k)
\end{array}
\]
which is a contradiction.
\end{proof}

Write
\[
E(\gamma,X) = \frac{1}{|\gamma|} \displaystyle \int_{\gamma} \log \|Dg(m).X_m\| d\gamma(m),
\]
where $(\gamma,X)$ is an adapted field. Let $\pi_1:\T^4 \to \T^2$ be the projection defined by $\pi_1(x,y,z,w) = (x,y)$. For $X$ a vector field on $\gamma$ define 
\[
\widetilde{X}_m = \frac{D\pi_1(X_m)}{\|D\pi_1(X_m)\|}.
\]

In what follows, we let $\tilde{\delta}>0$ be a positive constant that we will fix later.

\begin{definition}
\label{def.deltabadandgood}
Consider the cone $\Delta_{\tilde{\delta}} = \{(u,v) \in \R^2: N^{\tilde{\delta}}|u| \geq |v|\}$. Let $(\gamma, X)$ be an adapted vector field. If for every $m\in \gamma$ we have that $\tilde{X}(m) \in \Delta_{\tilde{\delta}}$ then we say that $(\gamma,X)$ is a {\bf$\tilde{\delta}$-good} adapted vector field. Otherwise we say that it is {\bf $\tilde{\delta}$-bad}.
\end{definition}

Recall that for $k\geq 0$ and an $u$-curve $\gamma$ the number $N_k= N_k(\gamma)$ denotes the maximum number of $u$ curves that subdivide $g^k \circ \gamma$. For an adapted field $(\gamma,X)$ define the unit vector field over $g^k(\gamma)$, $Y^k = \frac{g^k_*X}{\|g^k_*X\|}$, where $g^k_*X_m = Dg^k(g^{-k}(m))X_{g^{-k}(m)}$. 

\begin{lemma}[\cite{ch5bergercarrasco2014}, Lemma $9$]
\label{ob.continuesadapted}
For $N$ large and $\mathcal{U}_N$ small enough, let $g\in \mathcal{U}_N$ and $(\gamma,X)$ be an adapted field for $g$. For $k\geq 0$, every possible pair $(\gamma_j^k,Y^k|_{\gamma_j^k})$, with $1\leq j \leq N_k(\gamma)$ is an adapted field.
\end{lemma}

The following formula is proved in section $6$ of \cite{ch5bergercarrasco2014}.

\begin{lemma}
\label{ob.formula2}
For every adapted field $(\gamma,X)$ and any $n\in \N$
\[
I_n^{\gamma,X} = \displaystyle \sum_{k=0}^{n-1} \left( R_k+ \sum_{j=0}^{N_k} \frac{1}{|\gamma|} \int_{\gamma_j^k} \log \| Dg (m).Y^k_m\| J^{uu}_{g^{-k}} d \gamma_j^k \right),
\]

where $R_k=\frac{1}{|\gamma|} \int_{\gamma_{N_k+1}^k} \log \| Dg (m).Y^k_m\| J^{uu}_{g^{-k}} d \gamma_{N_k+1}^k$. 
\end{lemma}

As a consequence of lemma \ref{ob.formula2}, and using (\ref{eq.sizeucurves}), we obtain

\begin{eqnarray}
\label{ob.aiaiai}
I_n^{\gamma,X} \geq \displaystyle \sum_{k=0}^{n-1} \left( R_k+ e^{-\varepsilon_2}\sum_{j=0}^{N_k}(\min_{\gamma_j^k}  J^{uu}_{g^{-k}}) E( \gamma_j^k, Y^k) \right).
\end{eqnarray}

Since $\gamma^k_{N_k+1}$ is a piece of an $u$-curve, then 
\[
\displaystyle \frac{|\gamma^k_{N_k+1}|}{|\gamma|} < 2.
\]
By (\ref{eq.upperboundjacobian}), we have 
\[
\begin{array}{rcl}
\displaystyle |R_k|= \frac{1}{|\gamma|} \int_{\gamma_{N_k+1}^k} \log \| Dg (m).Y^k_m\| J^{uu}_{g^{-k}} d \gamma_{N_k+1}^k & < &\displaystyle \frac{|\gamma^k_{N_k+1}|}{|\gamma|}(e^{\varepsilon_1} \lambda)^{2Nk} \log 2N\\
&< & \displaystyle 2 (e^{\varepsilon_1} \lambda)^{2Nk} \log 2N \xrightarrow{k\rightarrow + \infty} 0. 
\end{array}
\]
Hence,
\[
\displaystyle \frac{1}{n} \sum_{k=0}^{n-1} |R_k| \longrightarrow 0.
\]

The following is the key proposition that will give us the estimate that we need.

\begin{proposition}[\cite{obataergodicity}, Proposition 7.29]
\label{ob.agoravai}
For $N$ large and $\mathcal{U}_N$ small enough, for every $g\in \mathcal{U}_N$, any $\tilde{\delta}$-good adapted field $(\gamma,X)$ and every $k\geq 0$, we have
\[
e^{-\varepsilon_2}\displaystyle \sum_{j=0}^{N_k} (\min_{\gamma_j^k}  J^{uu}_{g^{-k}}) E( \gamma_j^k, Y^k) \geq (1-12\tilde{\delta}) \log N.
\]
\end{proposition}

\begin{remark}
In \cite{obataergodicity}, the term $e^{-\varepsilon_2}$ on the right hand side of the equation (\ref{ob.aiaiai}) is missing. The same term is also missing in the statement of proposition $7.29$ in \cite{obataergodicity}. Since we can fix $\varepsilon_2$ arbitrarily close to $0$, this does not affect the rest of the proof in \cite{obataergodicity} to obtain the estimate of the center Lyapunov exponents. 
\end{remark}

Now, we can proceed with the proof of Theorem \ref{thm.estimateexponentsugibbs}.

\begin{proof}[Proof of Theorem \ref{thm.estimateexponentsugibbs}]

Take $\tilde{\delta} = \frac{2\delta}{15}$. Let $N$ be large and let $\mathcal{U}_N$ be small enough such that it verifies proposition \ref{ob.agoravai}. Fix $g\in \mathcal{U}_N$ and let $\mu$ be an $u$-Gibbs measure for $g$. Consider any $u$-curve $\gamma$ and any $\tilde{\delta}$-good vector field $X$ on $\gamma$. By proposition \ref{ob.agoravai}, and using inequality (\ref{ob.aiaiai}), for $n$ large enough  
\begin{equation}
\label{eq.estimateIn}
\frac{I^{\gamma,X}_n}{n} \geq (1-14\tilde{\delta})\log N.
\end{equation}
Since we could have chosen $\varepsilon_3>0$ small enough such that $e^{-\varepsilon_3} (1-14\tilde{\delta}) \geq (1-15\tilde{\delta})$ by proposition \ref{ob.estimaterob}, $\mu$-almost every point has a Lyapunov exponent for $g$ in the center direction larger than 
\[
(1-15\tilde{\delta}) \log N = (1-2\delta) \log N.
\]

By condition ($4$) in lemma \ref{ob.manyconsiderationsnotfibered}, we have that for $\mu$-almost every point $m$ the sum of the center Lyapunov exponents belongs to the interval $(-\beta, \beta)$, that is,  $-\beta<\lambda^-(m) + \lambda^+(m) < \beta$. By taking $\beta>0$ small, after fixing $\delta$, we conclude that 
\[
\lambda^-(m) < \beta - \lambda^+(m) < \beta - (1-2\delta) \log N < (1-\delta) \log N.
\] 
Therefore, we obtain that for $N$ large and $\mathcal{U}_N$ small enough, for $g\in \mathcal{U}_N$, any $u$-Gibbs measure $\mu\in \mathrm{Gibbs}^u(g)$ verifies that $\mu$-almost every point $m$ has both a positive and a negative Lyapunov exponent on the center with absolute value larger than $(1-\delta)\log N$.
\end{proof}

\section{Proof of Theorem \ref{thm.thmC}}
\label{section.proofthmC}
Recall that in section \ref{sec.preliminaries} we defined the notion of homoclinically related measures (see definition \ref{defi.homrelatedmeasures}). The goal of this section is to prove Theorem \ref{thm.thmC}. This is based in the techniques developed by the author in \cite{obataergodicity}. We actually prove the following theorem, which is more general than Theorem \ref{thm.thmC}:

\begin{theorem}
\label{thm.homoclinicrelatedugibbs}
For $N$ large and $\mathcal{U}_N$ small enough, for any $k\in \N$ the following holds: if $g\in \mathcal{U}_N$ and $\mu_1,\mu_2$ are two ergodic $u$-Gibbs measures for $g^k$, then $\mu_1$ is homoclinically related to $\mu_2$.
\end{theorem}

 For an SRB measure, we can also obtain the following proposition.

\begin{proposition}
\label{prop.srbfullsupport}
For $N$ large and $\mathcal{U}_N$ small enough, let $g\in \mathcal{U}_N$ and let $\mu$ be an SRB measure for $g$. Then $\mathrm{supp}(\mu) = \T^4$.
\end{proposition}

\begin{proof}[Proof of Theorem \ref{thm.thmC} assuming Theorem \ref{thm.homoclinicrelatedugibbs} and Proposition \ref{prop.srbfullsupport}]
Let $N$ be large and $\mathcal{U}_N$ be small enough such that Theorem \ref{thm.homoclinicrelatedugibbs} holds and fix $g\in \mathcal{U}_N$. If $\mu_1$ and $\mu_2$ are two ergodic SRB measures for $g$, by Theorem \ref{thm.homoclinicrelatedugibbs}, $\mu_1$ is homoclinically related to $\mu_2$. By Theorem \ref{thm.srbthesame}, $\mu_1=\mu_2$, and therefore $g$ has at most one SRB measure.

Suppose that $\mu$ is an SRB measure for $g$. By Theorem \ref{thm.ergodiccomponentsrb}, there exist $k\in \N$ and $k$ measures which are $g^k$-invariant and SRB, $\mu_1, \cdots ,\mu_k$, such that $\mu_i \neq \mu_j$ for $j\neq i$ and
\[
\mu = \frac{1}{k} \displaystyle \sum_{j=1}^k \mu_j.
\]
Moreover, $g_*(\mu_j) = \mu_{j+1}$, with the identification of $k+1 = 1$, and $(g^k, \mu^k)$ is Bernoulli. Observe that if $k=1$, then $\mu$ is Bernoulli for $g$.

 Suppose $k>1$, by Theorem \ref{thm.homoclinicrelatedugibbs}, we have that for any $i, j \in \{1, \cdots, k\}$ with $i\neq j$, the measures $\mu_i$ and $\mu_j$ are homoclinically related. Since these measures are SRB, we obtain that $\mu_i = \mu_j$, which is a contradiction with the fact that $\mu_i \neq \mu_j$. Hence, $k=1$ and the measure $\mu$ is Bernoulli for $g$. Proposition \ref{prop.srbfullsupport} states that if $\mu$ is SRB then it has full support. 
\end{proof}

The rest of this section is mostly dedicated to prove Theorem \ref{thm.homoclinicrelatedugibbs}. As we will see, the proof of this theorem is essentially contained in the proof of the stable ergodicity for the map $f_N$ in \cite{obataergodicity}. We will refer the reader to \cite{obataergodicity} for the proofs of several of the lemmas and propositions that we will use in this section, and we remark that they are also valid outside the volume preserving setting. At the end of the section we explain how to obtain Proposition \ref{prop.srbfullsupport}. The argument involved in the proof of Proposition \ref{prop.srbfullsupport} is a combination of some estimates obtained to prove Theorem \ref{thm.homoclinicrelatedugibbs} and arguments from \cite{carrascoobata}.

\subsection{Estimates for stable and unstable manifolds of $u$-Gibbs measures}
\label{subsection.estimate}

For a vector $v\in T_m\T^4$, write $v_1= D\pi_1(m).v$. For a direction $E \subset T_m \T^4$ we will write $(E)_1 = D\pi_1(m). E$. For this section we fix $0< \delta <<1$ small and we are assuming that $N$ is large and $\mathcal{U}_N$ is small enough such that Theorem \ref{thm.estimateexponentsugibbs} holds. For this subsection we fix two constants (depending on $N$), $\theta_1 := N^{-\frac{2}{5}}$ and $\theta_2: = N^{-\frac{3}{5}}$. 

Let $g\in \mathcal{U}_N$. For each ergodic measure $\mu$ for $g$ let $\Lambda_{\mu}$ be the set of points $m\in \T^4$ such that
\[
\displaystyle \frac{1}{n} \sum_{j=0}^{n-1} \delta_{f^j(m)} \xrightarrow[n\to + \infty]{} \mu \textrm{ and } \frac{1}{n} \sum_{j=0}^{n-1} \delta_{f^{-j}(m)} \xrightarrow[n \to +\infty]{} \mu \textrm{, in the $weak^*$-topology.}
\]
Where $\delta_p$ is the dirac mass on the point $p$. Birkhoff's theorem implies that $\mu(\Lambda_{\mu}) = 1$. Recall that $\mathcal{R}_g$ is the set of regular points for $g$. By Theorem \ref{thm.estimateexponentsugibbs}, if $\mu$ is an $u$-Gibbs measure for $g$, then for each $m\in \mathcal{R}_g \cap \Lambda_{\mu}$ there are two directions $E^-_{g,m}$ and $E^+_{g,m}$ contained in $E^c_{g,m}$, which are the Oseledets' directions with respect to the negative and positive Lyapunov exponent, respectively. 

For each $\mu\in \mathrm{Gibbs}^u(g)$, we define the sets
\[
\arraycolsep=1.2pt\def\arraystretch{2}
\begin{array}{rcl}
Z_{\mu}^- & = & \left \{ m\in \mathcal{R}_g \cap \Lambda_{\mu}: \forall n\geq 0 \textrm{ it holds } \displaystyle \left \Vert Dg^n(m)|_{E^{-}_{g,m}}\right \Vert  < \left(N^{-\frac{4}{5}} \right)^n\right \};\\
Z_{\mu}^+ & = & \left\{ m\in \mathcal{R}_g \cap \Lambda_{\mu}: \forall n\geq 0 \textrm{ it holds }\displaystyle\left \Vert Dg^{-n}(m)|_{E^{+}_{g,m}}\right \Vert < \left( N^{-\frac{4}{5}} \right)^n\right\};\\
Z_{\mu} & = & g(Z_{\mu}^-) \cap g^{-1}(Z_{\mu}^+);\\
Z_g &=& \displaystyle \bigcup_{\mu \in \mathrm{Gibbs}^u(g)} Z_{\mu}.
\end{array}
\]

The proof of the following lemma is the same as lemma $5.2$ in \cite{obataergodicity}. It is an application of Pliss lemma.

\begin{lemma}[\cite{obataergodicity}, lemma $5.2$]
\label{ob.stablefirstlemma}
Let $g\in \mathcal{U}_N$. If $\mu$ is an ergodic $u$-Gibbs measure for $g$, then $ \mu(Z_{g}) \geq \frac{1-7\delta}{1+7\delta}$.
\end{lemma}
Let $T=\left[ \frac{1+7\delta}{28\delta}\right]$ and define 
\begin{equation}
\label{ob.Xg}
X_g=  \displaystyle \bigcap_{k=-T+1}^{T-1} g^k(Z_g).
\end{equation}
An easy consequence of the estimate in lemma \ref{ob.stablefirstlemma} is given in the following lemma.

\begin{lemma}[\cite{obataergodicity}, lemma $5.3$]
\label{ob.measurerob}
For $N$ large and $\mathcal{U}_N$ small enough, if $\mu$ is an $u$-Gibbs measure for $g$ then 
$$
\mu(X_g) >0.
$$
\end{lemma}

For a vector $v\in \R^2$ we write $v= (v_h, v_v)$, where $v_h$ and $v_v$ are the coordinates of $v$ with respect to the basis $(1,0)$ and $(0,1)$. For each $\theta>0$ we consider the horizontal and vertical cones
\[
\C^{hor}_{\theta} = \{v\in \R^2: \theta \|v_h\| \geq \|v_v\|\} \textrm{ and } \C^{ver}_{\theta} = \{v\in \R^2: \theta \|v_v\| \geq \|v_h\|\}.
\]

One of the key ingredients in the proof of stable ergodicity of $f_N$ is based in a version of the stable manifold theorem given by Crovisier-Pujals in \cite{ch5crovisierpujals2016}. Using their construction we can obtain precise estimates on the sizes of stable and unstable manifolds inside the center direction for $u$-Gibbs measures. This is given in the following proposition. Fix $\theta_1 = N^{-\frac{2}{5}}$.

\begin{proposition}[\cite{obataergodicity}, Proposition $5.6$]
	\label{ob.stablesizenotfibered}
	Let $N$ be large and $\mathcal{U}_N$ be small enough. For $g\in \mathcal{U}_N$ and $m\in Z_g$, there are two $C^1$-curves, $W^*_g(m)$, contained in $W^c_g(m)$, tangent to $E^*_{g,m}$ and with length bounded from below by $r_0=N^{-7}$, for $*=-,+$. Those curves are $C^1$-stable and unstable manifolds for $g$, respectively. Moreover, $\left( T_pW^+_{g,r_0}(m) \right)_1\subset \mathscr{C}^{hor}_{\frac{4}{\theta_1}}(p)$ and  $\left(T_qW^-_{g,r_0}(m)\right)_1 \subset \mathscr{C}^{ver}_{\frac{4}{\theta_1}}(q)$, for every $p\in W^+_{g,r_0}(m)$ and $q\in W^-_{g,r_0}(m)$.
\end{proposition}
We remark that the proof of this proposition only uses the estimates for points in the set $Z_g$ and estimates on the $C^2$-norm of $g$. The proof is exactly the same as the proof of proposition $5.6$ in \cite{obataergodicity}

 Let $\theta_2 = N^{-\frac{3}{5}}$. Proposition \ref{ob.stablesizenotfibered} is one of the key ingredients to prove the following lemma.

\begin{lemma}[\cite{obataergodicity}, Lemma $5.7$]
\label{ob.notfiberedbigmanifold}
For $N$ large, $\mathcal{U}_N$ small and $n>15$, let $g\in \mathcal{U}_N$. Then for every $m\in X_g$ there are two curves $\gamma_{g,-n}^-(m) \subset g^{-n}(W^-_{g,r_0}(m))$ and $\gamma_{g,n}^+(m) \subset g^n(W^+_{g,r_0}(m))$ with length greater than $4\pi$, such that $\left( T \gamma^-_{g,-n}(m)\right)_1 \subset \C^{ver}_{\theta_2}$ and $\left( T\gamma^+_{g,n}(m)\right)_1 \subset \C^{hor}_{\theta_2}$.

\end{lemma}

We remark that the statement of lemma $5.7$ from \cite{obataergodicity}, which is the equivalent of lemma \ref{ob.notfiberedbigmanifold} above, involves a measure $\nu_{g,i}$. However, the proof only uses the estimates of the points in the set $Z_g$ and the definition of $X_g$. Therefore, the proof of lemma \ref{ob.notfiberedbigmanifold} is exactly the same as the proof of lemma $5.7$ from \cite{obataergodicity}

For $R>0$, let
\[
\displaystyle W^s_{g,R,-n}(m) = \bigcup_{q\in \gamma_{g,-n}^-(m)} W_{g,R}^{ss}(q),
\]
where the curve $\gamma_{g,-n}^-(m)$ is the curve given by the previous lemma. Define similarly the set $W^u_{g,R,n}(m)$, but using the strong unstable manifolds.  

Let $m\in X_g$ be a typical point for an $u$-Gibbs measure $\mu$. Recall that the stable Pesin manifold is a $C^1$-immersed submanifold and it has a topological characterization given by  
$$
W^s(m) = \{ y\in \T^4 : \displaystyle \limsup_{n\to +\infty} \frac{1}{n} \log d(f^n(m), f^n(y)) <0\}.
$$

By the topological characterization of the stable Pesin manifold and by the definition of $W^s_{g,R,-n}(m)$, it is easy to see that $W^s_{g,R,-n}(m) \subset g^{-n}(W^s(m))$. Observe that the strong stable manifolds subfoliate the Pesin stable manifold, in particular $W^s_{g,R,-n}(m)$ is open inside the Pesin manifold $g^{-n}(W^s(m))$. We conclude that $W^s_{g,R,-n}(m)$ is a $C^1$-submanifold. An analogous conclusion holds for unstable manifolds.

The next lemma allows us to control the tangent space of these stable and unstable manifolds inside the center direction.

\begin{lemma}[\cite{obataergodicity}, Lemma $5.8$]
\label{ob.transversalnotfibered}
Fix $\theta_3>0$ such that $\theta_3 > \theta_2$ and satisfies $\C^{hor}_{\theta_3} \cap \C_{\theta_3}^{ver} = \{0\}$. For $g\in \mathcal{U}_N$, there exists $0<R<1$ such that if $n\geq 15$, $m \in X_g$ and $m^- \in W^s_{g,R,-n}(m) \subset W^s_{g,2,-n}(m)$, then
\[
\left( T(W^s_{g,2,-n}(m) \cap W_g^c(m^-))\right)_1 \subset \C^{ver}_{\theta_3}.
\]
A similar result holds for $W^u_{g,R,n}(m)$.
\end{lemma}

The last ingredient for the proof we will need is the following proposition.

\begin{proposition}
\label{ob.brinarg}
For $N$ large and $\mathcal{U}_N$ small enough, if $g\in \mathcal{U}_N$ then for any ergodic $u$-Gibbs measure $\mu$ for $g$ and for any $k\in \N$, the following property holds: for $\mu$-almost every point $m \in \T^4$,  the sets $\{W^c_g(g^{nk}(m)): n\in \mathbb{N}\}$ and $\{W^c_g(g^{-nk}(m)): n\in \mathbb{N}\}$ are both dense  in $\T^4$. 
\end{proposition}

The proof of this proposition is essentially the same as the proof of Proposition $5.9$ in \cite{obataergodicity}. For the sake of completeness we will include it here.

\begin{proof}

For $\mathcal{U}_N$ small enough, for every $g\in \mathcal{U}_N$ there is a homeomorphism $h_g:\T^4 \to \T^4$, that takes center leaves of $f_N$ to center leaves of $g$, such that for every $m\in \T^4$ it is satisfied 
\[
g \circ h_g(W^c_f(m)) = h_g \circ f(W^c_f(m))
\]
Consider the quotients $M_f = \T^4/\sim^c_f $ and $M_g = \T^4/\sim^c_g$, where $p \sim^c_*q$ if and only if $q\in W^c_*(p)$ for $* = f, g$. We denote $\pi_f: \T^4 \to M_f$ and $\pi_g:\T^4 \to M_g$ the respective projections. Observe that $M_f = \T^2$ and that the induced dynamics $\tilde{f}:M_f \to M_f$ of $f$ is given by $A^{2N}$. Endow $M_g$ with the distance $d_g$ given by the Hausdorff distance on the center leaves, that is, 
\[
d_g(L,W) = d_{\textrm{Haus}}(\pi_g^{-1}(L), \pi_g^{-1}(W)).
\]

By the leaf conjugacy equation, the induced dynamics $\tilde{g}:M_g \to M_g$ of $g$ is conjugated to the linear Anosov $A^{2N}$ on $\T^2$ by the homeomorphism induced by $h_g$, which we will denote by $\tilde{h}_g$. Denote by $W^{uu}_{A^{2N}}(.)$ the unstable manifold of $A^{2N}$ on $\T^2$ and let 
\[
W^{uu}_{\tilde{g}}(L)= \{W \in M_g: \displaystyle \lim _{n\to + \infty} d_g(\tilde{g}^{-n}(L), \tilde{g}^{-n}(W)) = 0\}, 
\] 
be the unstable set of $L$. 

\begin{claim}[Claim 2 in the proof of Proposition $5.9$ from \cite{obataergodicity}]

For every $m\in \T^4$, for every $q\in W^c_g(m)$, it is satisfied that 
\[
\pi_g(W^{uu}_g(q)) = W^{uu}_{\tilde{g}}(\pi_g(m)) = \tilde{h}_g(W^{uu}_{A^{2N}}(\pi_f(h_g^{-1}(m)))),
\]
and $\pi_g$ is a bijection from $W^{uu}_g(q)$ to $W^{uu}_{\tilde{g}}(\pi_g(m))$.

\end{claim}

For the linear Anosov $A^{2N}$ the unstable foliation is minimal, that is, every unstable manifold of $A^{2N}$ is dense in $\T^2$. Let $\mu$ be an ergodic $u$-Gibbs measure for $g$ and fix $m$ a generic point for $\mu$. Using the minimality of the unstable foliation of the linear Anosov and by the leaf conjugacy $W^{uu}_{\tilde{g}}(\pi_g(m))$ is dense in $M_g$. 

Take $U$ a small open set in $M_g$. Since the center foliation is uniformly compact, $\hat{U} =\pi_g^{-1}(U)$ is a saturated open set such that any two center leaves in $\hat{U}$ are close to each other. By the previous claim $W^{uu}_g(m) \cap \hat{U} \neq \emptyset$.

Since $\mu$ is an $u$-Gibbs measure, we have that $W^{uu}_g(m)$ is contained in the support of $\mu$. Hence, $\mathrm{supp}(\mu) \cap \hat{U} \neq \emptyset$. In particular, $\mu(\hat{U})>0$. Recall that $m$ is a generic point for $\mu$, therefore, its future  and past orbits visit $\hat{U}$ infinitely many times. This is true for any open set $U$ inside $M_g$, which concludes the proof of the proposition for $k=1$.

For $k\in \N$, we remark that an unstable leaf for $A^{2N}$ is an unstable leaf for $A^{2Nk}$, in particular, the unstable foliation of $A^{2Nk}$ is minimal. The map $g^k$ is leaf conjugated to $A^{2Nk}$. The same argument as above concludes the proof of the proposition.
\end{proof}

\begin{proof}[Proof of Theorem \ref{thm.homoclinicrelatedugibbs}]
Let $N$ be large and $\mathcal{U}_N$ be small enough such that lemmas \ref{ob.notfiberedbigmanifold}, \ref{ob.transversalnotfibered} and proposition \ref{ob.brinarg} hold. Fix $g\in \mathcal{U}_N$ and $\mu_1, \mu_2$ be two ergodic $u$-Gibbs measures for $g$.

Recall that we defined the set $X_g$ in (\ref{ob.Xg}) and let $\Lambda_{\mu_i}$ be the set of typical points that we defined before for the measures $\mu_i$, for $i=1,2$. Since $\mu_i(X_g)>0$ and $\mu_i(\Lambda_{\mu_i})=1$, the set $X_i = X_g\cap \Lambda_i$ has positive $\mu_i$ measure as well, for $i=1,2$.

For any two points $m_1 \in X_1$ and $m_2 \in X_2$, we will prove that the stable manifold of $m_1$ has a transverse intersection with the unstable manifold of $m_2$. Fix a center leaf $W^c_g(q)$, the center leaf of some point $q\in \T^4$. By proposition \ref{ob.brinarg}, the forward and past iterates of $W^c_g(m_i)$ are dense in $\T^4$, for $i=1,2$.  In particular, we can find two sequences, $n_k \to +\infty$ and $l_j \to +\infty$, such that
\[
\displaystyle \lim_{k \to +\infty} \pi_g(g^{n_k}(m_1)) = \lim_{j \to +\infty} \pi_g( g^{-lj}(m_2)) = \pi_g(q),
\]
where $\pi_g:M \to M_g$ is the projection of $M$ to $M_g = \mathbb{T}^4/\sim^c_g$, as it was introduced in the proof of proposition \ref{ob.brinarg}. Since the center foliation is continuous, with uniformly compact leaves, we obtain that
\[
g^{n_k}(W^c_g(m_1))  = W^c_g(g^{n_k}(m_1)) \to W^c_g(q) \textrm{ and } g^{-l_j}(W_g^c(m_2)) = W^c_g(g^{-l_j}(m_2)) \to W^c(q),
\] 
where the convergence is in the $C^1$-topology, recall that the center foliation is uniformly compact. 

By lemma \ref{ob.notfiberedbigmanifold}, there are curves $\gamma_{g,n_k}^+(m_1)$ and $\gamma_{g,-l_j}^-(m_2)$ with length bigger that $4\pi$ and contained in the cone $\C^{hor}_{\theta_2}$ and $\C^{ver}_{\theta_2}$, respectively. Take $R$ given by lemma \ref{ob.transversalnotfibered} and consider the sets

\begin{eqnarray*}
L^u_k(m_1) = \displaystyle \bigcup_{z\in \gamma^+_{g,n_k}(m_1)} W^{uu}_{g,R}(z)\subset W^u(g^{n_k}(m_1)) \\
L^s_j(m_2) = \displaystyle \bigcup_{z\in \gamma^-_{g,-l_j}(m_2)} W^{ss}_{g,R}(z) \subset W^s(g^{-l_j}(m_2)).
\end{eqnarray*}

For $k$ and $j$ large enough, $g^{n_k}(W^c_g(m_1))$ and $g^{-l_j}(W^c_g(m_2))$ are very close to the leaf $W^c_g(q)$. Thus by the control on the angles that we obtained in lemma \ref{ob.transversalnotfibered}, there is a transverse intersection between $L^u_k(m_1)$ and $ L^s_j(m_2)$. In particular, $W^u_g(g^{n_k}(m_1))$ and $W^s_g(g^{-l_j}(m_2))$ intersect transversely. Since transverse intersections are invariant by iterates, we conclude that $W^u_g(m_1)$ and $W^s_g(m_2)$ have a transverse intersection.

Repeating this argument, exchanging the roles of $m_1$ and $m_2$, implies that $W^u_g(m_2)$ and $W^s_g(m_1)$ have a transverse intersection. Since the set $X_i$ has positive $\mu_i$ measure, for $i=1,2$, we conclude that $\mu_1$ is homoclinically related to $\mu_2$. This finishes the proof of Theorem \ref{thm.homoclinicrelatedugibbs} for $g$, in the case $k=1$.

Let $k\in \N$. Following the same steps as above, it is easy to prove that any two ergodic $u$-Gibbs measures for $g^k$, $\mu_1$ and $\mu_2$, are homoclinically related. 
\end{proof}

\subsection{Proof of proposition \ref{prop.srbfullsupport}}

We will need a few results from \cite{carrascoobata}. 
\begin{lemma}[\cite{carrascoobata}, Lemma $3.2$]
\label{lem.intersectioncsu}
There exists a constant $R>0$ with the following property: for $N$ sufficiently large, there exists a $C^1$-neighborhood $\mathcal{U}$ of $f_N$ such that for any $g\in \mathcal{U}_N$ and any two points $p,q \in \T^4$ we have that for any $m_p\in W^c_g(p)$ there exists $m_q\in W^c_g(q)$ such that $W^{uu}_{g,R}(m_p) \cap W^{ss}_{g,R}(m_q) \neq \emptyset$. 
\end{lemma}

Fix $\theta = N^{-\frac{3}{5}}$ and recall that in subsection \ref{subsection.estimate}, we defined the vertical cone $\C^{ver}_{\theta}$.  

\begin{lemma}[\cite{carrascoobata}, Proposition $3.3$]
\label{lem.goodcurves}
If $N$ is sufficiently large there exists $\mathcal{U}_N\subset \mathrm{Diff}^2(\T^4)$ a $C^1$-neighborhood of $f_N$ such that for any $g\in \mathcal{U}_N$ and any open set $U\subset \T^4$, there exists $n_s\geq 0$ such that for any $n\geq n_s$, there exists a $C^1$ curve $\gamma^-_n \subset g^{-n}(U)$ satisfying:
\begin{itemize}
	\item  $\gamma^-_n$ is contained in a center leaf.
	\item  $\pi_1(\gamma^-_n)$ is tangent to $\C^{ver}_{\theta}$.
	\item  $\gamma^-_n$ has length greater than $4\pi$
	\item  $\displaystyle \bigcup_{q\in \gamma^-_n} W^{ss}_{g,R}(q) \subset g^{-n}(U).$ 
\end{itemize}
\end{lemma}

Consider the vertical foliation $\mathcal{F}_{ver}=\{\{z\}\times\T^2:z\in\T^2\}$. Observe that for any diffeomorphism $g$ sufficiently $C^1$-close to $f_N$, we have that $W^c_g(m)$ intersects each vertical torus $\{z\}\times\T^2$ in exactly one point, for any $m\in \T^4$. Hence, for any two points $m_1, m_2\in \T^4$, the map from $W^c_g(m_1)$ to $W^c_g(m_2)$ defined by $h^g_{m_1,m_2}(p) = W^c_g(m_2) \cap \mathcal{F}_{ver}(p)$ is well defined. Note that, after identifying all the horizontal tori with $\T^2$, the map $h^{f_N}_{m_1,m_2}$ is just the identity, independently of the points $m_1, m_2$.

\begin{lemma}[\cite{carrascoobata}, Lemma $3.4$]
\label{lem.c0holonomycontrol}
For every $\varepsilon>0$, there exists $N_0:= N_0(\varepsilon)$ with the following property: for $N\geq N_0$ there exists a $C^1$-neighborhood $\mathcal{U}_N$ of $f_N$ such that if $g\in \mathcal{U}_N, p\in \T^4$ and $q\in W^{ss}_{g,R}(p)$ then $d_{C^0}(h^g_{p,q}, H^{s}_{p,q})< \varepsilon$. Analogous result holds for the unstable holonomy. 
\end{lemma}

\begin{proof}[Proof of proposition \ref{prop.srbfullsupport}]
Let $N$ be large and $\mathcal{U}_N$ be small enough such that lemmas \ref{ob.notfiberedbigmanifold}, \ref{lem.intersectioncsu}, \ref{lem.goodcurves}, and \ref{lem.c0holonomycontrol} hold. Let $g\in \mathcal{U}_N$ and suppose that $\mu$ is an SRB measure for $g$. Fix $U\subset \T^4$ an open set, we must prove that $\mathrm{supp}(\mu) \cap U \neq \emptyset$.

Since $\mu$ is SRB, its supports contains entire Pesin unstable manifolds. By lemma \ref{ob.notfiberedbigmanifold}, we can take a $\mu$-generic point $m_u$ with the property that for $n_u$ large enough there exists  $\gamma_{g,n_u}^+(m) \subset g^{n_u}(W^+_{g,r_0}(m))$ a curve of length greater than $4\pi$ and whose projection by $\pi_1$ is tangent to $\C^{hor}_{\theta}$. 

For $n_s$ large enough, let $\gamma_{n_s}^-$ be the curve given by lemma \ref{lem.goodcurves} for $U$ and $g$. As a consequence of lemmas \ref{lem.intersectioncsu}, and \ref{lem.c0holonomycontrol}, we conclude that 
\begin{equation}
\label{eq.referenceintersection}
\displaystyle \left(\bigcup_{q\in \gamma^-_{n_s}} W^{ss}_{g,R}(q)\right) \cap \left( \bigcup_{p\in \gamma^+_{n_u}} W^{uu}_{g,R}(p)\right) \neq \emptyset.
\end{equation}
We refer the reader to \cite{carrascoobata} for more details on this argument. By (\ref{eq.referenceintersection}), we obtain that $g^{n_s + n_u}(W^u(m_u)) \cap U \neq \emptyset $, and since $\mu$ is SRB we conclude that $\mathrm{supp}(\mu) \cap U \neq \emptyset$. \qedhere

\end{proof}

\section{Rigidity of $u$-Gibbs measures}
\label{section.rigidity}

The main tool to study the existence of SRB measures that we will use is a recent result by Brown-Rodriguez Hertz on measure rigidity for random dynamics of surface diffeomorphisms. The goal of this section is to explain the statement of their result and how it can be applied to our scenario after a measurable change of coordinates using the unstable holonomies (see Theorem \ref{thm.rigidity1u}). 

\subsection{Measure rigidity for general skew products}\label{sec.measureskew}

Let $(\Omega, \mathcal{B}_{\Omega}, \nu)$ be a Polish probability space, that is, $\Omega$ has the topology of a complete separable metric space, $\mathcal{B}_{\Omega}$ is the Borel $\sigma$-algebra of $\Omega$ and $\nu$ is a Borel probability measure on $\Omega$. Let $\theta:(\Omega, \mathcal{B}_{\Omega}, \nu) \to (\Omega, \mathcal{B}_{\Omega}, \nu)$ be an invertible, measure-preserving and ergodic transformation. Let $S$ be a compact smooth surface and $\mathrm{Diff}^2(S)$ be the set of $C^2$-diffeomorphisms of $S$. We consider a measurable map that for each point $\xi\in \Omega$ associates a diffeomorphism $f_{\xi} \in \mathrm{Diff}^2(S)$. For each $n\in \mathbb{Z}$ we define
\[
\begin{array}{l}
f^0_{\xi} := Id,\\
f^n_{\xi} := f_{\theta^{n-1}(\xi)} \circ \cdots \circ f_{\xi} \textrm{ for } n>0,\\
f^n_{\xi}: = (f_{\theta^n(\xi)})^{-1} \circ \cdots \circ (f_{\theta^{-1}(\xi)})^{-1} \textrm{ for } n<0. 
\end{array}
\]  

We consider the skew product over $\theta$ given by the map $\xi \mapsto f_{\xi}$, which is defined by
\[
\begin{array}{rcl}
F:S \times \Omega & \longrightarrow & S\times \Omega \\
(x,\xi) & \mapsto & ( f_{\xi}(x),\theta(\xi)).
\end{array}
\]
With the notation above, we may write $F^n(x,\xi) = (f^n_{\xi}(x),\theta^n(\xi))$. Write $X = S\times \Omega$ and let $\pi_2: X \to \Omega$ be the natural projection on $\Omega$.  

Let $\mu$ be an $F$-ergodic probability measure, such that $(\pi_2)_*\mu = \nu$. Observe that the partition by the fibers $S$ is measurable. Therefore, we have a family of conditional measures defined in a set $D$ of full $\nu$-measure $\{\mu_{\xi}\}_{\xi\in D}$ with respect to the partition induced by $\pi_2$. For $\nu$-almost every $\xi$, the measure $\mu_\xi$ is supported on $S_{\xi}:= S \times \{\xi\} $. There is a trivial identification of $S_{\xi}$ with $S$, hence, by an abuse of notation we consider the map $\xi \mapsto \mu_{\xi}$ to be a $\nu$-measurable map from $\Omega$ to the space of Borel probability measures of $S$. 

To talk about SRB measures in this setting, we need to first talk about Lyapunov exponents and stable and unstable manifolds. Write $TX := TS \times \Omega$ and let $DF: TX \to TX$ to be the linear cocycle defined by 
\[
DF((x,v),\xi) = ((f_{\xi}(x), Df_{\xi}(x)v), \theta(\xi)).
\]

Suppose that the following integrability condition holds
\begin{equation}
\label{eq.conditionc2}
\displaystyle \int_{\Omega} \log^+(\|f_{\xi}\|_{C^2}) + \log^+ (\|f^{-1}_{\xi}\|_{C^2}) d\nu(\xi) < \infty,
\end{equation}
where $\log^+(.) = \max \{0, \log(.)\}$ and $\|f_{\xi}\|_{C^2}$ is the $C^2$-norm of $f_{\xi}$. Applying Oseledec's theorem for the linear cocycle $DF$, there is a $\mu$-measurable decomposition $T_{(\xi, x)}X = \bigoplus_{j} E^j_{(x,\xi)}$ such that the space $E^j_{(x,\xi)}$ is the space corresponding to the Lyapunov exponent $\lambda^j_{\mu}$, where $\{\lambda^j_{\mu}\}_j$ are the Lyapunov exponents of $DF$.

From now on, let us suppose that the measure $\mu$ is hyperbolic on the fibers, meaning, all the Lyapunov exponents are non zero. The integrability condition (\ref{eq.conditionc2}) is used to have Pesin's theory for fibered systems. In particular, for $\mu$-almost every point there exists stable and unstable manifolds, which may possibly be just points in the case that all the exponents are negative or positive. We refer the reader to section 6 in \cite{brownhertz} for more details.

Suppose that $\mu$ has at least one positive Lyapunov exponent. The family of unstable manifolds $\{W^u(x,\xi)\}_{(x,\xi) \in X}$ forms a partition of a $\mu$-full measure subset of $X$. Usually this partition is not measurable. In this context, we say that a measurable partition $\mathcal{P}$ is \textbf{$u$-subordinated} if for $\mu$-almost every $(x,\xi)$, there exists a positive number $r>0$ such that $W^u_r(x,\xi) \subset \mathcal{P}(x,\xi) \subset W^u(x,\xi)$. 

\begin{definition}[Fiber-wise SRB]
\label{defi.fiberwisesrb}
An $F$-invariant probability measure $\mu$ is \textit{fiber-wise SRB} if for any $u$-subordinated measurable partition $\mathcal{P}$, for $\mu$-almost every $(x,\xi)$, the conditional measure $\mu_{(x,\xi)}^{\mathcal{P}}$ is absolutely continuous with respect to the riemannian volume on $W^u(x,\xi)$.
\end{definition}

Let $\mathcal{P}_{\Omega}$ be a measurable partition of $\Omega$. We say that $\mathcal{P}_{\Omega}$ is \textbf{increasing} if for $\nu$-almost every point $\xi$ we have
\[
\mathcal{P}_{\Omega}(\theta(\xi)) \subset \theta(\mathcal{P}_{\Omega}(\xi)).
\]
Let $\hat{\mathcal{F}}(\mathcal{P}_{\Omega}) \subset \mathcal{B}_{\Omega}$ be the sub-$\sigma$-algebra generated by $\mathcal{P}_{\Omega}$. We say that $\hat{\mathcal{F}}(\mathcal{P}_{\Omega})$ is an \textbf{increasing} sub-$\sigma$-algebra. We remark that in \cite{brownhertz}, the authors call these partitions and sub-$\sigma$-algebra decreasing instead of increasing. We changed it here to be in harmony with the notion of increasing that we defined in section \ref{sec.preliminaries}.

Let $\mathcal{F}(\mathcal{P}_{\Omega})$ be the $\mu$-completion of $\mathcal{B}_S \otimes \hat{\mathcal{F}}(\mathcal{P}_{\Omega})$, where $\mathcal{B}_{S}$ is the Borel $\sigma$-algebra on $S$. For a hyperbolic measure $\mu$, we may also look at the Oseledec's direction $E^s(x,\xi)$ as a measurable map of $X$ that takes values on the projectivization of $TX$. We are now ready to state the main theorem in \cite{brownhertz}.

\begin{theorem}[\cite{brownhertz}, Theorem $4.10$]
\label{thm.rigiditybrownhertz}
Let $F:X\to X$ be as above verifying the integrability condition (\ref{eq.conditionc2}), let $\mathcal{P}_{\Omega}$ be a measurable increasing partition of $\Omega$ and let $\mu$ be a hyperbolic $F$-invariant measure such that $(\pi_2)_*\mu=\nu$. Suppose that the family of conditional measures on the fibers $\{\mu_{\xi}\}$ are non-atomic almost surely. Furthermore, assume that
\begin{enumerate}
\item $\xi \mapsto f_{\xi}^{-1}$ is $\hat{\mathcal{F}}(\mathcal{P}_{\Omega})$-measurable, and
\item $\xi \mapsto \mu_{\xi}$ is $\hat{\mathcal{F}}(\mathcal{P}_{\Omega})$-measurable.
\end{enumerate} 
Then either $(x,\xi) \mapsto E^s(x,\xi)$ is $\mathcal{F}(\mathcal{P}_{\Omega})$-measurable of $\mu$ is fiber-wise SRB. 
\end{theorem}

\subsection{Change of coordinates}\label{subsec.changecoordinates}

Fix $\alpha \in (0,1)$. In this section, we show how to use Theorem \ref{thm.rigiditybrownhertz} to obtain the following theorem:

\begin{theorem}
\label{thm.rigidity1u}
For $N$ large enough, there exists $\mathcal{U}_N^{sk}$ a $C^2$-neighborhood of $f_N$ in $\mathrm{Sk}^2(\T^2\times \T^2)$ such that for $g\in \mathcal{U}_N^{sk} \cap \mathrm{Diff}^{2+\alpha}(\T^4)$, for any ergodic $\mu \in \mathrm{Gibbs}^u(g)$ one of the following holds:
\begin{enumerate}
\item $\mu$ is SRB;

\item for $\mu$-almost every $p\in \T^4$, and for Lebesgue almost every point $q$ in $W^{uu}_{loc}(p)$
\[
E^-_{g,q} = DH^u_{p,q}(p)E^-_{g,p};
\]
\item for $\mu$-almost every $p\in \T^4$ the measure $\mu^c_p$ is atomic.
\end{enumerate}
\end{theorem}

To prove Theorem \ref{thm.rigidity1u} we will define a measurable change of coordinates using the strong unstable holonomies, so that after this change of coordinates we are in the setting of Theorem \ref{thm.rigiditybrownhertz}.

Recall that $\lambda<1$ is the rate of contraction of the linear Anosov $A$. Let $N$ be large enough such that
\[
\displaystyle (4N^2)^2\left(\lambda^{2N}\right)^{\alpha} < 1. 
\]
In particular, if the $C^2$-neighborhood $\mathcal{U}_N^{sk}$ of $f_N$ is sufficiently small, then for every $g\in \mathcal{U}_N^{sk}$ we have
\begin{equation}
\label{eq.2abunching}
\displaystyle \left(\frac{\|Dg|_{E_g^c}\|}{m(Dg|_{E_g^c})}\right)^2 (m(Dg|_{E^{uu}_g}))^{-\alpha} <1.
\end{equation}
Fix $g\in \mathcal{U}_N^{sk} \cap \mathrm{Diff}^{2+\alpha}(\T^4)$ and some $R>1$. Condition (\ref{eq.2abunching}) above is the $(2,\alpha)$-unstable center bunching condition defined in (\ref{eq.newbunching}). By Theorem \ref{thm.holonomyc2}, for any $p\in \T^4$, $q\in W^{uu}_{g,R}(p)$ the unstable holonomy $H^u_{p,q}: W^c_g(p) \to W^c_g(q)$ is a $C^2$-diffeomorphism, whose $C^2$-norm varies continuously with the choices of $p$ and $q$ as above.

Since $g$ is a partially hyperbolic skew product, we have that $g(p_1,p_2)= (g_{p_2}(p_1), g_2(p_2))$, where $g_2(p_2)$ is a $C^{2+\alpha}$-Anosov diffeomorphism of $\T^2$ which is topologically conjugated to $A^{2N}$. It is well known that a transitive $C^{1+ \alpha}$-Anosov diffeomorphism has an unique ergodic $u$-Gibbs measure. Let $\nu$ be such a measure for $g_2$ on $\T^2$.

Fix $\mathcal{R} = \{R_1,\cdots, R_m\}$ a small Markov partition for $A$ and observe that $\mathcal{R}$ is also a Markov partition for $A^{2N}$ for every $N\in \N$. By taking $N$ sufficiently large we may suppose that the transition matrix $P_{2N}$ associated with $\mathcal{R}$ for $A^{2N}$ verifies $(P_{2N})_{i,j} =1$, for every $i,j =1, \cdots m$. Let $\mathcal{R}_g$ be the image of $\mathcal{R}$ by the conjugacy map between $A^{2N}$ and $g_2$. It is easy to see that $\mathcal{R}_g$ is a small Markov partition for $g_2$ and the conjugacy implies that it has the same transition matrix $P_{2N}$. Define $\Sigma := \{1,\cdots, m\}^{\Z}$ which is the shift space associated with $\mathcal{R}_g$ for $g_2$, let $\sigma: \Sigma \to \Sigma$ be the left shift map, and let $\Theta: \Sigma \to \T^2$ be the continuous surjection that defines the semi-conjugacy between $\sigma$ and $g_2$. 

Let us set some notations. Write $\Sigma^- = \{(\xi_i)_{i\leq 0}: \xi_i \in \{1,\cdots ,m\}\}$ and $\Sigma^+ : =\{(\xi_i)_{i>0}: \xi_i \in \{1,\cdots ,m\}\} $. Let $\pi^-: \Sigma \to \Sigma^-$ and $\pi^+: \Sigma \to \Sigma^+$ be the natural projections. For a point $\xi\in \Sigma$ we write $\xi^- := \pi^-(\xi)$ and $\xi^+:= \pi^+(\xi)$ and we use the notation $\xi = (\xi^-,\xi^+)$.  The local unstable set of a point $\xi\in \Sigma$ is
\[
\Sigma^u_{loc}(\xi) = \{\eta\in \Sigma: \eta^- = \xi^-\}.
\] 

Define $\nu_{\sigma} := \Theta_* \nu$, and observe that this is an ergodic, $\sigma$-invariant measure. The partition $\Sigma^u_{loc}$ on local unstable sets forms a $\nu_{\sigma}$-measurable partition of $\Sigma$. Let $\mathcal{P}^u$ be the $\nu$-measurable $u$-subordinated partition given by the intersection of local unstable manifolds of $g_2$ with the rectangles from the Markov partition $\mathcal{R}_g$. Notice that $\mathcal{P}^u$ is equivalent (on a set of full $\nu$-measure) to the partition $\Sigma^u_{loc}$ (on a set of full $\nu_{\sigma}$-measure).  

It is easy to see that the partition $\Sigma^u_{loc}$ is an increasing partition. Let $\mathcal{B}^u$ be the sub-$\sigma$-algebra generated by the partition on local unstable sets. This is an increasing sub-$\sigma$-algebra.

It is well known that $\Theta$ is bijective in a set of full $\nu_{\sigma}$-measure, which we will denote by $\hat{D}$. We may further assume that $\hat{D}$ is $\sigma$-invariant. Let $D:= \Theta(\hat{D})$ this is a $g_2$-invariant set of full $\nu$-measure. Define $\Psi = Id \times \Theta^{-1}$, and notice that it is an isomorphism between $\T^2 \times D$ and $\T^2 \times \hat{D}$. Let $\pi'_2: \T^2 \times \Sigma \to \Sigma$ be the natural projection on the second coordinate.

Let $\mu$ be an ergodic $u$-Gibbs measure for $g$. By lemma \ref{lemma.projectugibbs}, $\nu = (\pi_2)_*\mu$. Consider the measure $\hat{\mu} := \Psi_* \mu$, and observe that it verifies $(\pi'_2)_*\hat{\mu} = \nu_{\sigma}$. We define the skew product on $\T^2 \times \hat{D}$ by  $\hat{g}(x,\xi)=(\hat{g}_{\xi}(x), \sigma(\xi))$, where $\hat{g}_{\xi}:= g_{\Theta(\xi)}$. We may extend $\hat{g}$ to $\T^2 \times \Sigma$ by setting $\hat{g}_{\xi} = Id$, for $\xi\notin \hat{D}$. Observe that $(g,\mu)$ is isomorphic (or measurably conjugated) to $(\hat{g},\hat{\mu})$ by the isomorphism $\Psi$. Since $\Psi$ is just the identity in the first coordinate, it is immediate that the center Lyapunov exponents of $\mu$ are the same as the fiber Lyapunov exponents of $\hat{\mu}_{\sigma}$. Furthermore, $\mu$ is SRB if and only if $\hat{\mu}$ is fiber-wise SRB.

We now introduce a change of coordinate in the fibers for the skew product $\hat{g}$ in a way that the new skew product will verify the conditions to apply Theorem \ref{thm.rigiditybrownhertz}.

Fix $\eta^+ \in \Sigma^+$ and define the function $\phi:\Sigma \to \Sigma$ by $\phi(\xi) = (\xi^-,\eta^+)$ for every $\xi\in \Sigma$. Observe that for each $\xi\in \Sigma$, $\phi(\xi) \in \Sigma^u_{loc}(\xi)$. In particular $\phi$ is $\mathcal{B}^u$-measurable.

For each $\xi \in \hat{D}$, since $\Theta(\xi)$ and $\Theta(\phi(\xi))$ belong to the same local unstable manifold for $g_2$, we define
\[
\begin{array}{rcl}
\Phi_{\xi}: \T^2  & \longrightarrow & \T^2 \\
x & \mapsto & H^u_{\Theta(\xi), \Theta(\phi(\xi))}(x).
\end{array}
\] 
To simplify our notation, we write $H^u_{\xi, \phi(\xi)} := H^u_{\Theta(\xi), \Theta(\phi(\xi))}$. We also define $\Phi: \T^2 \times D \to \T^2 \times D$ by $\Phi(x,\xi) = (\Phi_{\xi}(x), \xi)$. We can extend the definition of $\Phi$ to $\T^2 \times \Sigma$ by setting $\Phi_{\xi} = Id$ for $\xi \notin \hat{D}$. We consider a skew product $\tilde{g}$ on $\T^2 \times \Sigma$ defined by 
\begin{equation}
\label{eq.definitiong}
\tilde{g} = \Phi \circ \hat{g} \circ \Phi^{-1}.
\end{equation}
Consider the ergodic $\tilde{g}$-invariant measure $\tilde{\mu} = \Phi_*\hat{\mu}$ and observe that $(\pi'_2)_*\tilde{\mu} = \nu_{\sigma}$. The partition on the fibers $\T^2$ forms a measurable partition of $\T^2 \times \Sigma$. Let $\{\tilde{\mu}_{\xi}\}_{\xi \in \Sigma}$ be the family of conditional measures with respect to the fibers. Figure \ref{diagram.conjugacies} represents all these changes of coordinates that are conjugacies on subsets of full measure. 
\begin{figure}
\centering
\begin{tikzcd}
(\T^2 \times \T^2, \mu) \arrow[r,"g"]\arrow[d,"\Psi"] & (\T^2\times \T^2,\mu)\arrow[d,"\Psi"]\\
(\T^2 \times \Sigma, \hat{\mu}) \arrow[r,"\hat{g}"] \arrow[d,"\Phi"] & (\T^2 \times \Sigma, \hat{\mu})\arrow[d,"\Phi"]\\
(\T^2\times \Sigma, \tilde{\mu}) \arrow[r,"\tilde{g}"] & (\T^2 \times \Sigma, \tilde{\mu})
\end{tikzcd}
\caption{Changes of coordinates}\label{diagram.conjugacies}
\end{figure}
\begin{lemma}
\label{lem.changeofcoordinates}
The maps $\xi \mapsto \tilde{g}_{\xi}^{-1}$ and $\xi \mapsto \tilde{\mu}_{\xi}$ are $\mathcal{B}^u$-measurable.
\end{lemma}
\begin{proof}
Recall that $\tilde{g}_{\xi}^{-1} = (\tilde{g}_{\sigma^{-1}(\xi)})^{-1}$. Since the unstable holonomy commutes with $g$, and by the definition of $\hat{g}$, in what follows we will use that $H^u_{\xi,\eta} \circ \hat{g}_{\sigma^{-1}(\xi)} = \hat{g}_{\sigma^{-1}(\eta)} \circ H^u_{\sigma^{-1}(\xi), \sigma^{-1}(\eta)}$. By (\ref{eq.definitiong}), we have
\[
\begin{array}{rcl}
\tilde{g}_{\sigma^{-1}(\xi)}(x) &=& H^u_{ \xi, \phi(\xi)} \circ \hat{g}_{\sigma^{-1}(\xi)} \circ H^u_{\phi(\sigma^{-1}(\xi)) ,\sigma^{-1}(\xi) }(x)\\
& = & H^u_{\xi, \phi(\xi)} \circ H^u_{\sigma(\phi(\sigma^{-1}(\xi))), \xi} \circ \hat{g}_{\phi(\sigma^{-1}(\xi))}(x) = H^u_{\sigma(\phi(\sigma^{-1}(\xi))), \phi(\xi)} \circ \hat{g}_{\phi(\sigma^{-1}(\xi))}(x).
\end{array}
\]
Notice that $\phi(\xi)$ and $\phi(\sigma^{-1}(\xi))$ depend only on $\xi^-$, in particular $\tilde{g}_{\sigma^{-1}(\xi)}$ depends only on $\xi^-$. If $\eta \in \Sigma^u_{loc}(\xi)$, which means that $\eta^- = \xi^-$, then $\tilde{g}_{\sigma^{-1}(\xi)} = \tilde{g}_{\sigma^{-1}(\eta)}$ and hence the map $\xi \mapsto \tilde{g}_{\xi}^{-1}$ is constant on local unstable sets and it is $\mathcal{B}^u$-measurable.

Since $\mu$ is an $u$-Gibbs measure, and it projects to $\nu$, corollary \ref{cor.whatiwant} implies that for $\nu$-almost every $p_2$, and for Lebesgue almost every $q_2 \in W^{uu}_{g_2}(p_2)$ (for the riemannian volume of $W^{uu}_{g_2}(p_2)$), we have
\begin{equation}
\label{eq.upropertyholonomy}
\mu^c_{q_2} = (H^u_{p_2,q_2})_*\mu^c_{p_2}.
\end{equation}
At first, the disintegration $\mu^c_{q_2}$ is defined for almost every point inside the unstable manifold of $p_2$. However, using (\ref{eq.upropertyholonomy}), for any $q_2\in W^{uu}_{g_2}(p_2)$, we may consider the measure $\mu_{q_2} = (H^u_{p_2,q_2})_*\mu_{p_2}$. This defines a new disintegration that coincides with the original one in $\mu$-almost every point with the advantage that for $\nu$-almost every point the disintegration is defined along entire unstable manifolds. 
 
Since $\Psi$ is the identity on the fibers and a conjugation with the shift on the basis, for $\nu_{\sigma}$-almost every $\xi$ we obtain $\mu_{\Theta(\xi)} = \hat{\mu}_{\xi}$. Let us see the equivalent of property (\ref{eq.upropertyholonomy}) for $\hat{\mu}$. Consider the disintegration of $\nu_{\sigma}$ on the measurable partition $\Sigma^u_{loc}$. For $\nu_{\sigma}$-almost every $\xi$, let $\nu_{\sigma}^{\xi}$ be the conditional measure on $\Sigma^u_{loc}(\xi)$. Hence, for $\nu_{\sigma}^{\xi}$-almost every $\eta$, we have that $\hat{\mu}_{\eta} = (H^u_{\xi,\eta})_*\hat{\mu}_{\xi}$. 

In an analogous way as we did for $\mu$, we define the measure $\mu_{\eta}$ for every $\eta$ in the local unstable set of $\xi$ and this defines a new disintegration that coincides with the original disintegration on a set of full measure. By an abuse of notation we will use the notation $\hat{\mu}_{\xi}$ for the conditional measure of this new disintegration. We remark that this disintegration has the advantage of being defined along entire local unstable sets.   

By the definition of $\Phi$ we see that for $\nu_{\sigma}$-almost every $\xi$ and for any $\eta \in \Sigma^u_{loc}(\xi)$ the measure $\tilde{\mu}_{\eta} = (H^u_{ \eta, \phi(\xi)})_*\hat{\mu}_{\eta} = \hat{\mu}_{\phi(\xi)}$. In particular, the map $\xi \mapsto \tilde{\mu}_{\xi}$  is constant on local unstable sets and it is $\mathcal{B}^u$-measurable.  
\end{proof}

\begin{proof}[Proof of Theorem \ref{thm.rigidity1u}]
First, let us explain how the skew product $\tilde{g}$ verifies the hypothesis of Theorem \ref{thm.rigiditybrownhertz}. Since $\Sigma^u_{loc}$ is a decreasing partition, we have that $\mathcal{B}^u$ is a decreasing sub-$\sigma$-algebra. Let $\mathcal{B}^*$ be the $\tilde{\mu}$-completion of $\mathcal{B}_{\T^2} \otimes \mathcal{B}^u$, where $\mathcal{B}_{\T^2}$ is the Borel $\sigma$-algebra on $\T^2$. Recall that
\[
\tilde{g}_{\xi} = H^u_{\sigma(\phi(\xi)), \phi(\sigma(\xi))} \circ \hat{g}_{\phi(\xi)}.
\]

We claim that there exists a constant $R>0$ such that for any $\xi \in \Sigma$, we have $\Theta(\sigma(\phi(\xi))) \in W^{uu}_{g_2,R}(\Theta(\phi(\sigma(\xi))))$. Indeed, recall that we had fixed $\mathcal{R}_g =\{ R_{g,1}, \cdots, R_{g,m}\}$ a small Markov partition for $g_2$. Since $\phi(\xi) \in \Sigma^u_{loc}(\xi)$, we obtain that $\Theta(\phi(\xi))$ and $\Theta(\xi)$ belongs to the same local unstable manifold intersected with some rectangle $R_{g,i}$. Since the expansion rate of unstable manifolds for $g_2$ is close to $\lambda^{-2N}$, which is a constant, there exists $R_1>0$ that verifies  $\Theta(\sigma(\phi(\xi))) \in W^{uu}_{g_2, R_1}(\Theta(\sigma(\xi)))$, for any $\xi\in \Sigma$. To conclude, we observe that $\Theta(\phi(\sigma(\xi))) \in W^{uu}_{g_2,loc}(\Theta(\sigma(\xi)))$. Hence, by fixing $R$ sufficiently large we conclude our claim.

Since $g$ is $C^{2+\alpha}$, Theorem \ref{thm.holonomyc2} in the appendix implies that for every $\xi\in \Sigma$, the holonomy $H^u_{\sigma(\phi(\xi)), \phi(\sigma(\xi))}$ is a $C^2$-diffeomorphism of $\T^2$ with uniformly bounded $C^2$-norm. Since $\hat{g}_{\xi} = g_{\Theta(\xi)}$, we also have that all the $C^2$-diffeomophisms $\hat{g}_{\xi}$ belong to a compact subset of $\mathrm{Diff}^2(\T^2)$. We conclude that for every $\xi$, the $C^2$-norm of $\tilde{g}_{\xi}$ is uniformly bounded. Similar conclusion holds for $\tilde{g}^{-1}_{\xi}$. In particular, the skew product $\tilde{g}$ verifies the integrability condition (\ref{eq.conditionc2}). 

It is easy to see that the fiber-wise Lyapunov exponents of $(\tilde{g},\tilde{\mu})$ are the same as the center Lyapunov exponents of $(g,\mu)$. In particular, $\tilde{\mu}$ is a hyperbolic measure with a positive and a negative fiber-wise Lyapunov exponent.

Lemma \ref{lem.changeofcoordinates} states that $(\tilde{g}, \tilde{\mu})$ verifies the conditions $(1)$ and $(2)$ in the hypothesis of Theorem \ref{thm.rigiditybrownhertz}. Since the skew products $\tilde{g}$ fibers over the system $(\sigma, \nu_{\sigma})$, which is ergodic, we conclude that either 
\begin{enumerate}
\item the measure $\tilde{\mu}_{\xi}$ is atomic for $\nu_{\sigma}$-almost every $\xi$;
\item $\tilde{\mu}$ is fiber-wise SRB;
\item the stable distribution $(x,\xi) \mapsto E^-_{\tilde{g}}(x,\xi)$ is $\mathcal{B}^*$-measurable.
\end{enumerate}

Notice that the composition $(\Phi \circ \Psi)$ takes fibers of $\T^2 \times \T^2$ into fibers of $\T^2 \times \Sigma$. Furthermore, it acts as a $C^2$-diffeomorphism on each fiber. Observe also that it measurably conjugates the dynamics of $g$ and $\tilde{g}$ on a set of full $\mu$-measure. In particular, for $\nu$-almost every $p_2\in \T^2$ we have
\begin{equation}
\label{eq.relationcentermeasures}
\mu^c_{p_2} = (\Phi\circ \Psi)^{-1}_* \tilde{\mu}_{\Theta^{-1}(p_2)}.
\end{equation}
From (\ref{eq.relationcentermeasures}) above, $\tilde{\mu}_{\xi}$ is atomic if and only if $\mu_{\Theta(\xi)}$ is atomic, for $\nu_{\sigma}$-almost every $\xi$. 

Since $\mu$ is a $u$-Gibbs measure, it will be an SRB measure if and only if  it is fiber-wise SRB in the sense of definition \ref{defi.fiberwisesrb}. From (\ref{eq.relationcentermeasures}), we conclude that $\tilde{\mu}$ is fiber-wise SRB for $\tilde{g}$ if and only if $\mu$ is fiber-wise SRB for $g$. 

For the map $(x,\xi) \mapsto E^-_{\tilde{g}}(x,\xi)$ to be $\mathcal{B}^*$-measurable, it is equivalent to the following: for $\tilde{\mu}$-almost every $(x,\xi)$ and for $\nu_{\sigma}^{\xi}$-almost every $\eta \in \Sigma^u_{loc}(\xi)$, we have that $E^-_{\tilde{g}}(x,\xi) = E^-_{\tilde{g}}(x,\eta)$. Observe that the points $(x,\xi)$ and $(x,\eta)$ belong to the same local unstable set for $\tilde{g}$. By the conjugacy $(\Psi\circ \Phi)$, we conclude that 
\[
E^-_{\tilde{g}}(x,\xi) = DH^u_{\Theta(\xi), \Theta(\phi(\xi))}(x) E^-_{g,(x,\Theta(\xi))}.
\]
Since the measure is $u$-Gibbs, the third condition above is equivalent to for $\mu$-almost every $p\in \T^4$, for Lebesgue almost every point $q\in W^{uu}_{loc}(p)$, we have $E^-_{g,q} = DH^u_{p,q}(p)E^-_{g,p}$. 

All these conclusions hold for any $g\in \mathrm{Diff}^{2+\alpha}(\T^2)$ sufficiently $C^2$-close to $f_N$. This concludes the proof.
\end{proof}

We remark that the same proof of Theorem \ref{thm.rigidity1u} also gives the following theorem.

\begin{theorem}
\label{thm.theoremurigiditygeneral}
Let $S$ be a compact surface and let $\alpha \in (0,1)$ be a constant. Suppose that $g\in \mathrm{Sk}^{2+\alpha}(S\times \T^2)$ is a partially hyperbolic skew product which is $(2,\alpha)$-unstable center bunched. If $\mu\in \mathrm{Gibbs}^u(g)$ is an ergodic measure with one positive and one negative exponent along the center, then either
\begin{enumerate}
\item $\mu$ is an SRB measure;
\item for $\mu$-almost every $p$ and for Lebesgue almost every point $q\in W^{uu}_{loc}(p)$,
\[
E^-_{g,q} = DH^u_{p,q}(p) E^-_{g,p};
\] 
\item for $\mu$-almost every $p$, the measure $\mu^c_p$ is atomic.
\end{enumerate}
\end{theorem}

\section{The non invariance of stable directions by $u$-holonomies}
\label{section.noninvariancestable}
In this section we fix $N$ large and $\mathcal{U}_N$ small enough such that Theorem \ref{thm.estimateexponentsugibbs} holds for some small fixed $\delta>0$. In particular, if $g\in \mathcal{U}_N$ then any $u$-Gibbs measure for $g$ has both a positive and a negative center Lyapunov exponent for $\mu$ almost every point. Since $\mu$ has absolutely continuous disintegration with respect to strong unstable manifolds, for $\mu$-almost every point $p$, Lesbesgue almost every point $q\in W^{uu}_g(p)$ has a well defined Oseledec's stable and unstable directions in the center, where the Lebesgue measure we are considering is the measure restricted to the strong unstable manifold $W^{uu}_g(p)$.

Recall that for any $p\in \T^4$ and any $q\in W^{uu}_{g}(p)$, there is a well defined unstable holonomy map $H^u_{p,q}:W^c_g(p) \to W^c_g(q)$. Furthermore, this map is a $C^1$-diffeomorphism. The main result in this section is the following:

\begin{proposition}
\label{prop.noninvariancestable}
Let $g\in \mathcal{U}_N$ and let $\mu$ be a $u$-Gibbs measure for $g$. For any $\varepsilon>0$, the following property holds: for $\mu$-almost every $p$, there exists a set $D^u$ contained in $W^{uu}_{g,\varepsilon}(p)$ with positive Lebesgue measure (for the riemannian volume of $W^{uu}_{g,\varepsilon}(p)$) such that for any $q\in D^u$ it is satisfied that
\[
DH^u_{p,q}(p)E^-_{g,p} \neq E^-_{g,q}.
\]
\end{proposition}

The rest of this section is dedicated to prove proposition \ref{prop.noninvariancestable}. 

Let $g\in \mathcal{U}_N$, for any $p\in \T^4$, for any piece of strong unstable manifold $\gamma^u_p$ containing $p$ and any unit vector $v\in E^c_{g,p}$, we define a unitary vector field over $\gamma^u_p$ defined as follows: for any $q\in \gamma^u_p$ we write 
\begin{equation}
\label{eq.vuvectorfield}
PH^u_{p,q}(p)v =   \displaystyle \frac{DH^u_{p,q}(p)v}{\|DH^u_{p,q}(p)v\|},
\end{equation}
and define  $v'_q:= PH^u_{p,q}(p)v$. First we study the regularity of the vector field $v'$. 

\begin{lemma}
\label{lemma.regularityvf}
Let $g\in \mathcal{U}_N$. There exists a constant $C>0$ that verifies the following: for any $p\in \T^4$, let $\gamma^u_p := W^{uu}_{g,1}(p)$ be the strong unstable manifold of size $1$, for any unit vector $v\in E^c_{g,p}$, the vector field $v'$ defined above is $(C,\frac{1}{2})$-H\"older.
\end{lemma}
\begin{proof}
Observe that, for $N$ large enough, we have
\[
\left( \lambda^{2N}\right)^{\frac{1}{2}} < (4N^2)^{-1} \textrm{ and } \left(\lambda^{2N}\right)^{\frac{1}{2}} < (2N)^{-1}.
\] 
This means that $f_N$ verifies the conditions (\ref{eq.ch5thetabunching}) and (\ref{eq.ch5thetapinching}) from Theorem \ref{ob.holonomies}, for $\theta=\frac{1}{2}$. In particular, any $g$ sufficiently $C^1$-close to $f_N$ also verifies (\ref{eq.ch5thetabunching}) and (\ref{eq.ch5thetapinching}). Lemma \ref{lemma.regularityvf} then follows from the conclusion  (\ref{eq.holderconditionholonomy}), for unstable holonomies, of Theorem \ref{ob.holonomies}.  
\end{proof}   

Next, we will see how the center bunching condition ``smoothes'' a center vector field over a piece of strong unstable manifold. This is a crucial point for us, so that it will allow us to apply some of the techniques and estimates from section \ref{section.centerlyapunovexpoenents} to prove proposition \ref{prop.noninvariancestable}.

\begin{lemma}
\label{lemma.vectorconstantbunching}
Let $g\in \mathcal{U}_N$. For any piece of strong unstable curve $\gamma^u$ and any $X$ unitary vector field over $\gamma^u$ tangent to $E^c_g$ which is $(C_0,\frac{1}{2})$-H\"older, for some $C_0:=C_0(X)>0$, the following holds: there exists $n_0\in \N$, which depends only on $C_0$,  such that for every $n\geq n_0$, the vector field $X_n := \frac{g^n_*(X)}{\|g^n_*(X)\|}$ over $g^n(\gamma^u)$ is $(C_n,\frac{1}{2})$-H\"older with $C_n < 30 N^2 \lambda ^N$.   
\end{lemma}
\begin{proof}
The proof of this lemma is essentially contained in the proof of Lemma $1$ from \cite{ch5bergercarrasco2014}. However, we will repeat the main steps of the argument here. For simplicity we will prove the lemma for $f_N$, which we will denote by $f$. Using the estimates from Lemma \ref{ob.manyconsiderationsnotfibered}, one can adapt the calculations for any $g \in \mathcal{U}_N$. 

Let us just review some estimates for $f$. Recall that
\[
Df(x,y,z,w) =
\begin{pmatrix}
Ds_N(x,y) & P_x \circ A^N(z,w)\\
0 & A^{2N}(z,w) 
\end{pmatrix}.
\]
Hence, $\|Df(x,y,z,w)|_{E^c}\| = \|Ds_n(x,y)\|\leq 2N$. Since $Ds_N$ is the only non linear term, $\|D^2f\|  = \|D^2s_N\|\leq N$.

Let $\gamma^u$ be a piece of a strong unstable manifold and $X$ a $(C_0,\frac{1}{2})$-H\"older unitary vector field over $\gamma^u$. Let us estimate $C_1$, the H\"older constant of $X_1$ over $f(\gamma^u)$. First, for any $m,m'\in \gamma^u$, we have
\[
\begin{array}{l}
\|Df(m) X_m - Df(m') X_{m'}\| \leq \\
\|Df(m)X_m - Df(m)X_{m'}\| + \|Df(m)X_{m'} - Df(m') X_{m'}\| = I + II. 
\end{array}
\] 

Since $X$ is $(C_0,\frac{1}{2})$-H\"older, we obtain
\[
I \leq 2N \|X_m - X_{m'}\| \leq 2N C_0 d(m,m')^{\frac{1}{2}}.
\]

If $d(m,m')< 1$, then
\[
II \leq N d(m,m')< N d(m,m')^{\frac{1}{2}}< 7Nd(m,m')^{\frac{1}{2}}.
\]
Observe that $d(m,m')\leq 2\pi< 7$, for any two points $m,m'\in \T^4$. If $d(m,m')\geq1$, then
\[
II \leq N d(m,m')\leq 7N < 7N d(m,m')^{\frac{1}{2}}.
\]
We conclude that
\begin{equation}
\label{eq.estimatebunching}
\|Df(m) X_m - Df(m') X_{m'} \| < (7N + 2N C_0) d(m, m')^{\frac{1}{2}}.
\end{equation}
Also,
\[\arraycolsep=1.2pt\def\arraystretch{2}
\begin{array}{rcl}
\|(X_1)_m - (X_1)_{m'} \| & = & \displaystyle \frac{1}{\|f_*X_m\|\|f_*X_{m'}\|}\left\| \|f_*X_{m'}\|f_*X_m - \|f_*X_m\| f_*X_{m'}\right\|\\
& \leq & \displaystyle \frac{1}{\|f_*X_m\|\|f_*X_{m'}\|}\left(\left\| \|f_*X_{m'}\|f_*X_m - \|f_*X_{m'}\| f_*X_{m'}\right\|\right. \\
&& \displaystyle + \left.\left\|\|f_*X_{m'}\| f_*X_{m'} - \|f_*X_m\| f_*X_{m'}\right\|\right)\\
& \leq & \displaystyle \frac{2}{\|f_*X_{m}\|} \|f_*X_{m} - f_*X_{m'}\| \\
&=& \displaystyle \frac{2}{\|f_*X_m\|} \|Df(f^{-1}(m))X_{f^{-1}(m)} - Df(f^{-1}(m'))X_{f^{-1}(m')}\|.
\end{array}
\] 
Lemma \ref{proposition1bc} states that the unstable direction is contained in a cone around $e^u$ of size bounded from above by $2\lambda^N$.  This implies that for any two points in the same $m$ and $m'$ in the same strong unstable manifold $d^u(m,m') \leq (1 + 2\lambda^N)d(m,m')$. Therefore, by (\ref{eq.estimatebunching}), lemma \ref{ob.manyconsiderationsnotfibered},  and since the points $m$ and $m'$ belong to the same strong unstable manifold,  we have 
\[
\begin{array}{rcl}
\|Df(f^{-1}(m))X_{f^{-1}(m)} - Df(f^{-1}(m'))X_{f^{-1}(m')}\| & \leq & N(7 + 2C_0) d(f^{-1}(m), f^{-1}(m'))^{\frac{1}{2}} \\
& \leq &  N\lambda^N(1+2\lambda^N)^{\frac{1}{2}}(7 +2 C_0) d(m,m')^{\frac{1}{2}}. 
\end{array}
\]
Recall that $\|Df|_{E^c_f}\|\geq (2N)^{-1}$, hence 
\[
\|X_1(m)- X_1(m') \| \leq \displaystyle \frac{2N\lambda^N(1+2\lambda^N)^{\frac{1}{2}}(7 + 2 C_0) d(m,m')^{\frac{1}{2}}}{\|f_*X(m)\|}\leq 4 (1+2\lambda^N)^{\frac{1}{2}} N^2\lambda^N (7+2C_0) d(m,m')^{\frac{1}{2}}.  
\]
Observe that $4 (1+2\lambda^N)^{\frac{1}{2}} N^2\lambda^N (7+2C_0) $ estimates the H\"older constant of $X_1$. If $C_0 \leq \frac{1}{10}$, then for $N$ large enough 
\[
\displaystyle 4 (1+2\lambda^N)^{\frac{1}{2}} N^2\lambda^N (7+2C_0)  \leq 4 (1+2\lambda^N)^{\frac{1}{2}} N^2 \lambda^N (7.2)< 30 N^2 \lambda^N.
\]
Hence, $C_1< 30 N^2 \lambda^N$ and the same calculations imply that $C_n < 30 N^2 \lambda^N$, for every $n\geq 1$. Now suppose that $C_0 >\frac{1}{10}$. Then, for $N$ large enough
\[
\displaystyle \frac{4 (1+2\lambda^N)^{\frac{1}{2}}N^2\lambda^N (7+2C_0)}{C_0} =   4 (1+2\lambda^N)^{\frac{1}{2}}N^2\lambda^N \left(\frac{7}{C_0}+2\right)<  4 (1+2\lambda^N)^{\frac{1}{2}} N^2 \lambda^N 72 < \frac{1}{2}.
\]
This implies that $C_1< \frac{C_0}{2}$. Therefore, there exists $\tilde{n}\in \N$ such that $C_{\tilde{n}} < \left(\frac{1}{2}\right)^{\tilde{n}} C_0\leq \frac{1}{10}$. Take $n_0 = \tilde{n} + 1$. We conclude that for every $n\geq n_0$, $C_n< 30 N^2 \lambda^N$.
\end{proof}

\begin{proof}[Proof of proposition \ref{prop.noninvariancestable}]
If the conclusion of proposition \ref{prop.noninvariancestable} did not hold, there would exist a diffeomorphism $g\in \mathcal{U}_N$, an $u$-Gibbs measure $\mu$ and a measurable set $D$ of positive $\mu$-measure such that for any $p\in D$ and for Lebesgue almost every point $q\in W^{uu}_{g}(p)$ we would have $DH^u_{p,q}(E^-_{g,p}) = E^-_{g,q}$. Fix $p\in D$ and let $\gamma^u := W^{uu}_{g,1}(p)$. Consider $v$ an unit vector on $E^-_{g,p}$ and let $v'$ be the unit vector field over $\gamma^u$ defined as in (\ref{eq.vuvectorfield}).

Let $C$ be the constant given by lemma \ref{lemma.regularityvf}. Therefore, $v'$ is a $(C,\frac{1}{2})$-H\"older vector field over $\gamma^u$. Let $n_0\in \N$ be given by lemma \ref{lemma.vectorconstantbunching}. Hence, for $n\geq n_0$, the vector field $v_n':=\frac{g^n_*(v')}{\|g^n_*(v')\|}$ is $(C_n,\frac{1}{2})$-H\"older over $\gamma^u_n := g^n(\gamma^u)$, with $C_n < 30 N^2 \lambda^N$.

Suppose that $n_0$ is large enough such that $l(\gamma^u_{n_0}) > 2\pi$. Hence, we may consider a $C^1$-curve $\tilde{\gamma} : [0,2\pi] \to \T^4$ such that $\tilde{\gamma} = (\tilde{\gamma}_x, \tilde{\gamma}_y, \tilde{\gamma}_z, \tilde{\gamma}_w)$ with $\left| \frac{d \tilde{\gamma}_x}{dt} \right| =1$, $\tilde{\gamma}([0,2\pi]) \subset \gamma^u_{n_0}$, and define $\tilde{v}= v^u_{n_0}$. Following definition \ref{ob.adaptedfield}, the pair $(\tilde{\gamma}, \tilde{v})$ is an adapted field.

Recall that $\delta>0$ is fixed and in section \ref{section.centerlyapunovexpoenents}, on the proof of Theorem \ref{thm.estimateexponentsugibbs}, we fixed $\tilde{\delta} = \frac{2\delta}{15}$. For each $k\geq 0$, we write $\tilde{v}_k =  v'_{n_0+k}$ and recall that there exists $N_k\in \N$ such that 
\[
g^k \circ \tilde{\gamma} = \tilde{\gamma}^k_1 \ast \cdots \ast \tilde{\gamma}^k_{N_k} \ast \tilde{\gamma}^k_{N_k+1}, 
\]
where $\tilde{\gamma}^k_j$ is an $u$-curve for $j=1, \cdots, N_k$ and $\tilde{\gamma}^k_{N_k+1}$ is a segment of a $u$-curve. By lemma \ref{ob.continuesadapted}, every pair $(\tilde{\gamma}^k_j, \tilde{v}_k|_{\tilde{\gamma}^k_j})$ is an adapted field for $j=1,\cdots, N_k$. 

Recall that in section \ref{section.centerlyapunovexpoenents}, we had defined the notion of $\tilde{\delta}$-good adapted field (see definition \ref{def.deltabadandgood}). We will need the following lemma.

\begin{lemma}[\cite{obataergodicity}, Lemma $7.27$]
Let $g\in \mathcal{U}_N$, and let $(\gamma, X)$ be a $\tilde{\delta}$-bad adapted field. Then there exists a strip $S$ of length $\pi$ such that for every $j$ satisfying $g^{-1} \gamma_j^1\subset S$, the field $(\gamma^1_j, \frac{g_*X}{\|g_*X\|})$ is $\tilde{\delta}$-good.
\end{lemma}

 Let $(\hat{\gamma}, \hat{v})$ be a $\tilde{\delta}$-good adapted field defined as follows: if $(\tilde{\gamma}, \tilde{v})$ is a $\tilde{\delta}$-good adapted field then $(\hat{\gamma},\hat{v}) = (\tilde{\gamma}, \tilde{v})$. Otherwise, by the previous lemma, we may choose $j\in \{1, \cdots N_k\}$ such that $(\tilde{\gamma}^1_j, \tilde{v}_1|_{\tilde{\gamma}^1_j})$ is a $\tilde{\delta}$-good adapted field. In this case, we define $(\hat{\gamma}, \hat{v}) = (\tilde{\gamma}^1_j, \tilde{v}_1|_{\tilde{\gamma}^1_j})$.   
   
Let $K \in \{n_0, n_0+1\}$ be such that $g^{-K}(\hat{\gamma}) \subset \gamma^u$ and write $\hat{\gamma}_{-K} := g^{-K}(\hat{\gamma})$. Recall that we had defined $J^{uu}_{g^k}(.) = | \det Dg^k(.)|_{E^{uu}_g}|$. For any $n\in \N$, 
\[\arraycolsep=1.2pt\def\arraystretch{2}
\begin{array}{rcl}
\displaystyle \frac{1}{|\hat{\gamma}|}\int_{\hat{\gamma}_{-K}} \log \|Dg^{K+n} v'\| d\hat{\gamma}_{-K}& = & \displaystyle \frac{1}{|\hat{\gamma}|}\left(\sum_{i=0}^{K-1}\int_{g^i \circ \hat{\gamma}_{-K}} \log \|Dg v'_i\|J^{uu}_{g^{-i}} d(g \circ\hat{\gamma}_{-K})\right.\\
&& \displaystyle \left. +\int_{\hat{\gamma}} \log \|Dg^{n} \hat{v}\| J^{uu}_{g^{-K}} d\hat{\gamma}\right)\\
& = & \displaystyle M_K + \frac{1}{|\hat{\gamma}|}\int_{\hat{\gamma}} \log \|Dg^{n} \hat{v}\| J^{uu}_{g^{-K}} d\hat{\gamma} = M_k + I^{\hat{\gamma},\hat{v}}_n,
\end{array}
\] 
where $M_K$ does not depend on $n$. Since $(\hat{\gamma}, \hat{v})$ is a $\tilde{\delta}$-good curve, by (\ref{eq.estimateIn}) in section \ref{section.centerlyapunovexpoenents}, for $n$ large enough we have
\[
\frac{ I^{\hat{\gamma},\hat{v}}_n}{n} \geq (1-14\tilde{\delta}) \log N.
\]
Therefore,
\begin{equation}
\label{eq.positive}
\displaystyle \limsup_{n\to + \infty}\frac{1}{|\hat{\gamma}|}\int_{\hat{\gamma}_{-K}} \frac{\log \|Dg^{K+n} v^u\| }{n}d\hat{\gamma}_{-K} = \limsup_{n\to + \infty} \frac{M_K}{n} + \frac{I^{\hat{\gamma}, \hat{v}}_n}{n} \geq (1-14\tilde{\delta}) \log N >0.
\end{equation}
However, by assumption, for Lebesgue almost every $q\in W^{uu}_g(p)$ the vector $v^u_q$ belongs to $E^-_{g,q}$. In particular, there exists a number $\lambda^-<0$ such that for Lebesgue almost every $q\in W^{uu}_g(p)$ 
\begin{equation}
\label{eq.negative}
\displaystyle \lim_{n\to +\infty}  \frac{\log \|Dg^{K+n}(q) v^u_q\| }{n} = \lambda^- <0.
\end{equation}
By (\ref{eq.negative}) and applying the dominated convergence theorem, we obtain
\[\arraycolsep=1.2pt\def\arraystretch{2}
\begin{array}{rcl}
\displaystyle \limsup_{n\to + \infty}\frac{1}{|\hat{\gamma}|}\int_{\hat{\gamma}_{-K}} \frac{\log \|Dg^{K+n} v^u\| }{n}d\hat{\gamma}_{-K} &=& \displaystyle \frac{1}{|\hat{\gamma}|}\int_{\hat{\gamma}_{-K}} \limsup_{n\to + \infty}\frac{\log \|Dg^{K+n} v^u\| }{n}d\hat{\gamma}_{-K},\\
&=& \displaystyle \frac{|\hat{\gamma}_{-K}|}{|\hat{\gamma}|} \lambda^- <0.  
\end{array}
\]  
which is a contradiction with (\ref{eq.positive}). 
\end{proof}

\section{Measures with atomic center disintegration and the proof of Theorem \ref{thm.thmB}}
\label{section.atomiccenter}
In this section we conclude the proof of Theorem \ref{thm.thmB}. The main ingredient that is missing to conclude this proof is the following theorem.
\begin{theorem}
\label{thm.atomicugibbs}
Let $g\in \mathrm{Sk}^2(\T^2\times \T^2)$ be a partially hyperbolic, center bunched skew product. Suppose that there exists a constant $\theta \in (0,1)$ such that
\[
E^{uu}_g \textrm{ is $\theta$-H\"older and } \|Dg|_{E^{ss}_g}\|^{\theta} < m(Dg|_{E_g^c}).
\]

Let $\mu$ be an ergodic $u$-Gibbs measure for $g$ that verifies:
\begin{enumerate}
\item $\mu$ has atomic center disintegration;

\item for $\mu$-almost every $p\in \T^4$, and for Lebesgue almost every point $q$ in $W^{uu}_{loc}(p)$
\[
E^-_{g,q} \neq DH^u_{p,q}(p)E^-_{g,p}.
\]
\end{enumerate}
Then there exists a finite number of $C^1$ two dimensional tori $T^1_{\mu}, \cdots,T^l_{\mu} $ such that each of them is tangent to $E^{ss}_g \oplus E^{uu}_g$ and $\mathrm{supp}(\mu) = \cup_{i=1}^{l}T^i_{\mu}$.  
\end{theorem}

\begin{remark}
\label{remark.ucomment}
Theorem \ref{thm.atomicugibbs} also holds for a partially hyperbolic skew product $g\in \mathrm{Sk}^2(S \times \T^2)$, where $S$ is a compact surface, verifying the rest of the  hypotheses of the theorem.
\end{remark}

In section \ref{subsection.reducing}, we reduce the proof of Theorem \ref{thm.atomicugibbs} into proving the $s$-invariance of the center conditional measures, given by Theorem \ref{thm.sinvariance}. The proof of Theorem \ref{thm.sinvariance} is then given in section \ref{subsection.sinvariance}. 

\subsection{Proof of Theorem \ref{thm.thmB} assuming Theorem \ref{thm.atomicugibbs}}
Let $\alpha\in (0,1)$ and take $N$ large enough such that Theorem \ref{thm.rigidity1u} holds and let $\mathcal{U}_N^{sk}$ be a small $C^2$-neighborhood of $f_N$ is $\mathrm{Sk}^2(\T^2\times \T^2)$. Take $g\in \mathcal{U}_N^{sk} \cap \mathrm{Sk}^{2+\alpha}(\T^4)$ and  take an ergodic measure $\mu\in \mathrm{Gibbs}^u(g)$. By Theorem \ref{thm.rigidity1u}, there are three possibilities:
\begin{enumerate}
\item for $\mu$-almost every $p\in \T^4$, the measure $\mu^c_p$ is atomic;
\item $\mu$ is SRB;
\item for $\mu$-almost every $p\in \T^4$, for Lebesgue almost every point $q$ in $W^{uu}_{loc}(p)$, we have
\[
E^-_{g,q} = DH^u_{p,q}(p)E^-_{g,p}.
\]
\end{enumerate}  

Let us verify that $g$ verifies the H\"older condition in the statement of Theorem \ref{thm.atomicugibbs}.  If $\theta \in (0,1)$ is a number that verifies
\[
\displaystyle \frac{\|Dg(p)|_{E^c_g}\|}{m(Dg(p)|_{E^{uu}_g})} < m(Dg(p)|_{E^{ss}_g})^{\theta}, \textrm{ for all $p\in \T^4$,}
\]
then $E^{uu}_g$ is $\theta$-H\"older (see Section $4$ from \cite{ch5pughshubwilkinson12}).  In particular, the maximum $\theta$ we can take is arbitrarily close to 
\begin{equation}\label{eq.thetaestimate1}
\displaystyle \inf_{p\in \T^4}\left\lbrace\frac{\log m(Dg(p)|_{E^{uu}_g}) - \log \|Dg(p)|_{E^c_g}\|}{-\log m(Dg(p)|_{E^{ss}_g})}\right\rbrace.
\end{equation}
From the estimates of Lemma \ref{ob.manyconsiderationsnotfibered}, for some small $\varepsilon_1>0$, if $N$ is sufficiently large, we have that \eqref{eq.thetaestimate1} is greater than
\[
\displaystyle \frac{2N \log \tilde{\mu} \left( 1- \frac{(\varepsilon_1 + \log 2N)}{2N \log \tilde{\mu}}\right)}{2N \log \tilde{\mu} \left( 1- \frac{\varepsilon_1}{2N \log \tilde{\mu}} \right)},
\] 
and this can be made arbitrarily close to $1$ by taking $N$ sufficiently large (we remark that an analogous estimate of item $1$ from Lemma \ref{ob.manyconsiderationsnotfibered} holds for the strong stable direction).

On the other hand, we also want $\|Dg|_{E^{ss}_g}\|^{\theta}< m(Dg|_{E^c_g})$. From the estimates of Lemma \ref{ob.manyconsiderationsnotfibered} for this inequality to hold we need
\[
\theta > \displaystyle \frac{\log 2N}{2N \log \tilde{\mu} \left( 1  - \frac{\varepsilon_1}{2N \log \tilde{\mu}} \right)}.
\]
The right side of the inequality goes to $0$ as $N$ increases. So for $N$ sufficiently large, we may take $\theta \in (0,1)$ that verifies the hypothesis of Theorem \ref{thm.atomicugibbs}.

Let $\mu$ be an ergodic $u$-Gibbs measure for $g$. By Proposition \ref{prop.noninvariancestable}, $\mu$ cannot verify item $3$ above. If $\mu$ has atomic center disintegration, then $\mu$ verifies the hypothesis of Theorem \ref{thm.atomicugibbs}, and therefore, there exist $T^1_{\mu}, \cdots, T^l_{\mu}$ which are $C^1$-tori tangent to $E^{ss}_g \oplus E^{uu}_g$ such that $\mathrm{supp}(\mu) = \cup_{i=1}^l T^i_{\mu}$. 

If $\mu$ does not have atomic center disintegration, then $\mu$ is an SRB measure and this concludes the proof of the theorem.

\subsection{Reducing the proof of Theorem \ref{thm.atomicugibbs} into proving $s$-invariance of the center conditional measures}\label{subsection.reducing}

In this subsection we will reduce the proof of Theorem \ref{thm.atomicugibbs} into proving the following theorem:

\begin{theorem}[The $s$-invariance of measures with atomic disintegration]
\label{thm.sinvariance}
Let $g\in \mathrm{Sk}^2(\T^2\times \T^2)$ be a partially hyperbolic, center bunched skew product. Suppose that there exists a constant $\theta \in (0,1)$ such that
\[
E^{uu} \textrm{ is $\theta$-H\"older and } \|Dg|_{E^{ss}}\|^{\theta} < m(Dg|_E^c).
\]
Let $\mu$ be an ergodic  $u$-Gibbs measure for $g$ and suppose that $\mu$ verifies:
\begin{enumerate}
\item $\mu$ has atomic center disintegration;

\item for $\mu$-almost every $p\in \T^4$, and for Lebesgue almost every point $q$ in $W^{uu}_{loc}(p)$
\[
E^-_{g,q} \neq DH^u_{p,q}(p)E^-_{g,p}.
\]
\end{enumerate}
Let $\nu$ be the unique SRB-measure on $\T^2$ for the Anosov diffeomorphism $g_2$. Then there exists a set $X$ of full $\nu$-measure such that for any $p_2,q_2\in X$ in the same stable leaf for $g_2$, we have that
\[
\mu^c_{q_2} = (H^s_{p_2,q_2})_* \mu^c_{p_q}.
\]
\end{theorem}

For the rest of this section we suppose that $g$ and $\mu$ verify the conditions of Theorem \ref{thm.atomicugibbs}.
Recall that $\mu \in \mathrm{State}^u_{\nu}(g)$, by Corollary \ref{cor.whatiwant}, the measure $\mu$ has $u$-invariant center conditional measures (see (\ref{eq.uinvariancemeasure}) for the definition of $u$-invariant). By Theorem \ref{thm.sinvariance}, the measure $\mu$ also has $s$-invariant center conditional measures.

Observe that for $g_2$, after fixing some small $\varepsilon>0$, for any $p_2 \in \T^2$ there exists a neighborhood $U$ of $p_2$ such that the map $[.,.]: W^{ss}_{g_2, \varepsilon}(p_2) \times W^{uu}_{g_2,\varepsilon}(p_2) \to U$ defined by $[x^s,y^u] = W^{ss}_{g_2,\varepsilon}(x^s) \cap W^{uu}_{g_2, \varepsilon}(y^u)$ is a homeomorphism. For a probability measure $\hat{\nu}$ on $\T^2$ we say that it has \textbf{local product structure} if for any $p_2\in \mathrm{supp}(\hat{\nu})$, using the homeomorphism $[.,.]$ as above, the measure $\hat{\nu}$ can be written as $\rho \hat{\nu}^s \times \hat{\nu}^u$, where $\rho$ is a positive measurable function.

Since the measure $\nu$ is the unique SRB measure of $g_{2}$, in particular, it is an equilibrium state for the logarithm of the unstable jacobian of $g_2$, it has local product structure (see \cite{ch5bowenbook}). Since the unstable foliation of $g_2$ is minimal (i.e. every unstable manifold is dense), then the support of $\nu$ is $\T^2$. The proof of the following proposition can be found in \cite{ch5avilavianainvariance}. It is a type of Hopf argument, and is a consequence of the local product structure of $\nu$, the $su$-invariance of the center conditional measures of $\mu$, and that the support of $\nu$ is the entire torus. 

\begin{lemma}
\label{lem.continuousdisintegration}
Let $\mu$ be a measure as above. Then there exists a disintegration $\{\overline{\mu}^c_{p_2}\}_{p_2\in \T^2}$ of $\mu$ with the following properties:
\begin{enumerate}
\item the measures $\overline{\mu}^c_{p_2}$ coincides with the measures $\mu^c_{p_2}$ for $\nu$-almost every $p_2$;
\item the map $p_2 \mapsto \overline{\mu}^c_{p_2}$ is continuous;
\item for any $p_2 \in \T^2$, and $q_2\in W^{uu}_{g_2}(p_2)$ we have $\overline{\mu}^c_{q_2} = (H^u_{p_2,q_2})_* \overline{\mu}^c_{p_2}$;
\item for any $p_2 \in \T^2$, and $q_2 \in W^{ss}_{g_2}(p_2)$ we have $\overline{\mu}^c_{q_2} =  (H^s_{p_2,q_2})_* \overline{\mu}^c_{p_2}$.  
\end{enumerate}

\end{lemma}

By an abuse of notation, we will denote the disintegration $\{\overline{\mu}^c_{p_2}\}_{p_2\in \T^2}$ by $\{\mu^c_{p_2}\}_{p_2\in \T^2}$. We will also need the following lemma:
\begin{lemma}
\label{lem.finiteatoms}
There exists $k\in \N$ such that for every $p_2\in \T^2$ the measure $\mu^c_{p_2}$ has $k$-atoms. 
\end{lemma}
\begin{proof}
We already know that the measure $\mu^c_{p_2}$ is atomic for every $p_2\in \T^2$. For each $n\in \N$ consider the set $B_n:=\{p\in \T^4: \mu^c_{\pi_2(p)}(\{p\}) >\frac{1}{n}\}$. It is easy to see that $B_n$ is a $g$-invariant set for each $n\in \N$. For $n$ sufficiently large $\mu(B_n)>0$, and by ergodicity $\mu(B_n)=1$. Hence, for $n$ large enough, every atom of $\mu^c_{p_2}$ has measure larger than $\frac{1}{n}$, for any $p_2\in \T^2$, and therefore there are finitely many atoms. Let $\mathcal{A}_{p_2}$ be the set of atoms of $\mu^c_{p_2}$. The $su$-invariance implies that for any $\# \mathcal{A}_{p_2} = \#\mathcal{A}_{q_2}$, for any $p_2,q_2\in \T^2$. Hence, there exists $k\in \N$ such that $\mu^c_{p_2}$ has exactly $k$-atoms, for every $p_2\in \T^2$.
\end{proof}

\begin{proof}[Proof of Theorem \ref{thm.atomicugibbs}]
Let $k\in \N$ be as in lemma \ref{lem.finiteatoms} and fix $p_2\in \T^2$. Let $\mathcal{A}_{p_2}:\{x_1, \cdots, x_k\}$ be the set of atoms of $\mu^c_{p_2}$. The $su$-invariance of $\{\mu^c_{p_2}\}_{p_2\in \T^2}$ implies that for any $x_i\in \mathcal{A}_{p_2}$, the endpoint of any $su$-path starting in $x_i$ and ending in $W^c(x_i)$ also belongs to $\mathcal{A}_{p_2}$. In particular, for each $x_i\in \mathcal{A}_{p_2}$, its accessibility class $AC(x_i)$ is trivial (see section \ref{sec.preliminaries} for the definition of trivial accessibility class).

There exists $l\in \N$ such that the union of the accessibility classes of the points $x_i$ can be partitioned into $l$ disjoint accessibility classes, that is,
\[
\displaystyle \bigcup_{i=1}^k AC(x_i) = T^1_{\mu} \sqcup \cdots \sqcup T^l_{\mu},
\]
where each $T^i_{\mu}$ is an accessibility class.

We will prove that each $T^i_{\mu}$ is a two dimensional torus tangent to $E^{ss}_g\oplus E^{uu}_g$. Since $\mathrm{supp}(\mu) \subset \bigcup_{i=1}^k AC(x_i)$, this will conclude the proof of the theorem.

By Theorem \ref{thm.accclassesc1}, we know that $T^i_{\mu}$ is a $C^1$-submanifold of $\T^4$. Furthermore, since the accessibility class $T^i_{\mu}$ is trivial, we obtain that it is a $2$-dimensional $C^1$-submanifold immersed in $\T^4$ and by the definition of accessibility class it is tangent to $E^{ss}_g\oplus E^{uu}_g$.

\begin{claim}
$T^i_{\mu}$ is compact.
\end{claim}
\begin{proof}
Since the direction $E^{ss}_{g} \oplus E^{uu}_g$ is uniformly transverse to $E^c_g$, the surface $T^i_{\mu}$ is uniformly transverse to the center foliation. If it was not compact, then it would intersect some center leaf $W^c(p)$ infinitely many times. However, this is a contradiction with the fact that $T^i_{\mu} \cap W^c(p) \subset \mathcal{A}_{\pi_2(p)}$ which is finite.  
\end{proof}

Since the strong stable and strong unstable manifolds of $g$ projects to the stable and unstable manifolds of $g_2$, which are dense in $\T^2$ we have that $\pi_2(T^i_{\mu}) = \T^2$. Let us see that $\pi_2|_{T^i_{\mu}}: T^i_{\mu} \to \T^2$ is a covering map. The property that $T^i_{\mu}$ is $su$-saturated implies that $\pi_2|_{T^i_{\mu}}$ is surjective. Since $T^i_{\mu}$ is tangent to $E^{ss}_g\oplus E^{uu}_g$, which is uniformly transverse to the fibers (center direction), then for any point $p_2\in \T^2$ and small neighborhood $U$ of $p_2$, any connected component of $\pi_2^{-1}(U) \cap T^i_{\mu}$ is diffeomorphic to $U$, hence, $\pi_2|_{T^i_{\mu}}$ is a covering map, and $T^i_{\mu}$ is a cover of $\T^2$. The only possible covers of $\T^2$ are homeomorphic to $\R^2$, $S^1 \times \R$ and $\T^2$. Using that $T^i_{\mu}$ is compact, we conclude that $T^i_{\mu}$ is actually a two torus. 
\end{proof}

\section{The proof of Theorem \ref{thm.sinvariance}: the $s$-invariance of the center conditional measures}\label{subsection.sinvariance}

The goal of this section is to prove Theorem \ref{thm.sinvariance}. Recall that for a partially hyperbolic diffeomorphism $f$ and $\mu$ an $f$-inviariant measure, we defined in (\ref{eq.upartialentropy}) the $\mu$-partial entropy along $\mathcal{F}^{uu}$, which is given by
\[
h_{\mu}(f,\mathcal{F}^{uu}) = -\displaystyle \int \log \mu^{uu}_p(f^{-1} \xi^{uu}(p)) d\mu(p), 
\]
where $\xi^{uu}$ is any $\mu$-measurable partition subordinated to $\mathcal{F}^{uu}$. 

Let us first see how the proof of Theorem \ref{thm.sinvariance} is reduced into proving the following theorem:
\begin{theorem}
\label{thm.entropys}
Let $g\in \mathrm{Sk}^2(\T^2\times \T^2)$ be a partially hyperbolic, center bunched skew product. Suppose that there exists a constant $\theta \in (0,1)$ such that 
\begin{equation}
\label{eq.holderconditionthm}
E^{uu}_g \textrm{ is $\theta$-H\"older and } \|Dg|_{E^{ss}}\|^{\theta} < m(Dg|_{E^c}).
\end{equation}
 Suppose that $\mu$ is an ergodic $u$-Gibbs measure that verifies:
\begin{enumerate}
\item $\mu$ has both a positive and a negative Lyapunov exponent along the center;

\item $\mu$ has atomic center disintegration;

\item for $\mu$-almost every $p\in \T^4$, and for Lebesgue almost every point $q$ in $W^{uu}_{loc}(p)$
\[
E^-_{g,q} \neq DH^u_{p,q}(p)E^-_{g,p}.
\]
\end{enumerate}
Let $\nu$ be the unique SRB-measure on $\T^2$ for the Anosov diffeomorphism $g_2$. Then 
\begin{equation}\label{eq.entropystable}
h_{\mu}(g^{-1}, \mathcal{F}^{ss}) = h_{\nu}(g_2).
\end{equation}
\end{theorem}

\begin{proof}[Proof of Theorem \ref{thm.sinvariance} assuming Theorem \ref{thm.entropys}]
Let $g$ be a partially hyperbolic skew product and $\mu$ be an ergodic $u$-Gibbs measure verifying the hypothesis of Theorem \ref{thm.sinvariance}. By Theorem \ref{thm.entropys}, $\mu$ verifies $h_{\mu}(g^{-1}, \mathcal{F}^{ss}) = h_{\nu}(g_2) = h_{\nu}(g_2^{-1})$. Applying the invariance principle (Theorem \ref{thm.tahzibiyanginvariance}), we have that $\mu$ is an $s$-state projecting on $\nu$. Proposition \ref{prop.uinvariancemeasure} applied to $g^{-1}$ then implies the conclusion of Theorem \ref{thm.sinvariance}.

\end{proof}

\begin{remark}
For the rest of the section we fix
\begin{itemize}
\item $g$ a partially hyperbolic skew product and
\item $\mu$ an ergodic $u$-Gibbs measure for $g$ that verifies the hypothesis of Theorem \ref{thm.entropys}.

\end{itemize}
Let 
\begin{itemize}
\item $\nu$ be the unique SRB measure for $g_2$ on $\T^2$, such that $(\pi_2)_* \mu = \nu$, and
\item $\lambda^+$ and $\lambda^-$ be the positive and negative center expoenents of $\mu$, respectively. 
\item Fix $\xi^{uu}_2$  a $\nu$-measurable partition of $\T^2$ which is subordinated to $\mathcal{F}^{uu}_2$ (the unstable foliation of the Anosov system $g_2$). 
\item Let $\xi^{uu}$ be a $\mu$-measurable partition which is subordinated to $\mathcal{F}^{uu}$, with the property that $\pi_2(\xi^{uu}(p)) = \xi^{uu}_2(\pi_2(p))$. Furthermore, we may assume that each element of $\xi^{uu}$ has small diameter.
\end{itemize}
\end{remark}

The proof of Theorem \ref{thm.entropys} follows closely the proof of Theorem $4.8$ in \cite{brownhertz}, with some adaptations.  Brown-Rodriguez Hertz's proof is based on the ``exponential drift'' arguments that were introduced in \cite{benoistquint1} and its modified version from \cite{eskinm}.   Since the proof of Theorem $4.8$ in \cite{brownhertz} is technical and to make this work more self contained, we repeat most of the proof here, with the necessary adaptations.  For the proof of some of the lemmas we will refer the reader to \cite{brownhertz}.

\subsection{Sketch of the proof of Theorem \ref{thm.entropys}}
Let us describe the idea of the proof as well as the main points where our setting differs from the random product setting. 

Suppose that $\mu$ is a $u$-Gibbs measure verifying the hypothesis Theorem \ref{thm.entropys}.  We fix a compact set $K$ of large measure such that the map $p \mapsto \mu^c_p$ is continuous. Since $\mu^c_p$ is atomic, in this set, we can fix a constant $\varepsilon>0$ such that any two atoms in a center leaf have distance greater than $\varepsilon$.

Suppose that the conclusion of the theorem does not hold, then we obtain that conditional measures along the $2$-dimensional Pesin stable manifolds are not supported in a single strong stable leaf. In particular,  we can find two points $p,q$ in the same $2$-dimensional stable manifold with $d(p,q) := \delta<< \varepsilon$ such that they don't belong to the same strong stable leaf, $p$ is an atom of $\mu^c_p$ and $q$ is an atom for $\mu^c_q$.   The idea of the argument is to find a point $p_0 \in K$ such that $\mu^c_{p_0}$ has two atoms with distance smaller than $\varepsilon$, which will give a contradiction with the choice of $\varepsilon$ and $K$.  
%
%

We want to find two sequences of times $l_j \to +\infty$ and $\tau_j \to +\infty$ such that for each $l_j$, we can choose two points $\overline{p}_j\in W^{uu}_{loc}(g^{l_j}(p))$ and $\overline{q}_j \in W^{uu}_{loc}(g^{l_j}(q))$ that verify the following:
\begin{enumerate}
\item $\overline{q}_j \in W^{cs}(\overline{p}_j)$;
\item  let $\overline{w}_j = H^s_{\overline{q}_j, \overline{p}_j}(W^-_{loc}(\overline{q}_j)) \cap W^+_{loc}(\overline{p}_j)$, then  $d(\overline{p}_j, \overline{w}_j)\approx d(g^{l_j}(p), g^{l_j}(q))$;
\item $d(g^{\tau_j}(\overline{p}_j), g^{\tau_j}(\overline{q}_j)) \approx \delta.$ 
\end{enumerate}

Then we prove that, up to a subsequence,  $g^{\tau_j}(\overline{p}_j)$ and  $g^{\tau_j}(\overline{q}_j)$ converge to $p_0$ and $ q_0$ which are atoms of $\mu^c_{p_0}$.  Item $3$ above implies that $d(p_0,q_0) \approx \delta <<\varepsilon$.  Then we show that we can do this so that $p_0$ belongs to the set of points for which the distance of any two atoms of $\mu^c_{p_0}$ is greater than $\varepsilon$ and this will give the contradiction.  

Let me remark that the hypothesis that $E^-$ is not $DH^u$ invariant appears to obtain item $2$, and item $2$ is used to prove item $3$. 

This is the strategy used in the proof of Theorem $4.8$ from \cite{brownhertz}.  As we follow this proof, let us mention the two points where more adaptation is needed in our setting.

\begin{itemize}
\item The main difficult in the strategy is to work this construction so that all the points involved belong to some ``good'' set.  To achieve that, in their proof, they work with the suspension flow and a reparametrization of this flow. This reparametrized flow is convenient because it gives precise times where a certain expansion is observed along the Oseledets unstable direction. One of the key tools they use to control these returns to the ``good'' set is a martingale convergence argument for this reparametrized flow (see Sections $9.3$ and $9.4$ in \cite{brownhertz}), in which they apply the reverse martingale convergence theorem. To apply this theorem, they need that the reparametrized flow verifies some measurability condition. 

In our setting, we use $DH^u$ to adjust the definition of the reparametrized flow so that the measurability condition is satisfied.  This is done in Section \ref{subsectionmartingale}.  Let me remark that this is similar to the change of coordinates done in Section \ref{subsec.changecoordinates}.

\item The other point where some adaptation is needed is to obtain item $2$ above. Let us briefly explain Brown-Rodriguez Hertz's proof of item $2$. Let $\overline{z}_j = H^s_{\overline{q}_j, \overline{p}_j}(W^-_{loc}(\overline{q}_j)) \cap H^u_{g^{l_j}(p), \overline{p}_j}(W^-_{loc}(g^{l_j}(p)))$. Let $q^*:= H^s_{q,p}(q)$ and observe that $q^* \in W^-_{loc}(p)$.  If we had local $su$-integrability, we would have that $\overline{z}_j = H^{u}_{g^{l_j}(p), \overline{p}_j}(g^{l_j}(q^*))$, since $\overline{q}_j = H^u_{q_j, \overline{q}_j} \circ H^s_{p_j, q_j}(g^{l_j}(q^*))$.  Then, by a geometrical argument, we could conclude that $d(\overline{p}_j, \overline{w}_j) \approx d(\overline{p}_j,  H^{u}_{g^{l_j}(p), \overline{p}_j}(g^{l_j}(q^*)) \approx d(g^{l_j}(p), g^{l_j}(q^*)).$ In the random product setting considered in \cite{brownhertz}, this local ``joint integrability'' is satisfied, since these holonomies maps are just the identity between fibers. 

In our setting, to obtain the estimate in item $2$, we need to use two ingredients: the unstable foliation is $\theta$-H\"older, and $\|Dg|_{E^{ss}}\|^{\theta} < m(Dg|_{E^c})$. This is where condition \eqref{eq.holderconditionthm} comes in. The estimate needed for item $2$ is obtained in Lemma \ref{lem.distancecontrol}.

\end{itemize}

\subsection{Pesin Theory and parametrization of invariant manifolds} \label{subsec.pesinthr}
From now on fix a constant $0< \varepsilon_0 \ll \min\{1,-\lambda^-, \lambda^+\} $. We will be interested in obtaining certain estimates for stable and unstable manifolds along the center. In our setting, these will correspond to curves contained in horizontal tori.

On $\R^2$ consider the the standard basis and for any $v\in \R^2$ write $v=v_1 + v_2$. Consider the metric on $\R^2$ given by $|v| = \max\{|v_1|, |v_2|\}$. For any $l>0$ write $\R^2(l)$ to be the ball of radius $l$ centered at the origin for this metric. Fix $0<\varepsilon_1 < \varepsilon_0$. There exists a measurable function $l:\T^2 \times \T^2 \to [1, +\infty)$, and a $\mu$-full measurable set $\Lambda$ such that:
\begin{enumerate}
\item For each point $p=(p_1, p_2) \in \Lambda$, there is a neighborhood $U_p \subset \T^2 \times \{p_2\}$ of $p_1$, and a diffeomorphism $\phi_p: U_p \to \R^2(l(p)^{-1})$ with:
\begin{enumerate}
\item $\phi_{p}(p_1) = 0$;
\item $D\phi_{p}(p_1) E^-_{p} = \R \times \{0\}$;
\item $D\phi_{p}(p_1) E^+_{p} = \{0\} \times \R$.
\end{enumerate}
\item Let 
\[
\tilde{g}_p = \phi_{g(p)} \circ g \circ \phi^{-1}_p \textrm{ and } \tilde{g}^{-1}_p = \phi_{g^{-1}(p)} \circ g^{-1} \circ \phi^{-1}_p.
\]
On the domain of definition of $\tilde{g}_p$, we have:
\begin{enumerate}
\item $\tilde{g}_p(0) = 0$;
\item $ D\tilde{g}_p(0) = 
\begin{pmatrix}
\beta^- & 0\\
0 & \beta^+
\end{pmatrix}$, where $\beta^- \in (e^{\lambda^--\varepsilon_1}, e^{\lambda^- + \varepsilon_1})$, and $\beta^+ \in (e^{\lambda^+-\varepsilon_1}, e^{\lambda^+ + \varepsilon_1})$. 
\item 
Writing $\mathrm{Lip}(.)$ the Lipschitz constant of a function in its domain of definition, we have
 $\mathrm{Lip}(\tilde{g}_p - D\tilde{g}_p(0))< \varepsilon_1$;
\item $\mathrm{Lip}(D\tilde{g}_p(0)) < l(p)$;
\item similar property holds for $\tilde{g}^{-1}_p$.
\end{enumerate}
\item There exists an uniform $k_0$ such that $k_0^{-1} \leq \mathrm{Lip}(\phi_p) \leq l(p)$;
\item $l(g^n(p)) < e^{|n|\varepsilon_1} l(p)$, for any $n\in \Z$.
\end{enumerate}
The diffeomorphisms $\phi$ above are called \textbf{Lyapunov charts}. Its construction can be found for instance in the appendix of \cite{ledrappieryoung1}.

We will also use a more quantified statement of the Pesin's stable manifold theorem. Let $\R^- := \R \times \{0\}$ and $\R^+ := \{0\} \times \R$.

\begin{theorem}[Local stable manifold theorem]\label{thm.localstablemfd}
For each $p\in \Lambda$, there exists a $C^{2}$-function $\varphi^-_p: \R^-(l(p)^{-1}) \to \R^+(l(p)^{-1})$ such that:
\begin{enumerate}
\item $\varphi^-_p(0) = 0$;
\item $D\varphi^-_p(0) = 0$;
\item $\|D\varphi^{-}_p\|< \frac{1}{3}$;
\item $\tilde{g}_p(\mathrm{graph}(\varphi^-_p)) \subset \mathrm{graph}(\varphi^-_{g(p)}) \subset \R^2(l(g(p))^{-1})$;
\item setting $W^-_{loc}(p) := \phi^{-1}_p(\mathrm{graph}(\varphi^-_p))$, we have that 
\begin{enumerate}
\item $g(W^-_{loc}(p)) \subset W^-_{loc}(g(p))$;
\item for any $z,y $ in $W^-_{loc}(p)$, and $n\geq 0$, we have
\[
d(g^n(z), g^n(y)) \leq l(p) k_0 e^{(\lambda^- + 2\varepsilon_1)n} d(z,y).
\]
\end{enumerate}
\end{enumerate}
Similarly, there exists a $C^2$-function $\varphi^+_p$ which will define the local unstable manifold. 
\end{theorem}
We may define the global stable manifold of $p$ by $W^-(p) := \cup_{n\geq 0} g^{-n}(W^-_{loc}(g^n(p)))$.

In our setting, for $\mu$-almost every point $p$, the stable manifold $W^{-}(p)$ is a one dimensional curve, and it can be parametrized by $\R$. We remark that this curve can also be obtained by intersecting a stable Pesin manifold, which in our setting has dimension two, with a center manifold.

For these one dimensional stable manifolds, it is convenient to use some special parametrization that conjugates the dynamics $g^n|_{W^-(p)}$ with the linear dynamics $Dg^n(p)|_{E^-_p}$. 

\begin{proposition}\label{prop.parametrizationstable}
For $\mu$-almost every $(p_1,p_2)$, and for any $(q_1, p_2)\in W^-(p)$, there exists a $C^2$-diffeomorphism
\[
h^-_{(q_1,p_2)}: W^-(p) \to T_{q_1} W^-(p),
\]
such that
\begin{enumerate}
\item Restricted to $W^-(p)$ we have
\begin{equation}\label{eq.commutationparametrization}
Dg(q_1,p_2) \circ h^-_{(q_1,p_2)} = h^-_{g(q_1,p_2)} \circ g;
\end{equation}
\item $h^-_{(q_1,p_2)}((q_1,p_2)) = 0$ and $Dh^-_{(q_1,p_2)}((q_1,p_2)) = Id$;
\end{enumerate}
\end{proposition}

The proof of Proposition \ref{prop.parametrizationstable} follows from the construction of the parametrizations that appeared in \cite{kalininkatok} Section $3.1$ (see also Proposition $6.5$ in \cite{brownhertz}).

For each $r>0$ and $p$ a point that verifies Proposition \ref{prop.parametrizationstable}, we define
\begin{equation}\label{eq.stablesizer}
W^-_r(p) := \left( h^-_p\right)^{-1}(\{v\in E^-_p: \|v\| <r\}).
\end{equation}
One obtains similarly functions $h^+_.$ and define $W^+_r(p):= (h_p)^{-1}(\{v\in E^+_p: \|v\| <r\})$.

We fix two $\mu$-measurable unitary vector fields $p \mapsto v^-_p$ and $p\mapsto v^+_p$ such that
\begin{itemize}
\item $v^*_p \in E^*_p$, for $*=-,+$;
\item for each $p$ and $q\in \xi^{uu}(p)$ we have that $v^+_q = \frac{DH_{p,q}(p).v^+_p}{\|DH_{p,q}(p).v^+_p\|}$; 
\end{itemize}

Using these measurable vector fields, we parametrize the stable and unstable manifolds by
\begin{equation}
\label{eq.parametrizationstableunstable}
\mathcal{I}_p^-: t \mapsto (h^-_p)^{-1}(tv^-_p) \textrm{ and } \mathcal{I}_p^+: t\mapsto (h^+_p)^{-1}(tv^+_p).
\end{equation}

It is convenient to consider another norm, called \textbf{Lyapunov norm}, where we can see contraction, or expansion, after one iterate. Let $X$ be a set of full $\mu$-measure where the Lyapunov exponents are well defined. For each $p\in X$, and $v\in E^{\sigma}_p$ consider the two-sided Lyapunov norm
\begin{equation}
\label{eq.lyapunovnorms}
\|v\|^{\sigma}_{\varepsilon_0,\pm,  p} := \left(\displaystyle \sum_{j\in \Z} \|Dg^j(p)v\|^2 e^{-2 \lambda^{\sigma}j - 2\varepsilon_0 |j|} \right)^{\frac{1}{2}}, \textrm{ where $\sigma =\{-,+\}$}
\end{equation}
and for $v\in E^+_p$, consider the one-sided Lyapunov norm
\begin{equation}\label{eq.onesidedln}
\|v\|_{\varepsilon_0,-,  p} := \left(\displaystyle \sum_{j\leq 0} \|Dg^j(p)v\|^2 e^{-2 \lambda^+j - 2\varepsilon_0 |j|} \right)^{\frac{1}{2}}
\end{equation}

For these norms, we have the following estimates (see \cite{brownhertzarxiv} for the estimate on the two-sided norm).

\begin{lemma}\label{lem.lyapunovnorms}
For $p\in X$, and $v\in E^c_p$, we have that
\[\begin{array}{rclcl}
e^{k\lambda^- - |k| \varepsilon_0}\|v\|^-_{\varepsilon_0,\pm,p} & \leq & \|Dg^k(p)v\|^-_{\varepsilon_0,\pm, g^k(p)} & \leq & e^{k \lambda^- + |k|\varepsilon_0} \|v\|^-_{\varepsilon_0, \pm, p}\\
e^{k\lambda^+ - |k| \varepsilon_0}\|v\|^+_{\varepsilon_0,\pm, p} & \leq & \|Dg^k(p)v\|^+_{\varepsilon_0,\pm, g^k(p)} & \leq & e^{k \lambda^+ + |k|\varepsilon_0} \|v\|^+_{\varepsilon_0, \pm, p}.
\end{array}
\]

If $v\in E^+_p$ then
\[
e^{k  \lambda^+ - k\varepsilon_0} \|v\|_{\varepsilon, -, p} \leq \|Dg^k	(p)v\|_{\varepsilon_0, - , g^k(p)}.
\]
\end{lemma}

The following lemma is a classical lemma in Pesin theory on the control of expansion/contraction and angle between expanding and contracting directions. 

\begin{lemma}
\label{lem.Lestimate}
There exists a measurable function $L:X\to \R$,  such that for any $p\in X$ and  $n\in \Z$:
\begin{enumerate}
\item For $v\in E^-_p$,
\[
\frac{1}{L(p)} e^{n\lambda^- -  \frac{|n|}{2} \varepsilon_0 } \|v\| \leq \|Dg^n(p) v\| \leq L(p) e^{n \lambda^- + \frac{|n|}{2} \varepsilon_0} \|v\|.
\]
\item For $v\in E^+_p$,
\[
\frac{1}{L(p)} e^{n\lambda^+ - \frac{|n|}{2} \varepsilon_0 } \|v\| \leq \|Dg^n(p) v\| \leq L(p) e^{n \lambda^+ + \frac{|n|}{2} \varepsilon_0} \|v\|.
\]
\item $\measuredangle (E^+_{g^n(p)}, E^-_{g^n(p)}) \geq \frac{1}{L(p)} e^{-|n| \varepsilon_0}$.  Furthermore, $L(g^n(p)) \leq L(p) e^{|n| \varepsilon_0}$.
\end{enumerate}
\end{lemma}

\subsection{Angles and some estimates}\label{subsec.angles}

Recall that in our setting there is a set of full $\mu$-measure such that for any two points $p$ and $q$ in this set with $q\in W^{uu}(p)$ we have that 
\begin{equation}\label{eq.hypothesis}
DH^u_{p,q}(p) E^-_p \neq E^-_q.
\end{equation}
Recall also that $X$ is the set of points of full $\mu$-measure with well defined Lyapunov exponents.

Given $\gamma_1>0$ consider $\Lambda_1$ to be the set of points $p$ such that
\begin{equation}\label{eq.lambda1}
\measuredangle (E^-_p, E^+_p) > \gamma_1,
\end{equation}
where $\measuredangle(E^-_p, E^+_p)$ is the angle for the natural riemannian metric of $\T^4$ between the subspaces $E^-_p$ and $E^+_p$. Observe that we can make the $\mu$-measure of $\Lambda_1$ arbitrarily close to $1$, by taking $\gamma_1$ small.

For $\gamma_2 \in (0, \frac{\gamma_1}{2})$ and $p\in \Lambda_1$, we define $\mathscr{A}_{\gamma_2}(p)$ to be the set of points $q\in \xi^{uu}(p)$ such that
\begin{itemize}
\item $q\in X$;
\item $\measuredangle(DH^u_{p,q}(p)E^-_p, E^-_q) > \gamma_2$;
\item $\measuredangle(E^+_q, E^-_q)> \gamma_2.$ 
\end{itemize}

Recall that for $p$ in a set of full $\mu$-measure, we defined $\mu^{uu}_p$ as the conditional measure on $\xi^{uu}(p)$ given by the disintegration of $\mu$ on the partition $\xi^{uu}$. 

For each $\gamma_1, \gamma_2$ as above and $a\in (0,1)$, define 
\begin{equation}\label{eq.definitionK}
\mathcal{A}_{\gamma_1, \gamma_2, a}:= \{p\in \Lambda_1 : \mu^{uu}_p(\mathscr{A}_{\gamma_2}(p)) > 1-a.\}.
\end{equation}
\begin{remark}
\label{remark.angleslargemeasure}
By (\ref{eq.hypothesis}), for any two numbers $a,c\in (0,1)$ there exist $\gamma_1>0$ and $\gamma_2 \in (0, \frac{\gamma_1}{2})$ sufficiently small such that $\mu(\mathcal{A}_{\gamma_1, \gamma_2, a})> 1-c$.
\end{remark}

%

By Lusin's theorem, there is a compact set $\Lambda_2\subset \T^4$ with measure arbitrarily close to $1$ such that the parametrized stable and unstable manifolds $W^-_r(p)$ and $W^+_r(p)$ vary continuously in the $C^1$-topology, for $p\in \Lambda_2$, on the space of embeddings $C^1([-r,r], \T^2)$ for any $r\in (0,1)$.

Given $\theta'\in (0, \pi)$ and a one-dimensional space $E$ in $\R^2$, let $\mathscr{C}_{\theta'}(E)$ be the cone centered in $E$ with angle $\theta'$. In what follows, we will consider $\exp$ to be the exponential map on $\T^2$, and we identify every center manifold with the two torus $\T^2$.  

\begin{lemma}\label{lem.closeoflinear}
Given $\theta'\in (0, \pi)$, there exist $\hat{r}_0, \hat{r}_1>0$ such that for any two points $p=(p_1,p_2)$ and $q = (q_1,q_2)$ both belonging to $\Lambda_2$, such that $d(p,q) < \hat{r}_0$ and $q_2 \in W_{g_2}^{ss}(p_2)$, we have:
\begin{enumerate}
\item $\exp^{-1}_{p_1}(W^*_{\hat{r}_1}(p)) \subset \mathscr{C}_{\theta'}(E^*_p)$, for $* = -$ or $+$;
\item $\exp^{-1}_{p_1}(H^s_{q_2,p_2}(W^-_{\hat{r}_1}(q))) \subset \mathscr{C}_{\theta'}(E^-_p) + \exp^{-1}_{p_1}(q_1)$.
\end{enumerate} 

\end{lemma}

\begin{proof}
Recall that on $\Lambda_2$, for $r\in (0,1)$, the map $\Lambda_2 \ni p \mapsto W^*_r(p)$ varies continuously on the space of embeddings $C^1([-r,r], \T^2)$. Hence, we may fix $\hat{r}_1$ sufficiently small such that for each $p\in \Lambda_2$ we have that $\exp^{-1}_{p_1}(W^*_{\hat{r}_1}(p)) \subset \mathscr{C}_{\frac{\theta}{4}}(E^*_p)$, for $*=-$ or $+$. 

Using that $p\mapsto E^-_p$ varies continuously on $\Lambda_2$, one may take $\hat{r}_0$ sufficiently small so that if $d(p,q)< \hat{r}_0$ then $E^-_q \subset \mathscr{C}_{\frac{\theta}{4}}(E^-_p)$. Recall that if $p_2$ and $q_2$ are in the same strong stable manifold of size $1$ for $g_2$, then $d_{C^1}(H^s_{p_2,q_2}, Id) \leq C d(p_2,q_2)$, for some constant $C\geq 1$.  If $\hat{r}_0$ is sufficiently small we also have that $\exp^{-1}_{p_1}(H^s_{q_2,p_2}(W^-_{\hat{r}_1}(q))) \subset \mathscr{C}_{\frac{\theta}{2}}(DH^s_{q_2,p_2}(q)E^-_p)$, and $DH^s_{q_2,p_2}(q) E^-_q \subset \mathscr{C}_{\frac{\theta}{2}}(E^-_p)$. This implies the second item of the lemma.
\end{proof}

 For each $l_0$, we may consider the set $\Lambda_3\subset \Lambda_2$ of points having the value $l(p)$ bounded above by $l_0$, where $l:\T^2\times \T^2 \to [1, +\infty)$ is the function defined in Section \ref{subsec.pesinthr}. We may also fix $\tilde{r}_0>0$ and $\tilde{r}_1>0$ small enough such that for each $p\in \Lambda_3$, we have

\begin{enumerate}[label=(\alph*)]
\item $W^-_{\tilde{r}_1}(p) \subset W^-_{loc}(p)$, where $W^-_{loc}(p)$ is the local stable manifold defined in Theorem \ref{thm.localstablemfd};
\item for $q=(q_1,q_2)\in \Lambda_2$ such that $q_2 \in W^{ss}_{g_2}(p_2)$, if $d(p,q)<\tilde{r}_0$ then $\phi_p^{-1}(H^s_{q_2,p_2}(W^-_{\tilde{r_1}}(q))$ is contained in the graph of a $1$-Lipschitz function $G:D\subset \R^- \to \R^+$, where $D\subset \R^-(l_0^{-1})$. 
\end{enumerate}
The second point above follows from combining the estimates from Theorem \ref{thm.localstablemfd} and Lemma \ref{lem.closeoflinear}.

Observe that by taking $l_0$ large, the set $\Lambda_3$ has $\mu$-measure arbitrarily close to $\mu(\Lambda_2).$

\begin{lemma}\label{lemma.manycontrolsangle}
For every $\gamma_1>0$, $\gamma_2 \in (0, \frac{\gamma_1}{2})$ and $\Lambda_3 \subset \Lambda_2 \subset \Lambda_1$ as above, there exist a measurable set $\Lambda'\subset \Lambda_3$, with $\mu(\Lambda')$ arbitrarily close  to $\mu(\Lambda_3)$, constants $r_0 \in (0,\tilde{r}_0)$, $ r_1 \in (0,\tilde{r}_1)$, and $C_1, C_2, C_3>1$, with the following properties:
For each $p \in \Lambda'$ we have
\begin{enumerate}
\item $\frac{1}{C_1} d(p,w) \leq \|h^*_p(w)\| \leq C_1 d(p,w)$, for every $w\in W^*_{r_1}(p)$, with $* = -$ or $+$, where $h^*$ is given by Proposition \ref{prop.parametrizationstable};
\end{enumerate}
Let $p = (p_1,p_2) \in \Lambda'$, $p'= (p_1', p_2') \in \mathscr{A}_{\gamma_2}(p) \cap \Lambda'$ and $q= (q_1,q_2)\in \Lambda'$ such that $q\in W^{cs}_{r_0}(p')$ and $d(p',q)< r_0$. Then,
\begin{enumerate}
\setcounter{enumi}{1}
\item $W^+_{r_1}(p') \cap H^s_{q_2,p_2'}(W^-_{r_1}(q))$ is a single point $w'$, furthermore this intersection is transverse with angles uniformly bounded away from zero inside $T_x\T^2$; 

\item $H^u_{p_2,p_2'}(W^-_{r_1}(p)) \cap H^s_{q_2,p_2'}(W^-_{r_1}(q))$ is a unique point $z'$, this intersection is transverse with angle uniformly bounded from below.

\item For $z'$ and $w'$ as above, 
\[
\frac{1}{C_2} d(p',z')\leq d(p',w') \leq C_2d(p',z');
\]

\item $\displaystyle \frac{1}{C_3} \leq \frac{\|Dg^n(w')|_{T_{w'} W^+_{r_1}(p')}\|}{\|Dg^n(q)|_{E^+_q}\|} \leq C_3$, for every $n\geq 0$.

\item For any $x \in \Lambda'$ and $y \in W^+_{r_1}(x)$ we have that 
\[
\displaystyle \frac{1}{C_3} \leq \frac{\|Dg^{-n}(y)|_{T_yW^+_{r_1}(x)}\|}{\|Dg^{-n}(x)|_{E^+_x}\|} \leq C_3, \textrm{ for every $n\geq 0$.}
\]

\end{enumerate}

\end{lemma}
\begin{proof}

Items $1-3$ follow from $C^1$-topology, Lusin's theorem and using that in our setting there exists a constant $C>1$ such that $x_2 \in W^{ss}_{g_2}(y_2)$ with $d(x_2,y_2)<1$, we have $d(H^s_{x_2,y_2}, Id)< C . d(x_2,y_2)$ (similar estimate holds for unstable holonomies). One can also conclude that for any $r_1>0$ small, if $r_0>0$ is sufficiently small the conclusions hold.

To prove item $4$, by items $2$ and $3$ above, the angles of the intersections $W^+_{r_1}(p') \cap H^s_{q_2,p_2'}(W^-_{r_1}(q))$ and $H^u_{p_2,p_2'}(W^-_{r_1}(p)) \cap H^s_{q_2,p_2'}(W^-_{\tilde{r}_1}(q))$ are uniformly bounded away from zero (depending on $\gamma_1, \gamma_2$ that are fixed). 

Fix $\theta'>0$ small. We may suppose that $r_1 < \hat{r}_1$ and $r_0<\hat{r}_0$ are sufficiently small so that  items $1-3$ remain valid, where $\hat{r}_0, \hat{r}_1$ are the constants given by Lemma \ref{lem.closeoflinear}. From the conclusion of Lemma \ref{lem.closeoflinear}, we obtain that for any $x_1 \in W^+_{r_1}(p')$, $x_2\in  H^s_{q_2,p_2'}(W^-_{r_1}(q))$, and $x_3\in H^u_{p_2,p_2'}(W^+_{r_1}(p))$, we have that the angles between $T_{x_1}W^+_{r_1}(p')$, $T_{x_2}H^s_{q_2,p_2'}(W^-_{r_1}(q))$, and $T_{x_3}H^u_{p_2,p_2'}(W^+_{r_1}(p)) $ are uniformly bounded from below. The estimate in item $4$ then follows from the law of sines (see Figure \ref{figure1}).

\begin{figure}[h]
\centering
\includegraphics[scale= 0.38]{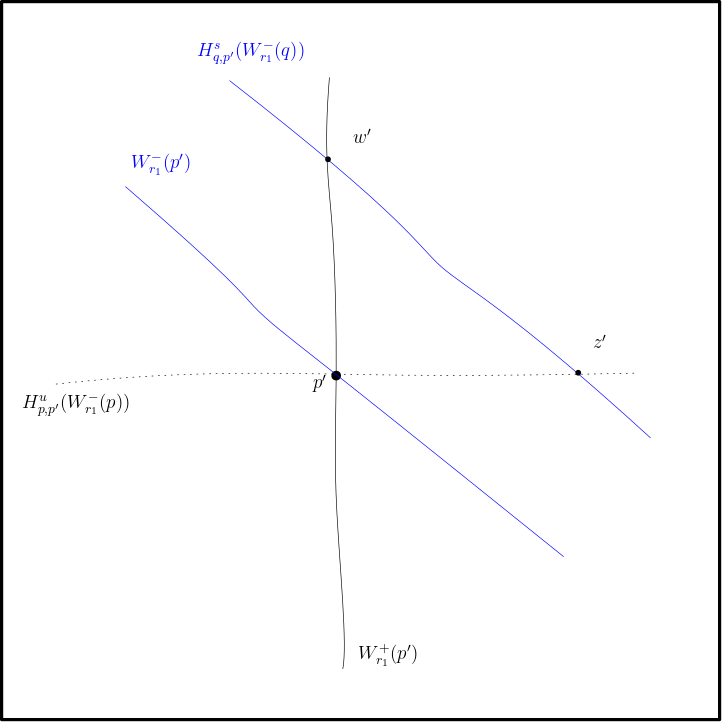}
\caption{Comparing $d(p',w')$ and $d(p',z')$.}
\label{figure1}
\end{figure}

Since in $\Lambda_3$ the function $l(.)$ is bounded by $l_0$, by the properties given by the Lyapunov charts, we obtain that there exists a constant $C(l_0)$ such that for any $n\geq 0$,
\[
\displaystyle d(T_{g^{-n}(y)}W^+_{r_1}(g^{-n}(x)), E^+_{g^{-n}(x)}) \leq C(l_0) \left( \frac{e^{\lambda^- + \varepsilon_1}}{e^{\lambda^+ - \varepsilon_1}} \right)^n e^{2\varepsilon_1 n}.
\]

We also have that $d(g^{-n}(x), g^{-n}(y))<l_0 k_0 e^{(-\lambda^+ + 2\varepsilon_1)n} d(x,y)$ for $n\geq 0$. In what follows denote $T_{g^{-j}(y)}W^+(g^{-j}(x))$ by $E_{y_{-j}}$, $x_{-j} = g^{-j}(x)$ and $y_{-j} = g^{-j}(y)$, for any $j\in \N$. Observe that
\[
\|Dg^{-1}(y_{-j})|_{E_{y_{-j}}}\| \leq \|g\|_{C^2}\|Dg^{-1}(x_{-j})|_{E^+_{x_{-j}}}\|\max\{d(x_{-j},y_{-j}), d(E_{y_{-j}}, E^+_{x_{-j}})\}.
\]
Since 
\[
\displaystyle \frac{\|Dg^{-n}(y)|_{E_{y_0}}\|}{\|Dg^{-n}(x)|_{E^+_{x}}\|}   = \prod_{j=0}^{n-1}\frac{  \|Dg^{-1}(y_{-j})|_{E_{y_{-j}}}\|}{\|Dg^{-1}(x_{-j})|_{E^+_{x_{-j}}}\|},
\]
the result then follows combining the estimates above.

The proof of item $5$ is similar to the proof of item $6$. One uses the information that the future orbit of the points $w'$ and $q$ converge exponentially and that the respective tangent directions considered also converge uniformly exponentially fast  on $\Lambda_3$.

%
%
%
%
%
%

\end{proof}

\subsection{Reparametrized suspension flow, stopping times and the Martingale convergence argument}\label{subsec.repflow}

\subsubsection{The suspension flow}
It will be convenient for us to work with the suspension flow associated with $g$ and a reparametrization of it. Let us recall the definition of the suspension flow.

Consider the $5$-dimensional manifold $\tilde{M} = \T^4 \times \R$. On $\tilde{M}$ consider the following equivalence relation
\[
(p, l) \sim (g(p), l-1).
\]
Let $M = \tilde{M} / \sim$ be the quotient manifold, and consider the flow $\Phi_t: M \to M$ defined by $\Phi_t([p, l]) = [p, l+t]$, where $[p,l]$ denotes the equivalence class of the point $(p,l) \in \tilde{M}$. For $\zeta = [p,l] \in M$, we consider the center fiber $\T^2_{\zeta}$ which is given by the projection of $\T^2 \times \{p_2\} \times \{l\} \subset \tilde{M}$ into $M$. The fiber $\T^2_{\zeta}$ is naturally identified with $\T^2$. We will use the coordinates on $M$ induced by $\T^4 \times [0,1)$.

Consider the measure on $\tilde{M}$ defined by $\tilde{\omega}:= \mu \times \mathrm{Leb}_{\R}$, where $\mu$ is the measure on $\T^4$ and $\mathrm{Leb}_\R$ is the usual Lebesgue measure on $\R$. This measure projects to a probability measure $\omega$ on $M$, which in the coordinates $\T^4 \times [0,1)$  can be written as $d\omega(p,l) = d\mu(p) dl$. Observe that this measure is invariant by the flow $\Phi_t$. Recall that $X$ is the set of full $\mu$-measure on $\T^4$ where the Lyapunov exponents are well defined. Let $Y$ be the projection on $M$ of the set $X\times \R$ defined on $\tilde{M}$. This is a set of full $\omega$-measure and for each $\zeta = [p,l] \in Y$, we may define the Oseledets splitting of the center direction $T_{\zeta}\T^2_{\zeta} = E^-_{\zeta} \oplus E^+_{\zeta}$, where $E^*_{\zeta} = D\Phi_{t}(\zeta) . E^{\sigma}_p $, for $* = -,+$.

We can naturally extend to $Y$ the vector fields $p \mapsto v_p^*$ (defined in Section \ref{subsec.pesinthr}) and the parametrizations defined in \eqref{eq.parametrizationstableunstable}.  We can also extend to $Y$ the Lyapunov norms defined in \eqref{eq.lyapunovnorms} in the following way. Let $\zeta = [p,l] \in Y$, then for any $v\in E^{\sigma}_{\zeta}$ we define 
\begin{equation}
\label{eq.lyapunovnormflow}
\|v\|^*_{\varepsilon_0, \pm, \zeta} := \displaystyle \left( \|v\|^*_{\varepsilon_0, \pm, p}\right)^{1-l} \left( \|Dg(p) v\|_{\varepsilon_0, \pm, g(p)}^{*}\right)^l, \textrm{ for $* = -,+$.}
\end{equation}
We define in a  similar way the one-sided norm $\|v\|_{\varepsilon_0, -, \zeta}$, for $v\in E^+_{\zeta}$. 
This norm allows us to have expansion or contraction varying continuously with the time. In particular, from the estimates of Lemma \ref{lem.lyapunovnorms}, we obtain:

\begin{lemma}
\label{lem.lyapunovnormsflows}
For $\zeta = [p,l] \in Y$, for $v\in T_{\zeta} \T^2_{\zeta}$, and for any $t\in \R$ we have that
\[\begin{array}{rclcl}
e^{t\lambda^- - |t| \varepsilon_0}\|v\|^-_{\varepsilon_0,\pm,\zeta} & \leq & \|D\Phi_t(\zeta)v\|^-_{\varepsilon_0,\pm, \Phi_t(\zeta)} & \leq & e^{t \lambda^- + |t|\varepsilon_0} \|v\|^-_{\varepsilon_0, \pm, \zeta}\\
e^{t\lambda^+ - |t| \varepsilon_0}\|v\|^+_{\varepsilon_0,\pm, \zeta} & \leq & \|D\Phi_t(\zeta)v\|^+_{\varepsilon_0,\pm, \Phi_t(\zeta)} & \leq & e^{t\lambda^+ + |t|\varepsilon_0} \|v\|^+_{\varepsilon_0, \pm, \zeta}.
\end{array}
\]

If $v\in E^+_\zeta$ then
\[
e^{t  \lambda^+ - |t|\varepsilon_0} \|v\|_{\varepsilon, -, \zeta} \leq \|D\Phi_t(\zeta)v\|_{\varepsilon_0, - , \Phi_t(\zeta)}.
\]
\end{lemma}
Recall that $L(.)$ is the function from Lemma \ref{lem.Lestimate}.  The proof of the following lemma can be obtained by a simple adaptation of the proof of Lemma 9.4 from \cite{brownhertz}.
\begin{lemma}
\label{lem.onemoreestimate}
For $\omega$-almost every $\zeta = [p,l]$, for any $v\in E^+_{\zeta}$
\[
\|v\| \leq \|v\|_{\varepsilon_0, - , \zeta} \leq L(p) \|g\|_{C^1}e^{\varepsilon_0} (1- e^{-\varepsilon_0})^{\frac{1}{2}} \|v\|.
\]
In particular, by defining $\hat{L}(\zeta) = L(p) \|g\|_{C^1} e^{\varepsilon_0} (1- e^{-\varepsilon_0})^{\frac{1}{2}}$, we have
\[
\hat{L}(\Phi_t(\zeta)) \leq e^{2 \varepsilon_0 (|t| + 1)} \hat{L}(\zeta)
\]
and 
\begin{equation}
\label{eq.thatineed}
\frac{1}{\hat{L}(\zeta)} \|D\Phi_t(\zeta)|_{E^+_{\zeta}}\| \leq \|D\Phi_t(\zeta)|_{E^+_{\zeta}}\|_{\varepsilon_0, -} \leq e^{2\varepsilon_0 (|t| + 1)} \hat{L}(\zeta) \|D\Phi_t(\zeta)|_{E^+_{\zeta}}\|.
\end{equation}

\end{lemma}

\subsubsection{The reparametrized flow and the Martingale convergence argument}\label{subsectionmartingale} \label{subsec.reparametrized}

From the partition $\xi^{uu}$ we may consider the partition $\tx^{uu}$ obtained by the sets of the form $[\xi, \{l\}]$ in $M$, where $\xi \in \xi^{uu}$ and $l\in [0,1)$. This forms an $\omega$-measurable partition.  For each $\xi \in \xi^{uu}$ fix $\xi^{uu}_p \in \xi$ an Oseledets regular point for $\mu$ and for $\tx = [\xi, l]\in \mathcal{P}$  let $\zeta_{\tx}$ be the point $[p_{\xi}, l]$.  Given two points $\zeta = [p,l]$ and $\eta = [q,l]$ such that $q\in W^{uu}(p)$, we write $H^u_{\zeta, \eta}$ as the map induced by $H^u_{p,q}$ in the first coordinate and fixing the $l$ coordinate. 

Consider the $\omega$-measurable bundle $V$ over $M$ such that for $\omega$-almost every point $\zeta$, the fiber $V_\zeta$ is given by $E^+_{\tx^{uu}_{\zeta}}$. This bundle can be obtained from the bundle $E^+$ over $M$ in the following way: for each $\zeta \in M$ we identify $E^+_\zeta$ with $E^+_{\tx^{uu}_{\zeta}}$ using the holonomy $DH^u_{\zeta, \tx^{uu}_{\zeta}}(\zeta)$, recall that $E^+$ is $DH^u$-invariant.On the bundle $V$ we may consider the linear cocycle over $\Phi_t$ given by $G_t(\zeta)v = DH^u_{\Phi_t(\zeta), \tx^{uu}_{\Phi_t(\zeta)}}(\Phi_t(\zeta))\circ D\Phi_t(\zeta)\circ  DH^u_{\tx^{uu}_\zeta, \zeta}(\tx^{uu}_{\zeta}) v$, where $v\in V_{\zeta}$. 

\begin{claim}
\label{claim.measurable}
For any $t>0$, and for any two points $\zeta, \eta$ such that $\eta \in \tx^{uu}(\zeta)$ it holds that $G_{-t}(\zeta) = G_{-t}(\eta)$. In other words, $G_{-t}(.)$  is constant on elements of the partition $\tx^{uu}$. 
\end{claim} 
\begin{proof}
Observe that 
\[
\begin{array}{rcl}
G_{-t}(\zeta) = \left( G_t(\Phi_{-t}(\zeta)) \right)^{-1} & = & DH^u_{\Phi_{-t}(\zeta), \tx^{uu}_{\Phi_{-t}(\zeta)}} (\Phi_{-t}(\zeta)) \circ D\Phi_{-t}(\zeta) \circ DH^u_{\tx^{uu}_\zeta, \zeta}(\tx^{uu}_\zeta) \\
& = & DH^u_{\Phi_{-t}(\zeta), \tx^{uu}_{\Phi_{-t}(\zeta)}} (\Phi_{-t}(\zeta)) \circ DH^u_{\Phi_{-t}(\tx^{uu}_\zeta), \Phi_{-t}(\zeta)} \circ D\Phi_{-t}(\tx^{uu}_\zeta)\\
 & = & DH^u_{\Phi_{-t}(\tx^{uu}_{\zeta}), \tx^{uu}_{\Phi_{-t}(\zeta)}} \circ D\Phi_{-t}(\tx^{uu}_{\zeta}).
\end{array}
\]

Recall that the partition $\tx^{uu}$ verifies $\Phi_{-t}(\tx^{uu}(\zeta)) \subset \tx^{uu}(\Phi_{-t}(\zeta))$. Thus,  for any $\eta \in \tx^{uu}(\zeta)$ we have $\tx^{uu}_{\Phi_{-t}(\eta)} = \tx^{uu}_{\Phi_{-t}(\zeta)}$ and we conclude that $G_{-t}(\eta) = G_{-t}(\zeta)$.\qedhere
\end{proof}

For $\omega$-almost every $\zeta = [p,l]$ and any vector $v\in V_\zeta$ define
\begin{equation}
\label{eq.normrep}
\|v\|^V_{\varepsilon_0, -, \zeta}:= \left(\|v\|_{\varepsilon_0, -, p}^V\right)^{1-l} \left( \|G_1([p,0]) v\|^V_{\varepsilon, -, g(p)}\right)^l,
\end{equation}
where 
\[
\|v\|^V_{\varepsilon_0, - , p} := \displaystyle \left( \sum_{j\leq 0, j\in \Z} \|G_{j}([p,0]) v\|^2 e^{-2 \lambda^+ j - 2\varepsilon_0|j|} \right)^{\frac{1}{2}}.
\]

Define
\begin{equation}
\label{eq.reparametrization}
\kappa_{\zeta}(t) := \log \|G_t(\zeta)\|^V_{\varepsilon_0, -}.
\end{equation}

Observe that $\|v\|_{\varepsilon_0, -, \zeta}^V$ is a one-sided Lyapunov norm for the linear cocycle $G_t$. From the construction of $G_t$ it is easy to see that it the Lyapunov exponent of $G_t$ is $\lambda^+$. In particular, for every $t\in \R$ we have the estimate
\[
t\lambda^+ - |t|\varepsilon_0 \leq \log\|G_t(\zeta)\|^V_{\varepsilon_0, -}.
\] 
We conclude that $\kappa_\zeta(.)$ is an increasing homeomorphism of $\R$. Moreover, the function $\kappa_{\zeta}$ verifies the cocycle condition, that is, $\kappa_\zeta(t_1 + t_2 ) = \kappa_{\Phi_{t_1}(\zeta)}(t_2) + \kappa_\zeta(t_1)$.  

\begin{claim}
\label{claim.repflowexp}
There exists a uniform constant $\hat{C}>1$ such that for $\omega$-a.e. $\zeta$ and any $s\in \R$, 
\[
\frac{1}{\hat{C}} e^s \leq \|D\Phi_{\kappa_{\zeta}^{-1}(s)} (\zeta)|_{E^+_\zeta}\|_{\varepsilon_0, -} \leq \hat{C} e^s.
\]
\end{claim}
\begin{proof}

Let $t = \kappa^{-1}_\zeta(s)$ and observe that by definition $\log \|G_t(\zeta)\|^V_{\varepsilon_0, - } = s$.  Fix a non-zero vector $v\in V_{\zeta}$, hence
\[
\displaystyle \frac{\|G_t(\zeta) v\|^V_{\varepsilon_0, -, \Phi_t(\zeta)}}{\|v\|^V_{\varepsilon_0,-, \zeta}} = e^s.
\]
Set $v^+ = DH^u_{\tx^{uu}_{\zeta}, \zeta} v$ and observe that 
\[
\|D\Phi_t(\zeta)|_{E^+_\zeta}\|_{\varepsilon_0, - } = \displaystyle \frac{\|D\Phi_t(\zeta) v^+\|_{\varepsilon_0, -, \Phi_t(\zeta)}}{\|v^+\|_{\varepsilon_0, -, \zeta}}.
\]
Recall that for any $t'\in \R$,  $G_{t'}(\zeta)v = DH^{u}_{\Phi_{t'}(\zeta), \tx^{uu}_{\Phi_{t'}(\zeta)}} (\Phi_{t'}(\zeta))\circ D\Phi_{t'}(\zeta) \circ DH^u_{\tx^{uu}_{\zeta}, \zeta}(\tx^{uu}_{\zeta})v$. In particular, $G_t(\zeta) v = DH^u_{\Phi_t(\zeta), \tx^{uu}_{\Phi_t(\zeta)}}(\Phi_t(\zeta)) \circ D\Phi_t(\zeta) v^+$. 

Since the distance between any point $q$ and $\xi^{uu}_q$ is uniformly bounded from above,  there exists a uniform constant $K>1$ such that for any non positive integer $j\leq 0$, 
\[
\frac{1}{K} \leq \frac{\|Dg^j(p)v^+\|}{\|G_j([p,0]) v\|} \leq K.
\]
The result then follows easily from the remarks above and the definitions of the norms $\|.\|^V_{\varepsilon,-,\zeta}$ and $\|.\|_{\varepsilon_0, -, \zeta}$. \qedhere

\end{proof}

Recall that $Y$ is the set of $\omega$-full measure of Oselede		ts regular points. We consider the reparametrized flow $\Psi_s: Y \to Y$, defined by
\[
\Psi_s(\zeta) = \Phi_{\kappa_{\zeta}^{-1}(s)}(\zeta).
\]


For $\zeta = [p,l] \in Y$, let $h(\zeta) = h(\xi^{uu}_p):= \log \| G_1([p,0])v\|^V_{\varepsilon_0, -, \Phi_1(\zeta)} $, where $\|v\|^V_{\varepsilon_0, -, \zeta} = 1$. By \eqref{eq.normrep}, for any $t \in [-l, 1-l)$ we have that
\begin{equation}
\label{eq.Th}
\kappa_{\zeta}(t) = th(\zeta).
\end{equation}
For $s\in \R$ such that $\frac{s}{h(\zeta)} + l \in [0,1)$, we have that $\Psi_s(\zeta) = [p, l+ \frac{s}{h(\zeta)}]$. That is, $\frac{1}{h(\xi^{uu}_\zeta)}$ gives the local 	change of speed of the flow $\Phi_t$ to obtain the flow $\Psi_s$. In particular, $\kappa_{\zeta}(t) = \int_0^t h(\Phi_{\tau}(\zeta)) d\tau$.  Observe that $h(\zeta)> \lambda^+ - \varepsilon_0$ for $\omega$-almost every $\zeta$.  We also have that $\int h(\zeta) d \omega(\zeta) < + \infty$ (see Claim 9.5 in \cite{brownhertz}) and hence the flow $\Psi_s$ preserves the probability measure $\hat{\omega}$, which in coordinates is given by 
\begin{equation}\label{eq.omegaandomegahat}
\displaystyle d\hat{\omega}(\zeta) := \frac{h(\zeta)}{ \int h(\eta) d\omega(\eta)}d\omega(\zeta). 
\end{equation}
Since the measure $\omega$ is ergodic for $\Phi_t$, we obtain that $\hat{\omega}$ is ergodic for $\Psi_s$.

%

Observe that the partition $\tx^{uu}$ is both $\omega$ and $\hat{\omega}$ measurable. Let $\tb$ be the $\sigma$-algebra generated by the partition $\tx^{uu}$.   Let $\omega^{\tb}_\zeta$ and $\hat{\omega}^{\tb}_{\zeta}$ be the conditional measures of $\omega$ and $\hat{\omega}$ with respect to the $\sigma$-algebra $\tb$. This is the same as considering the disintegrated measures of $\omega$ and $\hat{\omega}$ with respect to the measurable partition $\tx^{uu}$.  From \eqref{eq.omegaandomegahat} we obtain that
\[
d\hat{\omega}^{\tb}_\zeta(\eta) = \frac{h(\eta)}{\int h(\rho) d\omega^{\tb}_\zeta(\rho)} d\omega^{\tb}_\zeta(\eta).
\]
By construction, the function $h$ is $\tb$-measurable, since it is constant on elements of the partition $\tx^{uu}$.  Since $\omega^{\tb}_\zeta$ and $\hat{\omega}^{\tb}_\zeta$ are probability measures, we take
\begin{equation}
\label{eq.omegahatomega}
\hat{\omega}^{\tb}_\zeta = \omega^{\tb}_\zeta.
\end{equation}

The function $(\zeta, -t) \mapsto \kappa_\zeta(-t)$, where $t\geq 0$,  is $\tb$-measurable. This follows from Claim \ref{claim.measurable}. In particular, the semiflow $\Psi_{-s}$ is $\tb$ measurable.
We also have that $\Psi_s(\tb) \subset \mathcal{\tb}$ for $s\geq 0$,  where $\Psi_s(\tb):=\{\Psi_s(C): C\in \tb\}$,  this follows from the fact that $\tx^{uu}$ is decrasing for the measurable flow $\Psi_s$. Write $\tb_s = \Psi_s(\tb)$. We have that $\tb_{s} \subset \tb_{s'}$, for $s \geq s'$, and we obtain that $\{\tb_s\}_{s\geq 0}$ forms a decreasing filtration. Let $\tb_{\infty} := \bigcap_{s\geq 0} \tb_s$.

Let $\rho:Y \to \R$ be a $\omega$-integrable function, in particular it is also $\hat{\omega}$-integrable. For $\hat{\omega}$-almost every $\zeta$, define the conditional expectation $\mathbb{E}_{\hat{\omega}}(\rho| \tb_s)(\zeta)$ by
\[
\mathbb{E}_{\hat{\omega}}(\rho| \tb_s)(\zeta): = \int\rho \circ \Psi_s(\eta) d\hat{\omega}^{\tb_0}_{\Psi_{-s}(\zeta)} (\eta) = \int \rho(\eta') d((\Psi_s)_* \hat{\omega}^{\tb_0}_{\Psi_{-s}(\zeta)})(\eta').
\]
By the $\Psi_s$-invariance of $\hat{\omega}$ it is easy to conclude that $(\Psi_s)_* \hat{\omega}^{\tb_0}_{\Psi_{-s}(\zeta)} = \hat{\omega}^{\tb_s}_{\zeta}$. Furthermore, since $\Psi_{s'}(\tx^{uu})$ refines $\Psi_s(\tx^{uu})$, whenever $s \leq s'$, we can also conclude that $\mathbb{E}_{\hat{\omega}}(\mathbb{E}_{\hat{\omega}}(\rho| \tb_s)|\tb_{s'})(\eta) = \mathbb{E}_{\hat{\omega}}(\rho|\tb_{s'})(\eta)$.  Thus, $\mathbb{E}_{\hat{\omega}}(\rho|\tb_s)(.)$ defines a reverse martingale for the decreasing filtration $\{\tb_s\}_{s\geq 0}$ on $(Y, \hat{\omega})$. By the Reverse Martingale Convergence Theorem (see \cite{bookew} Theorem $5.8$) we obtain that for $\hat{\omega}$-almost every $\zeta$ we have the convergence  $\lim_{s\to +\infty} \mathbb{E}_{\hat{\omega}}(\rho|\tb_s)(\zeta) = \mathbb{E}_{\hat{\omega}}(\rho|\tb_{\infty})(\zeta)$. 

\subsection{Stopping time and bi-Lipschitz estimate}

Let $\zeta \in Y$. For $\delta>0$ and $t\in \R$, define
\[
\tau_{\zeta,\delta}(t) := \sup \left\{t'\in \R : \|D\Phi_t(\zeta)|_{E^-_{\zeta}}\|^-_{\varepsilon_0, \pm,\Phi_t(\zeta)} . \|D\Phi_{t'}(\Phi_t(\zeta))|_{E^+_{\Phi_t(\zeta)}} \|^+_{\varepsilon_0,\pm, \Phi_{t + t'}(\zeta)} \delta \leq \delta \right\}.
\]
Define $L_{\zeta, \delta}(t) := t + \tau_{\zeta, \delta}(t)$. Observe that the functions $\tau_{\zeta, \delta}:\R \to \R$ and $L_{\zeta ,\delta} :\R \to \R$ are increasing homeomorphisms. 
\begin{lemma}[\cite{brownhertz}, Lemma $9.7$]
\label{lem.bilipstop}
The functions $\tau_{\zeta,\delta}$ and $L_{\zeta,\delta}$ are bi-Lipschitz with constants uniform in $\zeta,\delta$. In particular, for $t'\geq 0$ 
\[
\begin{array}{rcccl}
\frac{-\lambda^- - 3 \varepsilon_0}{\lambda^+ + \varepsilon_0} t' & \leq & \tau_{\zeta,\delta}(t + t') - \tau_{\zeta,\delta}( t)  & \leq & \frac{-\lambda^- + 3\varepsilon_0 }{ \lambda^+ - \varepsilon_0} t',\\
\frac{ \lambda^+ - \lambda^- -2\varepsilon_0}{\lambda^+ + \varepsilon_0} t' & \leq & L_{\zeta,\delta}(t + t') - L_{\zeta, \delta}(t) & \leq & \frac{\lambda^+ - \lambda^- + 2\varepsilon_0}{\lambda^+ - \varepsilon_0} t'.  
\end{array}
\]
\end{lemma}

\subsection{Estimates for the holonomies}

We will need the following estimate on the holonomies.
\begin{lemma}
\label{lem.nonintegrabilityestimate}
Suppose that $E^{uu}_g$ is $\theta$-H\"older for some $\theta\in (0,1)$. There exists a constant $L>1$ such that the following holds true: given three points $q_2\in \T^2$, $q_2^u \in W^{uu}_{g_2,1}(q_2)$, and $q_2^s \in W^{ss}_{g_2, loc}(q_2)$; let $\tilde{q}_2$ be the unique point of the intersection between $W^{ss}_{g_2,loc}(q_2^u)$ and $W^{uu}_{g_2,2}(q_2^s)$, then for any $q_1\in \T^2$
\[
d\left(H^u_{ q_2,q_2^u}(q_1), H^s_{\tilde{q}_2,q_2^u} \circ H^u_{ q_2^s, \tilde{q}_2} \circ H^s_{ q_2, q_2^s}(q_1)\right) < L d(q_2,q_2^s)^{\theta}.
\]
\end{lemma}
 Observe that if the strong foliations were jointly integrable, then for any points as above, we would have
\[
H^u_{q_2^u, q_2}(q_1)= H^s_{q_2^u,\tilde{q}_2},
\]
and Lemma \ref{lem.nonintegrabilityestimate} would be immediate. This lemma will give a quantitative (upper) control on the non-integrability. 




\begin{proof}[Proof of Lemma \ref{lem.nonintegrabilityestimate}.]
By the triangular inequality, we have

\[
\begin{array}{l}
d\left(H^u_{ q_2,q_2^u}(q_1), H^s_{\tilde{q}_2,q_2^u} \circ H^u_{ q_2^s, \tilde{q}_2} \circ H^s_{ q_2, q_2^s}(q_1)\right)  \\
 \leq  d\left(H^u_{ q_2,q_2^u}(q_1), H^u_{ q_2^s, \tilde{q}_2} \circ H^s_{ q_2, q_2^s}(q_1)\right)
 + d\left(H^u_{ q_2^s, \tilde{q}_2} \circ H^s_{ q_2, q_2^s}(q_1), H^s_{\tilde{q}_2,q_2^u} \circ H^u_{ q_2^s, \tilde{q}_2} \circ H^s_{ q_2, q_2^s}(q_1)\right).\\
 \end{array}
\]
Since the foliation $\mathcal{F}^{uu}$ is $\theta$-H\"older and the distance between $q_2$ and $q_2^u$ is uniformly bounded form above,  there exists a uniform constant $K_1$ such that
\[
d\left(H^u_{ q_2,q_2^u}(q_1), H^u_{ q_2^s, \tilde{q}_2} \circ H^s_{ q_2, q_2^s}(q_1)\right) \leq K_1 d(q_1, H^s_{q_2, q_2^s}(q_1).
\]
Since the distance between $\tilde{q}_2$ and $q_2^s$ is uniformly bounded form above, we have
\[
d\left(H^u_{ q_2^s, \tilde{q}_2} \circ H^s_{ q_2, q_2^s}(q_1), H^s_{\tilde{q}_2,q_2^u} \circ H^u_{ q_2^s, \tilde{q}_2} \circ H^s_{ q_2, q_2^s}(q_1)\right) \leq K_2 d(\tilde{q}_2, q_2^u),
\]
where $K_2$ is a uniform constant. 

Notice that $d(q_1, H^s_{q_2, q_2^s}(q_1)) \leq K_3 d(q_2, q_2^s)$ and $d(\tilde{q}_2, q_2^u) \leq K_4 d(q_2,q_2^s)$. The result then follows.

\end{proof}

\subsection{The proof of Theorem \ref{thm.entropys}}
For this section we fix $g\in \mathrm{Sk}^2(\T^2 \times \T^2)$ and  $\mu$ an ergodic $u$-Gibbs measure for $g$ that verifies the hypothesis of Theorem \ref{thm.entropys}.   Suppose that $\mu$ does not verify the conclusion of Theorem \ref{thm.entropys}. By Theorem \ref{thm.tahzibiyanginvariance} and Proposition \ref{prop.uinvariancemeasure} (the Invariance Principle), we have that 
\begin{equation}
\label{eq.smallerentropy}
h_{\mu}(g^{-1}, \mathcal{F}^{ss}) < h_{\nu}(g^{-1}_2) \leq h_{\mu}(g^{-1}),
\end{equation}
where $\nu = (\pi_2)_* \mu$ is the unique SRB of $g_2$. 

Let $\xi^s$ be a $\mu$-measurable, $s$-subordinated partition, and observe that $\xi^s$ is a unstable partition for $g^{-1}$ (we recall that in our notation $\xi^s$ is subordinated to the two-dimensional Pesin stable manifolds of $g$). By Ledrappier-Young's entropy formula results (see Theorem C' in \cite{ledrappieryoung2}), we have that $h_{\mu}(g^{-1}) = h_{\mu}(g^{-1},\xi^s)$.  Hence,
\[
h_{\mu}(g^{-1}, \mathcal{F}^{ss}) < h_{\mu}(g^{-1}, \xi^s)
\]
Take $\xi^{ss}$ a measurable partition subordinated to $\mathcal{F}^{ss}$ that refines the partition $\xi^s$.  For $\mu$-almost every $p$ let $\mu^{ss}_p$ be the conditional measure of $\mu$ along $\xi^{ss}(p)$ and let $\mu^s_p$ be the conditional measure of $\mu$ along $\xi^s(p)$.  By Ledrappier-Young's entropy formula we also get that for $\mu$-almost every $p$, the dimension of the measure $\mu^s_p$ is strictly greater than the dimension of the measure $\mu^{ss}_p$. Since the measure $\mu^s_p$ can be written as $\int_{\xi^s(p)} \mu^{ss}_q d\mu^s_p(q)$ we conclude that the measure $\mu^s_p$ is not supported on $\xi^{ss}(p)$.  Moreover, for any $\delta>0$ we could have chosen these measurable partitions having its elements with diameter smaller than $\delta$. Since we are assuming $\mu$ to have atomic disintegration along the center, we conclude that for any $\delta>0$, for $\mu$-almost every $p$, there is a point $q\in W^s_{\delta}(p)$ such that $q\notin W^{ss}(p)$ and $q$ is an atom of $\mu^c_q$. 
%
%
%
%
%
%

\subsubsection{Fixing several parameters and sets}

We now fix the choices of several parameters to obtain a set of large measure of ``good'' points for which we can apply the strategy. 

\begin{enumerate}[label=(\Alph*)]
\item \label{condition.1} Fix $\beta \in (0,1)$ small such that $\frac{1+\beta}{1-\beta} < \frac{\lambda^+ - \lambda^--2\varepsilon_0}{-\lambda^- + \varepsilon_0}$. 

\item \label{condition.2} Fix $\kappa_1 = \frac{\lambda^+- \lambda^- - 2\varepsilon_0}{\lambda^+ + \varepsilon_0}$, $\kappa_2 = \frac{\lambda^+ - \lambda^- + 2\varepsilon_0}{\lambda_+ - \varepsilon_0}$ and $\alpha_0 = \frac{\kappa_1}{5(\kappa_1 + \kappa_2)}$. 

\item \label{condition.3} Recall that in Section \ref{subsec.repflow} we defined the equivalent measures $\omega$ and $\hat{\omega}$ which are invariant for the suspension flow $\Phi_t$ and the reparametrized suspension flow $\Psi_s$, respectively. We were also using the notation $\zeta = [p,l]$ for points in $M$, where $p\in \T^4$ and $l\in [0,1)$. 

Fix  $N_0>1$ large such that
\[
\omega\left\{\zeta: N_0^{-1} \leq \frac{d\hat{\omega}}{d\omega}(\zeta) \leq N_0.\right\}> 1-\frac{\alpha_0}{2}.
\] 

\end{enumerate}

On $M$ we may consider the measurable partition induced by the center foliation on $\T^4$. For the measure $\omega$, we will write $\omega^c_\zeta$ the conditional measure on the leaf containing $\zeta$. 

\begin{enumerate}[label=(\Alph*)]
\setcounter{enumi}{3}
\item \label{condition.4} By Lusin's Theorem we fix a compact set $K_0 \subset Y \subset M$ of $\omega$ and $\hat{\omega}$ measure arbitrarily close to one ( where $Y$ is the set of full $\omega$-measure defined at the beginning of Section \ref{subsec.repflow}) such that:
\begin{enumerate}[label = (\roman*)]
\item the vector fields $\zeta\mapsto v^*_{\zeta}$;
\item the parametrizations defined in \eqref{eq.parametrizationstableunstable} and extended to $Y$ in Section \ref{subsec.repflow};
\item the Lyapunov norms defined in \eqref{eq.lyapunovnormflow};
\item the map $\zeta\mapsto \omega^c_{\zeta}$; 
\item the function $\hat{L}(.)$ from Lemma \ref{lem.Lestimate}
\end{enumerate}
vary continuously in $K_0$.  In particular, $\hat{L}(.)$ is bounded by a constant $\hat{L}_0$ in $K_0$. 

Recall that by assumption $\omega_\zeta^c$ is atomic for $\omega$-almost every $\zeta$. By the continuity of $\zeta \mapsto \omega^c_\zeta$ in $K_0$, there exists a constant $\varepsilon_1>0$ such that for any $\zeta \in K_0$
 \begin{equation}
\label{eq.distanceatoms}
\min\{d(\eta_1,\eta_2): \textrm{ $\eta_1$ and $\eta_2$ are different atoms of $\omega^c_\zeta$}\}> \varepsilon_1.
\end{equation} 

\item \label{condition.5} Let $C_0>1$ be the maximal ratio between the Lyapunov norms defined in \eqref{eq.lyapunovnormflow} and the Riemannian norm for the points in $K_0$, that is,
\[
C_0 = \displaystyle \sup_{\zeta\in K_0}  \sup_{v\in T_{\zeta} \T^2_{\zeta}-\{0\}} \left\lbrace\left(\frac{\|v\|_{\varepsilon_0,\pm, \zeta}}{\|v\|_{\varepsilon,-, \zeta}}\right)^{\pm 1}, \left(\frac{\|v\|_{\varepsilon_0,\pm, \zeta}}{\|v\|}\right)^{\pm 1}, \left(\frac{\|v\|_{\varepsilon_0,-, \zeta}}{\|v\|}\right)^{\pm 1}\right\rbrace.
\]

\item  \label{condition.6} From Remark \ref{remark.angleslargemeasure}, fix $\gamma_1, \gamma_2>0$ small such that $\mu(\mathcal{A}_{\gamma_1, \gamma_2, 0.9})> 1- \alpha_0$. Let $\Lambda' \subset \T^4$, and the constants $C_1,C_2,C_3>1$ and $r_0,r_1>0$ small given by Lemma \ref{lemma.manycontrolsangle}. We may also suppose that $\mu(\Lambda')>1-2\alpha$. Write $\mathpzc{A} := \mathcal{A}_{\gamma_1, \gamma_2, 0.9} \times [0,1)$ and for $\zeta = [p,l]$ write $	\mathscr{A}_{\gamma_2}(\zeta) = \mathscr{A}_{\gamma_2}(p)$, where $\mathscr{A}_{\gamma_2}(p)$ is defined as in Section \ref{subsec.angles}. 

\item \label{condition.newconstant} Let $L^*$ be a constant such that for any $q\in \xi^{uu}(p)$ and for any $x,y\in \T^2$ we have
\[
\frac{1}{L^* }d(x,y) \leq d(H^u_{p,q}(x), H^u_{p,q}(y)) \leq L^* d(x,y),
\]
and let $L$ be the constant obtained in Lemma \ref{lem.nonintegrabilityestimate}. Take $C^*  = L^* + L$.
\item  \label{condition.7} Take $\hat{T}:= \displaystyle \frac{\log(\hat{C}^2C_0^2 \hat{L}_0^2 C_3^3)}{\lambda^+ - \varepsilon_0}$, where $\hat{C}$ is the constant from Claim \ref{claim.repflowexp}.

%
%

\item \label{condition.9} Fix $K = K_0  \cap [\Lambda' \times [0,1)]$.  By taking the previous sets with sufficiently large measure, we may suppose that $\omega(K) > 1- \frac{\alpha_0}{10}$ and $\hat{\omega}(K) > 1-\frac{\alpha_0}{20 N_0}$.

\end{enumerate}

Observe that if $\rho:Y\to [0,1]$ is an integrable function such that $\int \rho d\omega > 1- ab$, for constants $a,b$ then $\omega(\{p\in Y: \rho(p)>1-a\}) >1-b$. Indeed, write $B= \{p\in Y: \rho(p)> 1-a\}$, then
\[
1-ab <  \int\rho d\omega = \int_{B} \rho d\omega + \int_{Y-B} \rho d\omega  \leq   \omega(B) + (1-a)(1-\omega(B)),
\]
and this implies that $\omega(B)> 1-b$.

Recall that in section \ref{subsectionmartingale} we defined $\tb_s = \Psi_s(\tb)$.  Let $\mathds{1}_K(.)$ be the indicator function of $K$.Consider $\mathbb{E}_{\omega}(\mathds{1}_K|\tb_0)(\zeta)= \int \mathds{1}(\eta) d \omega^{\tb_0}_\zeta(\eta)$. Take $a=0.1$ and $b= \alpha_0$, we have that 
\[
\displaystyle 1-\frac{\alpha}{10}<\omega(K) = \int\left(\int \mathds{1}_K(\eta) d\omega^{\tb_0}_\zeta(\eta)\right)d\omega(\zeta).
\]
By the argument above, we conclude that $\omega(\{\zeta \in M: \mathbb{E}_{\omega}(\mathds{1}_K|\tb_0)(\zeta)>0.9\})>1-\alpha_0$. 

Similarly, using \ref{condition.9} and by taking $a=0.1$ and $b= \frac{\alpha_0}{2N_0}$,  we conclude that $\hat{\omega}(\{\zeta \in M: \mathbb{E}_{\hat{\omega}}(\mathds{1}_K|\tb_{\infty})(\zeta) >0.9\})> 1-\frac{\alpha_0}{2N_0}$.  From \ref{condition.3}, we can conclude that $\omega(\{\zeta \in M: \mathbb{E}_{\hat{\omega}}(\mathds{1}_K|\tb_{\infty})(\zeta) >0.9\})>1-\alpha_0$. 

\begin{enumerate}[label=(\Alph*)]
\setcounter{enumi}{8}
\item \label{condition.10} Take $\mathfrak{B}_0:= \{\zeta \in M: \omega^{\tb_0}_\zeta(K)>0.9\}$. From the argument above, $\omega(\mathfrak{B}_0)>1-\alpha_0$. 
\item \label{condition.11} For each $N>0$ let $\mathrm{B}_N:=\{\zeta \in M : \mathbb{E}_{\hat{\omega}}(\mathds{1}_K|\tb_s)(\zeta) > 0.9, \forall s\geq N\}= \{\zeta \in M: \hat{\omega}^{\tb}_{\Psi_{-s}(\zeta)}(\Psi_{-s}(K))>0.9, \forall s\geq N\}$. By the Martingale convergence argument from section \ref{subsectionmartingale}, we have that for $\hat{\omega}$-almost every $\zeta$,
\[
\lim_{s\to +\infty}\mathbb{E}_{\hat{\omega}}(\mathds{1}_K|\tb_s)(\zeta) = \mathbb{E}_{\hat{\omega}}(\mathds{1}_K|\tb_{\infty})(\zeta).
\]
Hence, fix $N$ sufficiently large so that $\min\{\hat{\omega}(\mathrm{B}_N),\omega(\mathrm{B}_N)\}> 1-\alpha_0$. 
\end{enumerate}

For each $T>0$, let $\mathcal{R}(T)$	be the set of points $\zeta \in K$ such that for $B = K,  \mathpzc{A}, \mathrm{B}_N,  \mathfrak{B}_0$ it holds that 
\[
\frac{1}{T} \mathrm{Leb}(\{t\in [0,T]: \Phi_t(\zeta) \in B\})> 1-\alpha._0 
\]

\begin{enumerate}[label=(\Alph*)]
\setcounter{enumi}{10}
\item \label{condition.12} By the pointwise ergodic theorem, fix $T_0>0$ large enough such that $\omega( \mathcal{R}(T_0))>0$. 

\end{enumerate}

\subsection{Back to the proof}

Recall that we supposed that $\mu$ does not verify the conclusion of Theorem \ref{thm.entropys}. Recall that $\varepsilon_1>0$ is a small constant fixed in \eqref{eq.distanceatoms}.  Since $\mathcal{R}(T_0)>0$, we may fix two points in $K$, $\zeta= [p,l]$ and $ \eta=[q,l]$, such that
\begin{itemize}
\item $\zeta, \eta \in \mathcal{R}(T_0)$;
\item $q\in W^s_{loc}(p)$;
\item $q\notin W^{ss}(p)$. 
\end{itemize}

Let $\delta = \|h^-_p(H^s_{q,p}(q))\|$, where $h^-_p$ is the local chart obtained in Proposition \ref{prop.parametrizationstable} . Furthermore, we may assume that
\[
\label{eq.choiceofdelta}
\delta< \frac{\varepsilon_1}{2C^* \hat{C}^2 C_0^6C_1^3C_2}.
\]
\begin{lemma}
\label{lem.several recurrences}
There exists a sequence $(t_j)_{j\in \N}$ of positive numbers with $t_j \to +\infty$, as $j\to +\infty$, such that
\begin{enumerate}
\item $\Phi_{t_j}(\zeta) \in K \cap \mathfrak{B}_0 \cap \mathpzc{A}$;
\item $\Phi_{t_j}(\eta) \in K \cap \mathfrak{B}_0$;
\item $\Phi_{L_{\zeta, \delta}(t_j)}(\zeta) \in K \cap \mathrm{B}_{N}$, where $\mathrm{B}_N$ is defined in \ref{condition.11};
\item $\Phi_{L_{\zeta, \delta}(t_j)}(\eta) \in K \cap \mathrm{B}_N$. 
\end{enumerate}
\end{lemma}

The proof of Lemma \ref{lem.several recurrences} is the same as Claim $12.2$ in \cite{brownhertz}.

Let $t_j \to + \infty$ be a sequence verifying Lemma \ref{lem.several recurrences}. Let $\zeta_j = \Phi_{t_j}(\zeta)$, $\eta_j = \Phi_{t_j}(\eta)$,  $\tilde{\zeta}_j = \Phi_{L_{\zeta,\delta}(t_j)}(\zeta)$ and $\tilde{\eta}_j = \Phi_{L_{\zeta, \delta}(t_j)}(\eta)$. Let $s'_j = \kappa_{\zeta_j}(\tau_{\zeta,\delta}(t_j))$ and $s_j'' = \kappa_{\eta_j}(\tau_{\zeta,\delta}(t_j))$. Notice that for $t_j$ sufficiently large, since $\tau_{\zeta, \delta}(t_j) \to \infty$, and by the definition of the $\kappa$ function, we have that
\[
\min\{s'_j, s_j''\} \geq (\lambda^+ -\varepsilon_0).\tau_{\zeta,\delta}(t_j) \geq N.
\]
Since $\tilde{\zeta}_j$, $\tilde{\eta}_j \in \mathrm{B}_N$, we have that
\begin{equation}
\label{eq.ss'}
\mathbb{E}_{\hat{\omega}}(\mathds{1}_K|\tb_{s'_j})(\tilde{\zeta}_j)> 0.9 \textrm{ and } \mathbb{E}_{\hat{\omega}}(\mathds{1}_K|\tb_{s_j''})(\tilde{\eta}_j)>0.9.
\end{equation}

For a point $\hat{\zeta} = [\hat{p}, l]$, by construction, there is a natural identification of $\omega^{\tb}_{\hat{\zeta}}$ with $\mu^{uu}_{\hat{p}}$. We also recall that in the skew product setting, since $\mu$ is $u$-Gibbs, $(\pi_2)_* \mu^{uu}_{\hat{p}} = \nu^{uu}_{\pi_2(\hat{p})}$.
Write $\zeta_j = [p_j, l_j]$ and $\eta_j = [q_j, l_j]$. We have the following:
\begin{itemize}
\item Since $\zeta_j$ and $\eta_j$ belong to $\mathfrak{B}_0$, we have that $\min\{\omega^{\tb}_{\zeta_j}(K), \omega^{\tb}_{\eta_j}(K)\}>0.9$.

\item Since $\zeta_j \in \mathpzc{A}$, we obtain that $\omega^{\tb}_{\zeta_j}(\mathcal{A}_{\gamma_2}(\zeta_j) )= \mu^{uu}_{p_j}(\mathcal{A}_{\gamma_2}(p_j))>0.9$.

\item Observe that $\Psi_{-s_j'}(\tilde{\zeta}_j) = \Phi_{\kappa_{\tilde{\zeta}_j}^{-1}(-s_j')}(\tilde{\zeta}_j)$, by the definition of $s'_j$ and $\tilde{\zeta}_j$, we have that $\Psi_{-s_j'}(\tilde{\zeta}_j) = \Phi_{-\tau_{\zeta, \delta}(t_j)} (\tilde{\zeta}_j) = \zeta_j$.  Similarly, $\Psi_{-s_j''}(\tilde{\eta}_j) = \eta_j$. 

\item By \eqref{eq.ss'} and the previous item, we obtain that
\[
\hat{\omega}^{\tb}_{\zeta_j}(\Psi_{-s_j'}(K))> 0.9 \textrm{ and } \hat{\omega}^{\tb}_{\eta_j}(\Psi_{-s_j''}(K))>0.9.
\]
Using the identification in \eqref{eq.omegahatomega}, we conclude that
\[
\omega^{\tb}_{\zeta_j}(\Psi_{-s_j'}(K))> 0.9 \textrm{ and } \omega^{\tb}_{\eta_j}(\Psi_{-s_j''}(K))>0.9.
\]
\end{itemize}

Observe that $\zeta_j$ and $\eta_j$ both have the same time coordinate. Recall that the invariant foliations for $\Phi_t$ are induced by the foliations of $g$.  In particular, for any two points with the same $l$-coordinate we can look at the holonomy map induced by the center stable foliation between the pieces of strong unstable manifolds. We remark that the center stable holonomy induces a $C^1$ map between strong unstable manifolds. In particular,  for $j\in \N$ sufficiently large,  we can choose points $\overline{\zeta}_j\in \tilde{\xi}^{uu}(\zeta_j)$ and $\ne{\eta}_j \in \tilde{\xi}^{uu}(\eta_j)$ such that
\begin{itemize}
\item $\overline{\eta}_j \in W^{cs}(\overline{\zeta}_j)$;
\item $\overline{\eta}_j \in \Psi_{-s_j''}(K) \cap K$;
\item $\overline{\zeta}_j \in \Psi_{-s_j'}(K) \cap K \cap \mathcal{A}_{\gamma_2}(\zeta_j)$.
\end{itemize}

Write $\overline{\zeta}_j = [\overline{p}_j, l_j]$ and $\overline{\eta}_j = [\overline{q}_j, l_j]$.  
%
%
%
%

Let $\overline{w}_j = W^+_{r_1}(\overline{p}_j) \cap H^s_{\overline{q}_j, \overline{p}_j}(W^-_{r_1}(\overline{q}_j))$ and let $\overline{\omega}_j = [ \overline{w}_j, l_j]$.  Let $\overline{z}_j = H^u_{p_j, \overline{p}_j}(W^-_{r_1}(p_j)) \cap H^s_{\overline{q}_j, \overline{p}_j}(W^-_{r_1}(\overline{q}_j))$.  Since $\overline{p}_j \in \mathcal{A}_{\gamma_2}(p_j)$ and $\overline{q}_j \in K$, by Lemma \ref{lemma.manycontrolsangle} we obtain that 
\[
\frac{1}{C_1C_2}d(\overline{p}_j, \overline{z}_j) \leq \|h^+_{\overline{p}_j}(\overline{w}_j)\| \leq C_1 C_2 d(\overline{p}_j,\overline{z}_j),
\]
observe that by taking $j$ sufficiently large and by Lemma \ref{lem.nonintegrabilityestimate}, we may suppose that $p_j, \overline{p}_j$ and $ \overline{q}_j$ verify the hypothesis of Lemma \ref{lemma.manycontrolsangle}.
 Write $q^*_j = H^s_{q_j, p_j}(q_j)$ 

\begin{lemma}
\label{lem.distancecontrol}
For $j$ sufficiently large, we have 
\[
\frac{1}{C^*} d(p_j,q_j^*)  \leq d(\overline{p}_j, \overline{z}_j) \leq C^* d(p_j, q_j^*),
\]
where $C^*$ is the constant defined in \ref{condition.newconstant}.
\end{lemma}
\begin{proof}
Write $\overline{q}_j^* = H^u_{p_j, \overline{p}_j}(q_j^*)$.  By Lemma \ref{lem.nonintegrabilityestimate} we have that
\begin{equation}
\label{eq.qzdistance}
d(H^u_{p_j, \overline{p}_j}(q^*_j), H^s_{\overline{q}_j, \overline{p}_j} \circ H^u_{q_j, \overline{q}_j}\circ H^s_{p_j, q_j}(q^*_j)) \leq L d((p_j)_2, (q_j)_2)^{\theta} \leq L \|Dg|_{E^{ss}}\|^{t_j\theta}.
\end{equation}

Since the angle between $H^u_{p_j, \overline{p}_j}(W^-_{r_1}(p_j))$ and $ H^s_{\overline{q}_j, \overline{p}_j}(W^-_{r_1}(\overline{q}_j))$ is uniformly bounded away from zero,  by \eqref{eq.qzdistance} we conclude that $d(\overline{q}_j^*, \overline{z}_j)) \leq C_4 L \|Dg|_{E^{ss}}\|^{t_j \theta}$, for some uniform constant $C_4>0$. 

\begin{figure}[h]
\centering
\includegraphics[scale= 0.38]{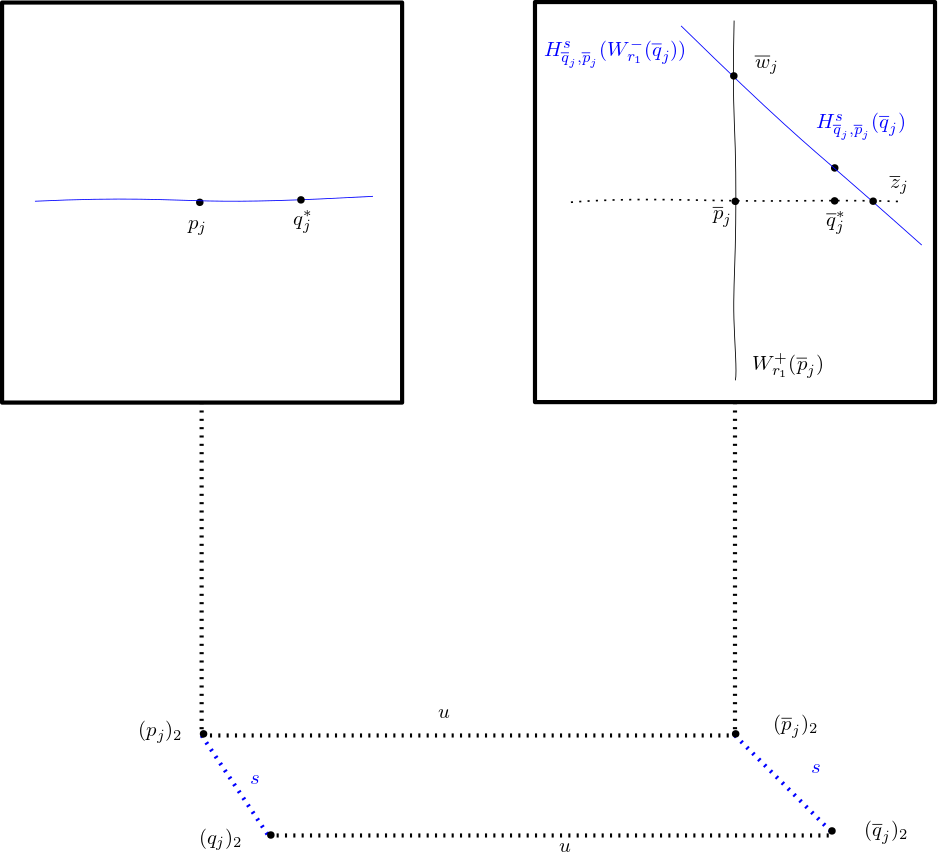}
\caption{Control on distances}
\label{figure2}
\end{figure}

Since $H^u$ is $C^1$ and the diameter of the elements of the partition $\xi^{uu}$ is uniformly bounded, there exists $L^*>1$ such that for any $x,y\in \T^2$,
\[
\frac{1}{L^*} d(x,y) < d(H^u_{p_j,\overline{p}_j}(x), H^u_{p_j, \overline{p}_j}(y))< L^* d(x,y). 
\]

Hence,
\[\def\arraystretch{2.2}
\begin{array}{lll}
d(\overline{p}_j, \overline{q_j}) & \leq & d(\overline{p}_j, \overline{q}_j^*) + d(\overline{q}_j^*, \overline{z}_j) \leq d(H^u_{p_j, \overline{p}_j}(p_j), H^u_{p_j, \overline{p}_j}(q^*_j)) + C_4 L \|Dg|_{E^{ss}}\|^{t_j\theta} \\
& \leq &\displaystyle  d(p_j,q_j^*) \left(L^* + C_4L \frac{\|Dg|_{E^{ss}}\|^{t_j \theta}}{d(p_j, q_j^*)} \right).
\end{array}
\]

However, $d(p_j, q_j^*) \geq m(Dg|_{E^c})^{t_j} d(p, H^s_{q,p}(q))$ and by the condition \eqref{eq.holderconditionthm}, for $j$ sufficiently large we have
\[
C_4\frac{\|Dg|_{E^{ss}}\|^{t_j \theta}}{d(p_j, q_j^*)} \leq \left(\frac{\|Dg|_{E^{ss}}\|^{\theta}}{m(Dg|_{E^c})}\right)^{t_j} \frac{C_4}{d(p, H^s_{q,p}(q))} < 1.
\]
Therefore, for $j$ sufficiently large
\[
d(\overline{p}_j, \overline{q}_j) \leq d(p_j, q_j^*) (L^* +L) = C^* d(p_j, q_j^*).
\]
The proof of the lower bound is similar.  \qedhere

\end{proof}

Combining the estimate from Lemma \ref{lem.distancecontrol} and the estimates for $\|h^+_{\overline{p}_j}(\overline{w}_j)\|$, we obtain
\[
\displaystyle \frac{1}{C^* C_1 C_2} d(p_j, q^*_j) \leq \|h^+_{\overline{p}_j}(\overline{w}_j)\| \leq C^* C_1 C_2 d(p_j, q_j^*).
\]
However,
\[
\begin{array}{rcl}
d(p_j, q_j^*) & \leq & C_1 \|h^-_{p_j}(q_j^*)\|  \leq C_0 C_1 \|h^-_{p_j}(q_j^*)\|_{\varepsilon_0, \pm, p_j}\\
& =& C_0 C_1 \|D\Phi_{t_j}(\zeta)|_{E^+_{\zeta}}\|_{\varepsilon_0,\pm} \|h^-_p(H^s_{q,p}(q))\|_{\varepsilon_0, \pm, p} \leq C_0^2 C_1 \|D\Phi_{t_j}(\zeta)|_{E^+_{\zeta}}\|_{\varepsilon_0,\pm} \delta.
\end{array}
\]
Similarly, we can obtain the lower bound 
\[
d(p_j, q_j^*) \geq \displaystyle \frac{1}{C_0^2 C_1} \|D\Phi_{t_j}(\zeta)|_{E^+_{\zeta}}\|_{\varepsilon_0,\pm} \delta. 
\]

Let $t_j' = \kappa_{\oz_j}^{-1}(s_j')$ and $t_j'' = \kappa_{\oet_j}^{-1}(s_j'')$. Observe that $\Psi_{s_j'}( \oz_j) = \Phi_{t_j'}(\oz_j)$ and $\Psi_{s_j''}(\oet_j) = \Phi_{t_j''}(\oet_j)$.  Write  $\overline{\omega}_j' = \Phi_{t_j'}(\overline{\omega}_j)$, $\oz_j' := \Phi_{t_j'}(\oz_j) = [p_j', l_j']$, $\oet_j':= \Phi_{t_j'}(\oet_j) = [q_j', l_j']$ and $\oet_j'' := \Phi_{t_j''}(\oet_j) =  [q_j'', l_j'']$.

Recall that $\hat{C}$ is the constant given by Claim \ref{claim.repflowexp}.

\begin{claim}
\label{claim.preciseestimate}
$\displaystyle \frac{1}{\hat{C}^2 C_0^6 C_1^2 C_2} \delta \leq \|h^+_{\overline{\zeta}_j'}(\overline{\omega}_j') \| \leq \hat{C}^2 C_0^6 C_1 ^2 C_2 \delta.$
\end{claim}
\begin{proof}

\[
\begin{array}{lll}
\|h^+_{\overline{\zeta}_j'}(\overline{\omega}_j')\| & \leq & C_0 \|h^+_{\overline{\zeta}_j'}( \overline{\omega}_j')\|_{\varepsilon_0, - , \overline{\zeta}_j'}  = C_0 \|D\Phi_{t_j'}(\overline{\zeta}_j)|_{E^+_{\overline{\zeta}_j}}\|_{\varepsilon_0, -} \|h^+_{\overline{\zeta}_j}(\overline{\omega}_j)\|_{\varepsilon_0, - ,\overline{\zeta}_j}\\
&  \leq &C_0^2 \hat{C} e^{s'_j} \|h^+_{\overline{\zeta}_j}(\overline{\omega}_j)\|\leq C^* \hat{C} C_0^4C_1^2 C_2 \|D\Phi_{t_j}(\zeta) |_{E^+_{\zeta}}\|_{\varepsilon_0, \pm} \delta e^{s_j'} \\
& =& C^*\hat{C} C_0^4C_1^2 C_2 \|D\Phi_{t_j}(\zeta) |_{E^-_{\zeta}}\|_{\varepsilon_0, \pm} \delta \|G_{\tau_{\zeta, \delta}(t_j)}(\zeta_j)\|^V_{\varepsilon_0, -}\\
& \leq & C^*\hat{C}^2 C_0^4C_1^2 C_2 \|D\Phi_{t_j}(\zeta) |_{E^-_{\zeta}}\|_{\varepsilon_0, \pm} \delta \|D\Phi_{\tau_{\zeta, \delta}(t_j)}(\zeta_j)|_{E^+_{\zeta_j}}\|_{\varepsilon_0, -}\\
& \leq &C^* \hat{C}^2 C_0^6C_1^2 C_2 \|D\Phi_{t_j}(\zeta) |_{E^-_{\zeta}}\|_{\varepsilon_0, \pm} \delta \|D\Phi_{\tau_{\zeta, \delta}(t_j)}(\zeta_j)|_{E^+_{\zeta_j}}\|_{\varepsilon_0, \pm} = C^* \hat{C}^2 C_0^6 C_1^2 C_2 \delta.
\end{array}
\]
The proof of the lower bound is similar.  
\end{proof}
 
Since $\oz_j' \in K$, we have that $\frac{1}{C_1} \| h^+_{\oz_j'}( \overline{\omega}_j') \| \leq d(\oz_j', \overline{\omega}_j') \leq C_1  \| h^+_{\oz_j'}( \overline{\omega}_j') \|.$ Notice as well that $\overline{\omega}_j$ belongs to the stable manifold of $\oet_j'$, indeed,  $\overline{w}_j \in H^s_{\overline{q}_j, \overline{p}_j}(W-_{r_1}(\overline{q}_j))$. From the definition of $t_j'$, we have that $t_j' \to +\infty$ as $j$ goes to infinity. In particular, $d(\overline{\omega}_j', \oet_j')\to 0$ as $j$ increases. Hence, for $j$ large enough, we have that
\begin{equation}
\label{eq.smallomega}
d(\overline{\omega}_j', \oet_j')<  \frac{1}{2 C^* \hat{C}^2 C_0^6 C_1^3C_2} \delta.
\end{equation}
 
From Claim \ref{claim.preciseestimate}, the estimate \eqref{eq.smallomega} and triangular inequality,  we obtain
\begin{equation}
\label{eq.controldistancept}
\displaystyle \frac{1}{2C^*  \hat{C}^2 C_0^6 C_1^3C_2} \delta \leq d(p_j', H^s_{q_j', p_j'}(q_j')) \leq 2C^* \hat{C}^2 C_0^6 C_1^3C_2 \delta.
\end{equation}
 
 \begin{claim}
 \label{claim.boundedtime}
 $|t_j' - t_j''| < \hat{T}$, where $\hat{T}$ is the constant given in \ref{condition.7}.
 \end{claim}
\begin{proof}
We consider two cases.
\paragraph{Case 1: $t_j' \geq t_j''$.}
As $\oz_j \in K$, from \ref{condition.4} and  \eqref{eq.thatineed},
\[
\|G_{t_j''}(\oz_j)\|^V_{\varepsilon_0, -} \geq \frac{1}{\hat{C}} \|D\Phi_{t_j''}(\oz_j)|_{E^+_{\oz_j}}\|_{\varepsilon_0,-} \geq \frac{1}{\hat{C} \hat{L}_0} \|D\Phi_{t_j''}(\oz_j)|_{E^+_{\zeta}}\|. 
\]
Since $\oet_j'' \in K$, we have
\[
\|G_{t_j''}(\oet_j)\|^V_{\varepsilon_0, -} \leq \hat{C} \|D\Phi_{t_j''}(\oet_j)|_{E^+_{\oet_j}}\|_{\varepsilon_0, -} \leq \hat{C}C_0^2\|D\Phi_{t_j''}(\oet_j)|_{E^+_{\oet_j}}\|.
\]
Let $n'= \lfloor t_j'\rfloor$ be the integer part of $t_j'$, and let $n'' = \lfloor t_j'' \rfloor$ be the integer part of $t_j''$. Observe that we are assuming that $n'' \leq n'$. We also have that $n'' \geq 0$. 

\[\def\arraystretch{2}
\begin{array}{l}
\displaystyle \frac{\|Dg^{n''}(\overline{q}_j)|_{E^+_{\overline{q}_j}}\|}{\|Dg^{n''}(\overline{p}_j)|_{E^+_{\overline{p}_j}}\|} =  \displaystyle \frac{\|Dg^{n''}(\overline{q}_j)|_{E^+_{\overline{q}_j}}\|}{\|Dg^{n''}(\overline{w}_j)|_{T_{\overline{w}_j} W^+(\overline{p}_j)}\|} .   \frac{ \|Dg^{-n'}(\overline{p}_j')|_{E^+_{\overline{p}_j'}}\|}{\|Dg^{-n'}(\overline{w}_j')|_{T_{\overline{w}_j'}W^+(\overline{p}_j')}\|}. 
  \frac{\|Dg^{-(n' - n'')}(\overline{w}_j')|_{T_{\overline{w}_j'}(W^+(\overline{p}_j')}\|}{\|Dg^{-(n'-n'')}(\overline{p}_j')|_{E^+_{\overline{p}_j'}}\|}\\
 \hspace{78pt}=  I. II . III
\end{array}
\]

Recall that $\overline{w}_j \in  H^s_{\overline{q}_j, \overline{p}_j}(W^-_{r_1}(\overline{q}_j))$. In particular, it belongs to the stable manifold of $\overline{q}_j$.  By item $5$ in Lemma \ref{lemma.manycontrolsangle}, $I$ is bounded by $C_3^{-1}$ and  $C_3$.  Observe that $-(n'-n'')$ and $-n'$ are negative numbers. Since $\overline{w}_j' \in W^+_{r_1}(\overline{p}_j)$, from item $6$ of Lemma \ref{lemma.manycontrolsangle}, we obtain that $II$ and $III$ are bounded by $C_3^{-1}$ and  $C_3$.  Hence,
\begin{equation}
\label{eq.eqquotient}
\frac{1}{C_3^3} \leq \displaystyle \frac{\|Dg^{n''}(\overline{q}_j)|_{E^+_{\overline{q}_j}}\|}{\|Dg^{n''}(\overline{p}_j)|_{E^+_{\overline{p}_j}}\|}  \leq C_3^3. 
\end{equation}

Thus,
\[\def\arraystretch{2}
\begin{array}{lll}
\hat{C} C_0^2 \|D\Phi_{t_j''}(\oet_j)|_{E^+_{\oet_j}}\| \geq \|G_{t_j''}(\oet_j)\|^V_{\varepsilon_0, -} &=& \|G_{t_j'}(\oz_j)\|^V_{\varepsilon,-} \geq e^{(\lambda^+- \varepsilon_0)(t_j' - t_j'')}\|G_{t_j''}(\oz_j)\|^V_{\varepsilon_0, -} \\
& \geq & e^{(\lambda^+ - \varepsilon_0)(t_j' - t_j'')} \frac{1}{\hat{C} \hat{L}_0} \|D\Phi_{t_j''}(\oz_j)|E^+_{\zeta}\|
\end{array}
\]

By \eqref{eq.eqquotient}, 
\[
e^{(\lambda^+ - \varepsilon_0)(t_j' - t_j'')} \leq \hat{C}^2 C_0^2 \hat{L}_0 C_3^3.
\]
Therefore,
\[
t_j' - t_j'' \leq \displaystyle \frac{\log(\hat{C}^2 C_0^2 \hat{L}_0 C_3^3)}{\lambda^+ - \varepsilon_0}= \hat{T}.
\]

\paragraph{Case 2: $t_j' \leq t_j''$.}
Observe that, similar to the first case, 
\[
\|G_{t_j'}(\oet_j)\|^V_{\varepsilon_0, -} \geq \frac{1}{\hat{C} \hat{L}_0}\|D\Phi_{t_j'}(\oet_j)|_{E^+_{\oet_j}}\| \textrm{ and } \|G_{t_j'}(\oz_j)\|^V_{\varepsilon_0, -} \leq \hat{C}C_0^2 \|D\Phi_{t_j'}(\oz_j)|_{E^+_{\oz_j}}\|.
\]
We have
\[
\|D\Phi_{t_j'}(\oet_j)|_{E^+_{\oet_j}}\|_{\varepsilon_0, -} \geq \frac{1}{\hat{L}_0} \|D\Phi_{t_j'}(\oet_j)|_{E^+_{\oet_j}}\| \textrm{ and } \|D\Phi_{t_j'}(\oz_j)|_{E^+_{\oz_j}}\|_{\varepsilon_0, -} \leq C_0^2 \|D\Phi_{t_j'}(\oz_j)|_{E^+_{\oz_j}}\|. 
\]

We also have
\[
\displaystyle \frac{\|Dg^{n'}(\overline{q}_j)|_{E^+_{\overline{q}_j}}\|}{\|Dg^{n'}(\overline{p}_j)|_{E^+_{\overline{p}_j}}\|} = \frac{\|Dg^{n'}(\overline{q}_j)|_{E^+_{\overline{q}_j}}\|}{\|Dg^{n'}(\overline{w}_j)|_{T_{\overline{w}_j}W^+_{r_1}(\overline{p}_j)}\|} . \frac{\|Dg^{-n'}(\overline{p}_j')|_{E^+_{\overline{p}_j'}}\|}{\|Dg^{-n'}(\overline{w}_j')|_{T_{\overline{w}_j'}W^+_{r_1}(\overline{p}_j')}\|}.
\]
From items $5$ and $6$ of Lemma \ref{lemma.manycontrolsangle}, we obtain 
\[
\frac{1}{C_3^2} \leq \frac{\|Dg^{n'}(\overline{q}_j)|_{E^+_{\overline{q}_j}}\|}{\|Dg^{n'}(\overline{p}_j)|_{E^+_{\overline{p}_j}}\|} \leq C_3^2.
\]
Hence,
\[\def\arraystretch{2}
\begin{array}{lll}
\hat{C} C_0^2 \|D\Phi_{t_j'}(\oz_j)|_{E^+_{\oz_j}}\| \geq \|G_{t_j'}(\oz_j)\|^V_{\varepsilon,-} &=& \|G_{t_j''}(\oet_j)\|^V_{\varepsilon_0, -} \geq e^{(\lambda^+- \varepsilon_0)(t_j'' - t_j')}\|G_{t_j'}(\oet_j)\|^V_{\varepsilon_0, -} \\
& \geq & e^{(\lambda^+ - \varepsilon_0)(t_j'' - t_j')} \frac{1}{\hat{C} \hat{L}_0} \|D\Phi_{t_j''}(\oz_j)|E^+_{\zeta}\|.
\end{array}
\]
Therefore,
\[
t_j''- t_j' \leq \frac{\log(\hat{C}^2 C_0^2 \hat{L}_0 C_3^2)}{\lambda^+ - \varepsilon_0}\leq \hat{T}. 
\]

\end{proof}

Up to taking a subsequence, we may suppose that the sequence $\oz_j'$ converges to a point $\hat{\zeta}_0$, the sequence $\oet_j''$ converges  to a point $\hat{\eta}_1$, the sequence $t_j' - t_j''$ converges to a number $\hat{t} \in [-\hat{T}, \hat{T}]$ and $\oet_j' = \Phi_{t_j' - t_j''}(\oet_j'')$ converges to a point $\hat{\eta}_0 = \Phi_{\hat{t}}(\hat{\eta}_1)$. 

Since $q_j'$ belongs to the center stable leaf of $p_j'$, from \eqref{eq.controldistancept} we obtain that 
\[
d(\hat{\eta}_0, \hat{\zeta}_0) \geq \frac{1}{2 C^* \hat{C}^2 C_0^6 C_1^3 C_2} \delta >0.
\]
Observe that since $\oz_j' \in K$ and $\oz_j'$ is an atom of $\omega_{\oz_j'}^c$, then $\hat{\zeta}_0$ is an atom of $\omega_{\hat{\zeta}_0}^c$.  Similarly, $\hat{\eta}_1$ is an atom of $\omega^c_{\hat{\eta}_1}$. By the $\Phi_t$-invariance of $\omega$, we obtain that $(\Phi_{\hat{t}})_* \omega^c_{\hat{\eta}_1} = \omega^c_{\hat{\eta}_0} = \omega^c_{\hat{\zeta}_0}$. Therefore, $\omega^c_{\hat{\zeta}_0}$ has an atom in $\hat{\eta}_0$. However, $d(\hat{\zeta}_0, \hat{\eta}_0) \leq 2 C^* \hat{C}^2 C_0^6 C_1^3 C_2 \delta < \varepsilon_1$ and this is a contradiction with \eqref{eq.distanceatoms}.  This concludes the proof of Theorem \ref{thm.entropys}.

\section{Appendix A: $C^2$-regularity of unstable holonomies}
\label{appendix.A}
In this appendix we prove Theorem \ref{thm.holonomyc2}. Let $f$ be a $C^{2+\alpha}$ absolutely partially hyperbolic skew product of $\T^4 = \T^2 \times \T^2$ and let $\chi^{ss}, \chcs,\chcu , \chu$ be the partially hyperbolic constants of $f$. We say that $f$ verifies the \textbf{$(2,\alpha)$-center unstable bunching condition} if
\begin{equation}
\label{eq.newbunching2}
\displaystyle  \left( \frac{\chcu}{\chcs}\right)^2 < \chu \textrm{ and } \frac{\chcu}{(\chcs)^2} < (\chu)^{\alpha}.
\end{equation}
Similarly, $f$ verifies the \textbf{$(2,\alpha)$-center stable bunching condition} if
\begin{equation}
\label{eq.stablenewbunching2}
\displaystyle \chi^{ss}< \left(\frac{\chcs}{\chcu}\right)^2 \textrm{ and } (\chi^{ss})^{\alpha} < \frac{\chcs}{(\chcu)^2}.
\end{equation}
If $f$ verifies condition (\ref{eq.newbunching2}) and (\ref{eq.stablenewbunching2}) then we say that $f$ is \textbf{$(2,\alpha)$-center bunched}. In this section, for any point $p\in \T^4$ and any $n\in \Z$ we write $p_n:=f^n(p)$.

In this appendix, we use the $(2,\alpha)$-center unstable bunching condition to obtain $C^2$-regularity of the unstable holonomy inside a center unstable leaf. Recall that given $p$ and $q$ belonging to the same strong unstable leaf, then there exists a well defined strong unstable holonomy map $H^u_{p,q}: W^c(p) \to W^c(q)$. Since the center manifolds are $\T^2$, we have that each unstable holonomy is a diffeomorphism of $\T^2$. For each $R>0$, we consider the family $\{H^u_{p,q}\}_{p\in \T^4, q\in W_R^{uu}(p)}$.  The main theorem of the appendix is the following:

\begin{theorem}[Theorem \ref{thm.holonomyc2}]
Let $f$ be a $C^{2+\alpha}$ absolutely partially hyperbolic skew product of $\T^4$, and fix $R>0$. If $f$ is $(2,\alpha)$-center unstable bunched, then $\{H^u_{p,q}\}_{p\in \T^4, q\in W_R^{uu}(p)}$ is a family of $C^2$-diffeomorphisms of $\T^2$ whose $C^2$-norm varies continuously with the choices of $p$ and $q$.
\end{theorem}
It is easy to see that this theorem follows from the case that $R=1$. Observe that the $(2,\alpha)$-center unstable bunching condition implies $\frac{\chcu}{\chcs} < \chu.$

This condition is the regular bunching condition which is sufficient to prove that inside a center unstable manifold, the unstable holonomy is a $C^1$-diffeomorphism. For each $n\in \Z$, for each $p\in \T^4$ and $q\in W^{uu}_1(p)$ we have
\begin{equation}
\label{eq.commutationholonomy}
f^n\circ H^u_{p,q} = H^u_{p_n,q_n} \circ f^n \textrm{ and } Df^n(H^u_{p,q}(.))DH^u_{p,q}(.) = DH^u_{p_n,q_n}(f^n(.))Df^n(.),
\end{equation}
where $p_n = f^n(p)$ and $q_n = f^n(q)$.

Since the center leaves are $\T^2$, all its tangent spaces have a canonical identification with $\R^2$. In particular, we may consider $DH^u_{p,q}(.)$ to be a continuous map from $\T^2$ to $L(\R^2, \R^2)$, where $L(\R^2,\R^2)$ is the set of linear maps from $R^2$ to $\R^2$. Thus, the family $\{ DH^u_{p,q}(.)\}_{p\in \T^4, q\in W^{uu}_1(p)}$ is a continuous family that takes values on $C^0(\T^2, L(\R^2, \R^2))$. Furthermore, there exists an uniform constant $C\geq 1$ such that 
\begin{equation}
\label{eq.dhuid}
\|DH^u_{p,q}(.) - Id \| < C d(p,q).
\end{equation}
Fix some constant $K>C$ and let $\mathcal{L}$ be the set defined as follows: an element $\mathcal{L}$ is a continuous family of maps $\{A_{p,q}\}_{p\in \T^4, q\in W^{uu}_1(p)}$ that takes value on $C^0(\T^2, L(\R^2,\R^2))$ such that $\|A_{p,q} - Id\| < K d(p,q)$. For simplicity, we will denote a family $\{A_{p,q}\}_{p\in \T^4, q\in W^{uu}_1(p)}$ by $\mathcal{A}$, such that $\mathcal{A}_{p,q}(.)= A_{p,q}(.)$. We will also write the continuous family given the derivative of the unstable holonomy just by $DH^u$.

 Observe that $\mathcal{L}$ has a natural distance defined by 
\[
d(\mathcal{A},\mathcal{B}) =\displaystyle \sup_{p\in \T^4, q\in W^{uu}_1(p)}\left\{\sup_{x\in \T^2}\|\mathcal{A}_{p,q}(x) - \mathcal{B}_{p,q}(x)\|\right\}.
\] 
For each $n\in \N$ we define $\Gamma_n:\mathcal{L} \to \mathcal{L}$ in the following way: for each $p\in \T^4$ and $q\in W^{uu}_1(p)$, then
\begin{equation}
\label{eq.operatorn}
\Gamma_n(\mathcal{A})_{p,q} (.)= Df^n(H^u_{p_{-n},q_{-n}}( f^{-n}(.)))\mathcal{A}_{p_{-n},q_{-n}}( f^{-n}(.))Df^{-n}(.).
\end{equation} 
By (\ref{eq.commutationholonomy}), for any $n\in \N$ the derivative of the unstable holonomy $DH^u$ is $\Gamma_n$-invariant, that is, $\Gamma_n(DH^u) = DH^u$. In the next lemma we prove that it is the only element of $\mathcal{L}$ that has this property.
\begin{lemma}
\label{lemma.uniqueness}
For any $\mathcal{A}\in \mathcal{L}$, the limit $\lim_{n\to +\infty} \Gamma_n(\mathcal{A})$ exists and it is equal to $DH^u$. Moreover, $DH^u$ is the only element of $\mathcal{L}$ which is $\Gamma_n$-invariant for every $n\in \N$.
\end{lemma}
\begin{proof}
Let $\mathcal{A}\in \mathcal{L}$. Fix $p\in \T^4$ and $q\in W^{uu}_1(p)$, and we write $H^u_{-n}(.) = H^u_{p_{-n},q_{-n}}(.)$. We will use a similar notation for $\mathcal{A}_{p_{-n},q_{-n}}$. For any $x\in W^c(p)$, we have
\[\arraycolsep=1.2pt\def\arraystretch{2}
\begin{array}{rcl}
\|\Gamma_{n}(\mathcal{A})_{p,q}(x) - DH^u_{p,q}(x)\| & = &\displaystyle  \|Df^n(H^u_{-n}(x_{-n}))\left(\mathcal{A}_{-n}(x_{-n}) - DH^u_{-n}(x_{-n})\right)Df^{-n}(x)\|\\
&\leq & \displaystyle \left(\frac{\chcu}{\chcs}\right)^{n}\|\mathcal{A}_{-n}(x_{-n}) - DH^u_{-n}(x_{-n})\|\\
& \leq &\displaystyle \left(\frac{\chcu}{\chcs}\right)^{n}\left(\|\mathcal{A}_{-n}(x_{-n}) - Id\| + \|DH^u_{-n}(x_{-n})-Id\|\right)\\
&\leq &\displaystyle \left(\frac{\chcu}{\chcs}\right)^{n}(K+C)(\chu)^{-n}d(p,q).
\end{array}
\]
The center bunching condition implies that
\[
\displaystyle \frac{\chcu}{\chcs}(\chu)^{-1}<1.
\]
Hence, $\|\Gamma_{n}(\mathcal{A})_{p,q}(x) - DH^u_{p,q}(x)\|$ goes to zero uniformly as $n$ goes to infinity. Since $d(p,q) \leq 1$, this estimate is independent of the points $p$, $q$ and $x$. In other words, $\lim_{n\to +\infty}d(\Gamma_n(\mathcal{A}), DH^u)=0$. Moreover, if $\mathcal{A}$ is $\Gamma_n$-invariant for every $n\in \N$, then $\lim_{n\to + \infty} d(\Gamma_n(\mathcal{A}), DH^u)= d(\mathcal{A},DH^u) = 0$ and thus $\mathcal{A} = DH^u$.  
\end{proof}
For a $C^2$-diffeomorphism $g:\T^2 \to \T^2$, we have that $Dg(.)$ is a map that belongs to $C^1(\T^2,L(\R^2,\R^2))$. In particular, $D^2g(.)$ is a map that belongs to $C^0(\T^2, L(\R^2, L(\R^2,\R^2)))$, where $L(\R^2,L(\R^2,\R^2))$ is the space of linear maps from $\R^2$ to $L(\R^2,\R^2)$. The space $L(\R^2,L(\R^2,\R^2))$ can be identified with the space $L^2(\R^2, \R^2)$, which is the space of bilinear maps of $\R^2$ taking values in $\R^2$. The space $L^2(\R^2,\R^2)$ has a norm given by
\[
\|B\| = \displaystyle \sup\{\|B(u,v)\|: \|u\|=\|v\| =1\}.
\]
Using this norm, we can naturally define a $C^0$-metric in $C^0(\T^2,L^2(\R^2,\R^2))$, which gives a $C^1$-metric in $C^1(\T^2,L(\R^2,\R^2))$ that we will denote it by $d^*_{C^1}(.,.)$. We remark that the space $C^1(\T^2,L(\R^2,\R^2))$ is complete with $d^*_{C^1}(.,.)$.

Consider the set $\mathcal{L}^1$ of the elements $\mathcal{A}$ of $\mathcal{L}$ such that for each $p\in \T^4$ and $q\in W^{uu}_1(p)$ we have $\mathcal{A}_{p,q}(.) \in C^1(\T^2,L(\R^2,\R^2))$ and it varies continuously in the $C^1$-topology with the choices of the points $p$ and $q$. We define the $C^1$-distance on $\mathcal{L}^1$ by
\[
\displaystyle d_{C^1}(\mathcal{A}, \mathcal{B}) = \sup_{p\in \T^4, q\in W^{uu}_1(p)} \left\{ d^*_{C^1}(\mathcal{A}_{p,q}(.), \mathcal{B}_{p,q}(.))\right\}.
\]
It is easy to see that $\mathcal{L}^1$ is closed for the metric $d_{C^1}$. The strategy to prove Theorem \ref{thm.holonomyc2} is the following: we consider the family $Id$ in $\mathcal{L}^1$ which is just the identity for any choices of $p\in \T^4$ and $q\in W^{uu}_1(p)$, next we consider the sequence $\{\Gamma_n(Id)\}_{n\in \N}$ and we prove that this sequence is Cauchy for the metric $d^*_{C^1}$. Then, by Lemma \ref{lemma.uniqueness} we know that $\Gamma_n(Id)$ converges $C^0$ to $DH^u$. However, $\Gamma_n(Id)$ also converges $C^1$ and therefore $DH^u\in \mathcal{L}^1$, which implies that $\{H_{p,q}^u\}_{p\in\T^4, q\in W^{uu}_1(p)}$ is a continuous family of $C^2$-diffeomorphisms.

\begin{remark}
In what follows, we will use the identification of any tangent space of $\T^2$ with $\R^2$.  So that it makes sense, for any vector $v\in \R^2$, to consider the composition $Df(x) Df(y)v$, for any two points $x$ and $y$. Theorem \ref{thm.holonomyc2} also holds for other surfaces, the main point that will change in the proof is two include the parallel transport between different tangent spaces of the surface, so that we can make sense of similar compositions. This would include some extra terms in the computation presented below, which can also be controlled to obtain the same conclusion. For simplicity, and having our original problem in mind (perturbations of Berger-Carrasco's example), we will work only on  $\T^2$. 
\end{remark}

\begin{proof}[Proof of Theorem \ref{thm.holonomyc2}]
As we explained in the previous paragraph, to prove Theorem \ref{thm.holonomyc2} it is enough to prove that the sequence $\{\Gamma_n(Id)\}_{n\in \N}$ is a Cauchy sequence. We fix $p\in \T^4$, $q\in W^{uu}_1(p)$ and $x\in \T^2$. For each $n\in \N$, we define $H^u_{-n}:=H^u_{p_{-n},q_{-n}}(x_-n)$ and $\Gamma_n:=\Gamma_n(Id)_{p,q}(x)$. Observe that 
\[
\Gamma_n = Df^n(H^u_{-n}) Df^{-n}(x)= Df^n(H^u_{-n})Df(x_{-n-1}) Df^{-1}(x_{-n})Df^{-n}(x).
\]
By (\ref{eq.commutationholonomy}), for each $j =1, \cdots, n$, we have $f^j(H^u_{-n}) = H^u_{-n+j}$. Hence,
\[
\Gamma_{n+1}= Df^n(H^u_{-n}) Df(H^u_{-n-1})Df^{-1}(x_{-n}) Df^{-n}(x).
\]
We want to estimate $\|D\Gamma_{n+1}- D\Gamma_n\|$. First, let us evaluate $D\Gamma_{n+1}$ and $D\Gamma_n$. In what follows, for a diffeomorphism $g$, we will write $D^2g(y)[.,.]$ to represent the bilinear form of its second derivative on the point $y$.  By the chain rule and using that $Df(x_{-n-1}) Df^{-1}(x_{-n})= Id$, we obtain
\[\arraycolsep=1.2pt\def\arraystretch{2}
\begin{array}{rclr}
D\Gamma_n[.,.] & = & D\left(Df^n(H^u_{-n})Df(x_{-n-1}) Df^{-1}(x_{-n})Df^{-n}(x)\right)[.,.]&\\
&=& D^2f^n(H^u_{-n})\left[DH^u_{-n}Df^{-n}(x)., Df^{-n}(x). \right]&\left(\mathrm{I}_n\right)\\
&& + Df^n(H^u_{-n}) D^2f(x_{-n-1})\left[Df^{-n-1}(x)., Df^{-n-1}(x).\right]&\left(\mathrm{II}_n\right)\\
&& +  Df^n(H^u_{-n}) Df(x_{-n-1}) D^2f^{-1}(x_{-n})\left[Df^{-n}(x).,Df^{-n}(x).\right]&\left(\mathrm{III}_n\right)\\
&& + Df^n(H^u_{-n}) D^2f^{-n}(x)[.,.]&\left(\mathrm{IV}_n\right)\\
&=& \mathrm{I}_n + \mathrm{II}_n + \mathrm{III}_n + \mathrm{IV}_n. &
\end{array}
\]
Similarly,
\[\arraycolsep=1.2pt\def\arraystretch{2}
\begin{array}{rclr}
D\Gamma_{n+1}[.,.] & = & D\left(Df^n(H^u_{-n})Df(H^u_{-n-1}) Df^{-1}(x_{-n})Df^{-n}(x)\right)[.,.]&\\
&=& D^2f^n(H^u_{-n})\left[DH^u_{-n}Df^{-n}(x)., Df(H^u_{-n-1})Df^{-n-1}(x). \right]&\left(\mathrm{I}'_n\right)\\
&& + Df^n(H^u_{-n}) D^2f(H^u_{-n-1})\left[DH^u_{-n-1}Df^{-n-1}(x)., Df^{-n-1}(x).\right]&\left(\mathrm{II}'_n\right)\\
&& +  Df^{n}(H^u_{-n})Df(H^u_{-n-1}) D^2f^{-1}(x_{-n})\left[Df^{-n}(x).,Df^{-n}(x).\right]&\left(\mathrm{III}'_n\right)\\
&& + Df^{n+1}(H^u_{-n-1}) Df^{-1}(x_{-n})D^2f^{-n}(x)[.,.]&\left(\mathrm{IV}'_n\right)\\
&=& \mathrm{I}'_n + \mathrm{II}'_n + \mathrm{III}'_n + \mathrm{IV}'_n. &
\end{array}
\]

To estimate $\|\Gamma_{n+1} - \Gamma_n\|$ we will separate it into four estimates.
\subsection*{The estimate for $\|\mathrm{I}'_n- \mathrm{I}_n\|$}

Let us first write the expressions for $I_n$ and $I_n'$. In what follows we use that $f^{j}(H^u_{-n}) = H^u_{-n+j}$, for any $j\in \Z$. Then, 

\begin{align*}
I_n & = \displaystyle D^2f(H^u_{-1}) \left[Df^{n-1}(H^u_{-n})DH^u_{-n} Df^{-n}(x)., Df^{n-1}(H^u_{-n})Df^{-n}(x).\right]& (\mathrm{\tilde{I}}_{n,1})\\
&+ \displaystyle Df(H^u_{-1})D^2f(H^u_{-2})\left[ Df^{n-2}(H^u_{-n}) DH^u_{-n} Df^{-n}(x)., Df^{n-2}(H^u_{-n}) Df^{-n}(x). \right]& (\mathrm{\tilde{I}}_{n,2})\\
&\;\;\;\;\;\;\;\;\;\;\;\;\;\;\;\;\;\;\;\;\;\;\;\;\;\;\;\;\;\;\;\;\;\;\;\;\;\;\;\;\;\;\;\;\;\;\;\;\;\;\;\;\;\;\;\;\;\;\; \vdots \\
&+ \displaystyle Df(H^u_{-1}) \cdots Df(H^u_{-n+1}) D^2f(H^u_{-n}) \left[DH^u_{-n} Df^{-n}(x)., Df^{-n}(x).\right].& (\mathrm{\tilde{I}}_{n,n})
\end{align*}
We also have

\begin{align*}
I_n' & = \displaystyle D^2f(H^u_{-1}) \left[Df^{n-1}(H^u_{-n})DH^u_{-n} Df^{-n}(x)., Df^{n}(H^u_{-n-1})Df^{-n-1}(x).\right]& (\mathrm{\tilde{I}}_{n,1}')\\
&+ \displaystyle Df(H^u_{-1})D^2f(H^u_{-2})\left[ Df^{n-2}(H^u_{-n}) DH^u_{-n} Df^{-n}(x)., Df^{n-1}(H^u_{-n-1}) Df^{-n-1}(x). \right] & (\mathrm{\tilde{I}}_{n,2}')\\
&\;\;\;\;\;\;\;\;\;\;\;\;\;\;\;\;\;\;\;\;\;\;\;\;\;\;\;\;\;\;\;\;\;\;\;\;\;\;\;\;\;\;\;\;\;\;\;\;\;\;\;\;\;\;\;\;\;\;\; \vdots \\
&+ \displaystyle Df(H^u_{-1}) \cdots Df(H^u_{-n+1}) D^2f(H^u_{-n}) \left[DH^u_{-n} Df^{-n}(x)., Df^{-n-1}(x).\right]. & (\mathrm{\tilde{I}}_{n,n}')
\end{align*}
Let $\tilde{C}_n =Df(H^u_{-n-1})Df^{-1}(x_{-n})- Id$. For each $j=1, \cdots n$, we obtain
\[\arraycolsep=1.2pt\def\arraystretch{2}
\begin{array}{rcl}
\|\mathrm{\tilde{I}}_{n,j} - \mathrm{\tilde{I}}_{n,j}\| & = & \displaystyle \|Df(H^u_{-1}) \cdots Df(H^u_{-j+1}).\\
&&\displaystyle D^2f(H^u_{-j})\left[Df^{n-j}(H^u_{-n})DH^u_{-n}Df^{-n}(x)., Df^{n-j}(H^u_{-n}) \tilde{C}_n Df^{-n}(x).\right]\|\\
&\leq & \displaystyle \|Df^{j-1}|_{E^c}\|\|f\|_{C^2}\|Df^{n-j}|_{E^c}\|^2 \|DH^u_{-n}\|\|Df^{-n}|_{E^c}\|^2 \|\tilde{C}_n\|\\
 & \leq & \displaystyle \|f\|_{C^2}\frac{(\chcu)^{j-1}.(\chcu)^{2(n-j)}}{(\chcs)^{2n}} \|DH^u_{-n}\| \|\tilde{C}_n\|\\
 & = & \displaystyle \|f\|_{C^2}\|DH^u_{-n}\| \left(\frac{\chcu}{\chcs} \right)^{2n}\|\tilde{C}_n\| (\chcu)^{-j-1}<\|f\|_{C^2}\|DH^u_{-n}\| \left(\frac{\chcu}{\chcs} \right)^{2n}\|\tilde{C}_n\|. 
\end{array}
\]
We remark that in the last inequality we used that $\chcu>1$.  By (\ref{eq.dhuid}), for every $n\in \N$, we have $\|DH^u_{-n}\| < K$, for some constant $K\geq 1$. Also
\[\arraycolsep=1.2pt\def\arraystretch{2}
\begin{array}{rcl}
\|\tilde{C}_n\| & = &\displaystyle \| Df(H^u_{-n-1})Df^{-1}(x_{-n}) - Id\| \\
& =& \displaystyle \|\left(Df(H^u_{-n-1}) - Df(x_{-n-1}) \right) Df^{-1}(x_{-n})\|\\
& \leq & \displaystyle \frac{1}{\chcs} \|f\|_{C^2} d(x_{-n-1}, H^u_{-n-1}) \\
& \leq & \displaystyle \frac{1}{\chcs} \|f\|_{C^2} (\chu)^{-n-1} d(p,q)\leq \frac{1}{\chcs} \|f\|_{C^2} (\chu)^{-n-1}.
\end{array}
\]
Hence,
\[
\displaystyle \|\mathrm{\tilde{I}}_{n,j} - \mathrm{\tilde{I}}_{n,j}\| \leq \frac{\|f\|_{C^2}^2 K }{\chu \chcu \chcs}\left[ \left( \frac{\chcu}{\chcs}\right)^2 (\chu)^{-1}\right]^n.  
\]
Take the constant
\[
C_1 := \displaystyle \frac{\|f\|_{C^2}^2 K }{\chu \chcu \chcs} 
\]
and observe that 
\begin{equation}
\label{eq.estimateone}
\displaystyle \|\mathrm{I}_n'-\mathrm{I}_n\| \leq \sum_{j=1}^n \|\mathrm{\tilde{I}}_{n,j} - \mathrm{\tilde{I}}_{n,j}\| \leq C_1  \left(\frac{(\chcu)^2}{\chu(\chcs)^2}\right)^n.
\end{equation}
This gives the estimate we need for $\|\mathrm{I}_n' - \mathrm{I}_n\|$.

\subsection*{The estimate for $\|\mathrm{II}'_n- \mathrm{II}_n\|$}
This is the only part in the proof of Theorem \ref{thm.holonomyc2} that we use that $f$ is $C^{2+ \alpha}$. Let
\[
\tilde{\mathrm{II}}_n := D^2f(x_{-n-1})\left[Df^{-n-1}(x)., Df^{-n-1}(x).\right] - D^2f(H^u_{-n-1})\left[DH^u_{-n-1}Df^{-n-1}(x)., Df^{-n-1}(x).\right].
\]
Notice that
\[
\|\mathrm{II}'_n- \mathrm{II}_n\| = \|Df^n(H^u_{-n})\tilde{\mathrm{II}}_n\| \leq (\chcu)^{n} \|\tilde{\mathrm{II}}_n\|.
\]

By the triangular inequality,
\[\arraycolsep=1.2pt\def\arraystretch{2}
\begin{array}{rcl}
\|\tilde{\mathrm{II}}_n\| & \leq & \displaystyle \|D^2f(x_{-n-1})\left[Df^{-n-1}(x)., Df^{-n-1}(x).\right]\\
& & \displaystyle -  D^2f(x_{-n-1})\left[DH^u_{-n-1}Df^{-n-1}(x)., Df^{-n-1}(x).\right]\| \\
& & + \displaystyle \| D^2f(x_{-n-1})\left[DH^u_{-n-1}Df^{-n-1}(x)., Df^{-n-1}(x).\right]\\
& & \displaystyle -  D^2f(H^u_{-n-1})\left[DH^u_{-n-1}Df^{-n-1}(x)., Df^{-n-1}(x).\right]\|\\
& = & \|D_n\| + \|E_n\|.
\end{array}
\]

Let us estimate each of these terms.. 
\[\arraycolsep=1.2pt\def\arraystretch{2}
\begin{array}{rcl}
\|D_n\| & = & \displaystyle \left\lVert D^2f(x_{-n-1})\left[\left( Id - DH^u_{-n-1} \right) Df^{-n-1}(x)., Df^{-n-1}(x).\right]\right\lVert\\
& \leq & \displaystyle \|f\|_{C^2} \|Id - DH^u_{-n-1}\| \|Df^{-n-1}(x)|_{E^c}\|^2 \\
& \leq & \displaystyle \|f\|_{C^2} \left(\frac{1}{(\chcs)^2}\right)^{n+1} C (\chu)^{-n-1} d(p,q) \\
& \leq & \displaystyle \|f\|_{C^2}C \left( \frac{1}{\chu (\chcs)^2}\right)^{n+1} \leq \|f\|_{C^2}C \left( \frac{1}{(\chu)^{\alpha} (\chcs)^2}\right)^{n+1}. 
\end{array}
\]
Since $f$ is $C^{2+ \alpha}$, There exists a constant $C_H\geq 1$ such that $\|D^2f(z)[.,.] - D^2f(w)[.,.]\| \leq C_H d(z,w)^{\alpha}$. Recall that $\|DH^u_{-j}\| < K$, for every $j\in \N$ and some constant $K\geq 1$. Therefore,
\[\arraycolsep=1.2pt\def\arraystretch{2}
\begin{array}{rcl}
\|E_n\| & \leq & C_H d(x_{-n-1}, H^u_{-n-1})^{\alpha} \|DH^u_{-n-1}\|\|Df^{-n-1}|_{E^c}\|^2\\
& \leq & \displaystyle C_H K \frac{1}{(\chcs)^{2(n+1)}} (\chu)^{-\alpha(n+1)} d(p,q)\\
& \leq & \displaystyle C_H K  \left(\frac{1}{(\chu)^{\alpha}(\chcs)^2}\right)^{n+1}.
\end{array}
\]
Take the constant
\[
C_2:= \displaystyle \left(\|f\|_{C^2}C + C_HK\right) \frac{1}{(\chu)^{\alpha} (\chcs)^2}.
\]
We obtain
\begin{equation}
\label{eq.estimatetwo}
\displaystyle \|\mathrm{II}'_n- \mathrm{II}_n\| \leq C_2 \left(\frac{\chcu}{(\chu)^{\alpha} ( \chcs)^2}\right)^n.
\end{equation}
The $(2,\alpha)$-center bunching condition implies that the right hand side of (\ref{eq.estimatetwo}) goes exponentially fast to zero. This gives the estimate we need for $ \|\mathrm{II}'_n- \mathrm{II}_n\| $.

\subsection*{The estimate for $\|\mathrm{III}'_n- \mathrm{III}_n\|$}

Observe that
\[\arraycolsep=1.2pt\def\arraystretch{2}
\begin{array}{rcl}
\displaystyle \|\mathrm{III}_n' - \mathrm{III}_n\| & =& \displaystyle \|Df^n(H^u_{-n}) \left(Df(H^u_{-n-1}) - Df(x_{-n-1}) \right) D^2f^{-1}(x_{-n}) \left[Df^{-n}(x).,Df^{-n}(x).\right]\|\\
& \leq & \displaystyle (\chcu)^n \| Df(H^u_{-n-1}) - Df(x_{-n-1})\| (\chcs)^{-2n}.
\end{array}
\]
We have
\[
\displaystyle \| Df(H^u_{-n-1}) - Df(x_{-n-1})\| \leq \|f\|_{C^2} (\chu)^{-n-1}.
\]
By taking
\[
C_3:= \frac{\|f\|_C^2}{\chu},
\]
we conclude that
\begin{equation}
\label{eq.estimatethree}
\displaystyle \|\mathrm{III}_n' - \mathrm{III}_n\| \leq C_3 \displaystyle \left(\frac{\chcu}{\chu (\chcs)^2}\right)^n.
\end{equation}
This concludes the estimate we need for $\|\mathrm{III}_n' - \mathrm{III}_n\|$. 

\subsection*{The estimate for $\|\mathrm{IV}'_n- \mathrm{IV}_n\|$}
Notice that
\[\arraycolsep=1.2pt\def\arraystretch{2}
\begin{array}{rcl}
\|\mathrm{IV}_n' - \mathrm{IV}_n\|& = & \|Df^n(H^u_{-n}) \left(Df(H^u_{-n-1}) Df^{-1}(x_{-n}) - Id \right) D^2f^{-n}(x)[.,.]\|\\
& \leq & (\chcu)^n\|\left(Df(H^u_{-n-1}) - Df(x_{-n-1} \right)Df^{-1}(x_{-n})\| \|D^2f^{-n}(x)\|\\
& \leq & (\chcu)^n \|f\|_{C^2}(\chu)^{-n-1} (\chcs)^{-1} \|D^2f^{-n}(x)\|. 
\end{array}
\]
Let us estimate $\|D^2f^{-n}(x)\|$. First, observe that
\begin{align*}
D^2f^{-n}(x)[.,.] & = D^2f^{-1}(x_{-n+1})\left[Df^{-n+1}(x)., Df^{-n+1}(x).\right]\\
& + Df^{-1}(x_{-n+1}) D^2f^{-1}(x_{-n+2}) \left[Df^{-n+2}(x)., Df^{-n+2}(x).\right]\\
&\;\;\;\;\;\;\;\;\;\;\;\;\;\;\;\;\;\;\;\;\;\;\;\;\;\;\;\;\;\;\;\;\; \vdots \\
& + Df^{-n+1}(x_{-1}) D^2f(x)[.,.].
\end{align*}
Using that $\|D^2f^{-1}(.)\| \leq \|f^{-1}\|_{C^2}$ and by the expression above, we obtain
\[
\|D^2f^{-n}(x)\| \leq \displaystyle \|f^{-1}\|_{C^2} \sum_{j=0}^{n-1}(\chcs)^{-j} (\chcs)^{-2n+2j} = \|f^{-1}\|_{C^2} ( \chcs)^{-2n} \sum_{j=0}^{n-1} (\chcs)^j. 
\]
Since $\chcs <1$, the sum $\sum_{j\in \N} (\chcs)^j$ converges. Define the constant $C_4$ as 
\[
C_4:= \displaystyle \frac{\|f\|_{C^2}\|f^{-1}\|_{C^2} \sum_{j\in \N} (\chcs)^j}{\chu \chcs}.
\]
We conclude that
\begin{equation}
\label{eq.estimatefour}
\displaystyle \|\mathrm{IV}_n' - \mathrm{IV}_n\| \leq C_4 \left(\frac{\chcu}{\chu (\chcs)^2} \right)^n.
\end{equation}

\subsection*{Conclusion of the proof of Theorem \ref{thm.holonomyc2}}
Take
\[
\displaystyle \chi = \max \left\{ \frac{(\chcu)^2}{\chu (\chcs)^2}, \frac{\chcu}{(\chu)^{\alpha}(\chcs)^2}, \frac{\chcu}{\chu (\chcs)^2}\right\},
\]
and observe that by the $(2,\alpha)$-center bunching condition $\chi<1$. Fix the constant $\hat{C} := C_1 + C_2 + C_3 + C_4$. By (\ref{eq.estimateone}),(\ref{eq.estimatetwo}), (\ref{eq.estimatethree}) and (\ref{eq.estimatefour}) we obtain that
\[
\|\Gamma_{n+1}- \Gamma_n\| \leq \hat{C} \chi^n.
\] 
Therefore, $\{\Gamma_n\}_{n\in\N}$ is a Cauchy sequence for the $C^1$-topology. Observe that all these estimates and constants are uniform with the choices of $p\in \T^4$, $q\in W^{uu}_1(p)$ and $x\in W^c(p)$. We conclude that $\{\Gamma_n(Id)\}_{n\in \N}$ is a Cauchy sequence in $\mathcal{L}^1$ for the $C^1$-topology. Since $\Gamma_n(Id)$ converges $C^0$ to $DH^u$, we conclude that $DH^u$ is $C^1$. This implies that $\{H^u_{p,q}(.)\}_{p\in \T^4,q\in W^{uu}_1(p)}$ is a continuous family of $C^2$-diffeomorphisms whose $C^2$-norm varies continuously with the choices of $p$ and $q$ as above.
\end{proof}    
%

\information
\end{document}